\documentclass[reqno]{amsart}

\usepackage{amssymb,mathrsfs,fig4tex184}
\usepackage{hyperref} 


\newcommand\MSym[2][]{\mathcal{M}^{(#2)}_{\mbox{$\mathcal{S}$}^{#1}}}
\newcommand\MhgSym[2][]{\mathcal{M}^{(#2)}_{\mbox{$\mathcal{S}$}^{#1}_{\rm hg}}}
\newcommand\MphgSym[2][]{\mathcal{M}^{(#2)}_{\mbox{$\mathcal{S}$}_{\rm phg}^{#1}}} 
\newcommand\OPNerror[2]{{\displaystyle\mathcal{E}}^{(#1,#2)}_{\psi}}

\begin{document}

\title[Factorization of hyperbolic operators and microlocalization]{Factorization of Second-order strictly hyperbolic operators with non-smooth coefficients and microlocal diagonalization}

\author[MARTINA GLOGOWATZ]{\textrm{Martina Glogowatz\\
	\newline
	Faculty of Mathematics\\
	University of Vienna\\
	Austria\\
	\url{martina.glogowatz@oeh.univie.ac.at}\\
	\newline
	December 23, 2011}}
%	
%\date{\today}

\thanks{Research mainly supported by the Austrian Science Fund (FWF) project Y237}

%\address{Faculty of Mathematics, University of Vienna, Austria}
%\email{martina.glogowatz@univie.ac.at}
\subjclass[2000]{35S05, 46F30.}
\keywords{algebras of generalized functions, parameter dependent pseudodifferential operators, wave front sets.}
%\subjclass[2010]{Primary 35S05; Secondary 46F30.}

\begin{abstract}
We study strictly hyperbolic partial differential operators of second order with non-smooth coefficients. After modelling them as semiclassical Colombeau equations of log-type we provide a factorization procedure on some time-space-frequency domain. As a result the operator is written as a product of two semiclassical first-order constituents of log-type which approximates the modelled operator microlocally at infinite points. We then present a diagonalization method so that microlocally at infinity the governing equation is equal to a coupled system of two semiclassical first-order strictly hyperbolic pseudodifferential equations. Furthermore we compute the coupling effect. We close with some remarks on the results and future directions. 
\end{abstract}

\maketitle

\numberwithin{equation}{section}
\newtheorem{thm}{Theorem}[section]
\newtheorem{defn}[thm]{Definition}
\newtheorem{lem}[thm]{Lemma}
\newtheorem{prop}[thm]{Proposition}
\newtheorem{cor}[thm]{Corollary}

\theoremstyle{definition}
\newtheorem*{rem}{Remark}

\section{Introduction}\label{sec:Intro} 

When studying strictly hyperbolic partial differential equations with generalized coefficients the results commonly depend on the appropriate choice of the asymptotic scale of the regularizing parameter. In this paper we consider certain types of second order partial differential equations with generalized coefficients of the form
\begin{equation}\label{eqn:Basic}
LU=F
\end{equation}
where $U, F$ are Colombeau generalized functions in $\mathcal{G}_{2,2}$\footnote{We refer to Section~\ref{sec:BasicNotions} for the precise asymptotic behavior that one has to study in this case.} on $\mathbb{R}^{n+1}$ and on the level of representatives the operator $L$ acts as 
\begin{equation*}
(u_\varepsilon)_\varepsilon \mapsto (L_{\varepsilon}(x,z,D_t,D_x,D_z)u_\varepsilon)_\varepsilon \qquad \forall (u_\varepsilon)_\varepsilon \in \mathcal{M}_{H^{\infty}}.
\end{equation*}
In detail the operator $L_\varepsilon$ is considered to be of the following form
\begin{equation}\label{eqn:Gov}
L_\varepsilon := \partial_{z}^2 + \sum_{j=1}^{n-1} b_{j,\varepsilon}(x,z) \partial_{x_j}^2 -  c_{\varepsilon}(x,z) \partial_{t}^2.
\end{equation}
Here the coefficient $c_\varepsilon(x,z)$ is of log-type with exponent $r \in \mathbb{N}$, i.e.\ for some constant $C > 0$ we have $\lVert \partial^{\alpha} c_\varepsilon \rVert_{L^{\infty}} = \mathcal{O} (\log^{|\alpha|/r}(C/\varepsilon))$ as $\varepsilon \to 0$, and strictly non-zero in the sense that $\exists \varepsilon_1 \in (0,1]$ such that $\inf_{y \in \mathbb{R}^n} |c_\varepsilon(x,z)| \geq C$ for all $\varepsilon \in (0,\varepsilon_1]$ and some constant $C>0$ independent of $\varepsilon$. Also $b_{j,\varepsilon}$ is of log-type with exponent $r$ and strictly non-zero. Moreover we assume that $c_\varepsilon \to c$ and $b_{j,\varepsilon} \to b_{j}$ in the H\"{o}lder space $\mathcal{C}^{0,\mu}(\mathbb{R}^n)$ with exponent $\mu \in (0,1)$ as $\varepsilon \to 0$.

The aim of this paper is to provide a symbolic calculus in order to explain a diagonalization with respect to the parameter $z$ of the equation in (\ref{eqn:Basic}). To do so we will reduce the diagonalization problem to a factorization theorem which in the case of smooth coefficients can be found in \cite[Appendix II]{Kumano-go:81}, \cite[Chapter 23]{Hoermander:3}. We note that in these references the results are only valid for operators with simple characteristics on the phase space with the zero section excluded. This is different to our approach as we are interested in factorizations with respect to the parameter $z$.

For recent contributions to a related topic we refer the reader to the work by Garetto and Oberguggenberger \cite{GaMoe:2011a}. There the authors established existence and regularity results for solutions for strictly hyperbolic systems with Colombeau coefficients using symmetrisation up to regularizing errors. In detail, they proved existence of generalized solutions in the case of slow scale coefficients and, in addition they showed regularity in the case of logarithmic slow scale coefficients. Again, the square roots of the principal symbol of the operator are assumed to be simple on the phase space without the zero section in this case. 

So when studying an equation of the form (\ref{eqn:Basic}) with coefficients that satisfy strongly positive logarithmic slow scale estimates one can adapt the results in \cite{GaHo:05,Kumano-go:81} to obtain a diagonalization in some microlocal subregion of the phase space and microlocal regularity has to be understood in a $\mathcal{G}_{2,2}^\infty$ sense\footnote{Recall that $\mathcal{G}_{2,2}^\infty$ is the space of regular generalized functions in $\mathcal{G}_{2,2}$ and is characterized by uniform $\varepsilon$-growth in all derivatives.}. This is a generalization of the results in the smooth coefficient case which will be briefly explained in the beginning of Subsection~\ref{subsec:PreviousResults}. Since in the Colombeau framework the evolution behavior of propagating singularities is not yet sufficiently understood it is not clear how one can derive well-posed approximated Cauchy problems from the resulting microlocal first-order equations due to the diagonalization. For the theory for first-order hyperbolic pseudodifferential equations with generalized symbols we refer to \cite{Hoermann:04, Moe:09, GaMoe:2011a}. 

However, to analyze operators as in (\ref{eqn:Gov}) with less regular coefficients as in \cite{GaMoe:2011a} we try a different approach. In order to overcome the necessity of the logarithmic slow scale assumption (e.g. construction of an approximative inverse) we associate to the operator $L_\varepsilon$ in (\ref{eqn:Gov}) a semiclassical modification $L_{\psi,\varepsilon}$ such that $L_{\psi,\varepsilon}=\varepsilon^2 L_\varepsilon$. Here the $\psi=\psi(\varepsilon)$ refers to the phase function in which the semiclassical scaling is kept to be retained and depend on the regularizing parameter $\varepsilon \in (0,1]$. The detailed explanation of the correspondence between the operators $L$ and $L_\psi$ is given in Subsection~\ref{subsec:ScColombeau}. Then instead of working with the equation (\ref{eqn:Basic}) we consider the corresponding semiclassical problem
\begin{equation}\label{eqn:ScBasic}
L_\psi U=F
\end{equation}
where $U, F$ are generalized functions in $\mathcal{G}_{2,2}$ and the operator $L_\psi: \mathcal{G}_{2,2} \to \mathcal{G}_{2,2}$ acts on the level of representatives as 
\begin{equation*}
(u_\varepsilon)_\varepsilon \mapsto (L_{\psi,\varepsilon}(x,z,D_t,D_x,D_z)u_\varepsilon)_\varepsilon \qquad \forall (u_\varepsilon)_\varepsilon \in \mathcal{M}_{H^{\infty}}.
\end{equation*}
Due to the additional $\varepsilon$-dependent semiclassical scaling parameter we then apply the theory of generalized semiclassical pseudodifferential operators which we call the $\psi$-pseudo-differential operators in the following for short. This enables us to state a factorization theorem for $L_\psi$ in terms of two first-order $\psi$-pseudodifferential operators with respect to the parameter $z$ modulo two different types of error operators. As already mentioned above the factorization is only valid when imposing certain restrictions on the underlying time-space-frequency domain which we left $\varepsilon$-independent. In detail the error operators are characterized by a semiclassical negligibility condition outside a compact set. Since global semiclassical negligible operators are easy to handle as they map moderate nets to negligible ones the interesting case is to understand the compactly supported error operators. But, due to the semiclassical scaling, worse $\varepsilon$-asymptotics are introduced when studying boundedness of such compactly supported operators. The basic notions and a general calculus of $\psi$-pseudodifferential operators can be found in Sections~\ref{sec:BasicNotions} and ~\ref{sec:PsiPseudoCalculus}, the factorization procedure is explained in Section~\ref{sec:Fact}.

To eliminate the remainder terms produced in the factorization procedure we then proceed in Section~\ref{sec:WFInfty} by presenting an adapted notion of microlocal regularity which essentially corresponds to the semiclassical wave front set at infinite points within the Colombeau framework. This allows us to overcome the difficulties that arise due to the compactly supported error operators occurring in the factorization. Concepts of the wave front set in Colombeau's theory have been explored in \cite{NPS:98, GaHo:05, GaHo:06, Garetto:06}. For an introduction to semiclassical analysis and microlocalization we refer to \cite{Martinez:02, EvZw, GuSt:10, Alexandrova:08}. Again we remark that unlike to the case of smooth coefficients we lose the property that singularities are propagating on geometrically determined phase space trajectories.

In Section~\ref{sec:Diag} we then combine the factorization with our notion of microlocal regularity in order to describe the desired microlocal diagonalization method. As a result we obtain a coupled system of two first-order $\psi$-pseudodifferential equations that approximate equation (\ref{eqn:ScBasic}) microlocally at infinite points on an adequate subdomain of the phase space.
\subsection{The smooth background case}\label{subsec:PreviousResults}

A particular example of a strictly hyperbolic partial differential equation of second order is the wave equation describing phenomena such as propagation of elastic waves and vibrations. However, in this subsection of the introduction we give a motivation for the present survey and is devoted to one-way wave propagation in inhomogeneous acoustic media in the case of smooth background coefficients. For more information on the relevance in seismic imaging and migration models we refer the reader to \cite{Claerbout:85, BCS:01, SdH:05, SdH:06}.

To start with a real life problem we may ask whether and how one can establish a tap-proof communication link between a source and a receiver in an underwater environment. In the following we analyze how the correlation between the source and the receiver location can help to detect a desired information in downward continuation problems. The mathematical reformulation of such problems corresponds to initial value problems of second-order partial differential equations with a space-like direction as evolution parameter and is in general not well-defined. 

To overcome the ill-posedness the concept of wave-field decomposition enables us to rewrite the full wave equation into a coupled system of approximative one-way wave equations. Using microlocal techniques the system of one-way equations can then be decoupled assuming that the wave field propagates in one direction within certain propagation angles and prohibits propagation in the opposite direction. Therefore the approximation allows to determine the high-frequency components of the solution as they propagate along curved trajectories in the phase space.

In the following we recall the crucial statements given by Stolk for the directional wave field decomposition in the present of smooth background data \cite{Stolk:04, Stolk:05}. For a discussion on the inverse scattering problem with simultaneous consideration of possible reflection data see for example \cite{Stolk:00, SdH:06, dHlRB:03} and the references therein. Also we refer to \cite{LeRousseau:06, LeRH:06} for investigations of the Cauchy problem of first-order pseudodifferential equations and to \cite{ZZB:03} for numerical implementations.

As in \cite{Stolk:04} our basic type of model will be the wave equation for inhomogeneous acoustic media in n-dimensions, $n \geq 2$. Since we are interested in approximative one-way equations we allocate the vertical direction $z$, which we call the depth, the lateral directions are denoted by $x$. The medium itself is described by the wave speed $c = c(x,z)$ and the fluid density $\rho = \rho(x,z)$. Also we let $U = U(t,x,z)$ denote the acoustic wave field and $F=F(t,x,z)$ a source which usually describes the initiation of an acoustic wave. Then the acoustic wave equation is given by
\begin{equation} \label{eqn:Stolk}
PU:=\Big( -\frac{1}{\rho} \frac{1}{c^2} \partial_{t}^2 + \sum_{j=1}^{n-1} \partial_{x_j} \frac{1}{\rho} \partial_{x_j} + \partial_{z} \frac{1}{\rho} \partial_{z} \Big) U = F 
\end{equation}
where $U$ and $F$ are in the distribution space $\mathcal{D}'(\mathbb{R}^{n+1})$. Further we assume smooth background data, i.e.\ the wave speed and the density are functions in $\mathcal{C}^{\infty} (\mathbb{R}^n)$, which shall also satisfy the boundedness conditions $0 < c(x,z), \rho(x,z) < \infty$ for all $(x,z) \in \mathbb{R}^{n-1}\times \mathbb{R}$. 

In the case of slow varying media and the presence of a source function we follow Stolk's approach and restrict the analysis to a microlocal region $I_{\Theta_{2}}$ that is given by
\begin{equation*}
\begin{split}
I_{\Theta_{2}} &:= \{ (t,x,z,\tau,\xi,\zeta) \ | \ (x,z,\tau,\xi) \in I_{\Theta_{2}}',\ |\zeta| \leq C |\tau| \}\\
I_{\Theta_{2}}' &:= \{ (x,z,\tau,\xi) \ | \ \tau \neq 0,\ |\xi| \leq \sin\Theta_2 |c(x,z)^{-1}\tau| \}
\end{split}
\end{equation*}
for some angle $\Theta_2  \in (0,\pi/2)$. To give an impression of the first of these domains we give the following picture (see \cite{LeRousseau:06}): \vspace{10pt}

% 1. Definition of characteristic points
\figinit{0.6cm}
\figpt 1:(0,0) \figpt 2:(-0.4,0) \figpt 3:(4,0)
\figpt 4:(0,-0.4) \figpt 5:(0,4) \figpt 6:(3,0)
\figpt 7:(0,3) \figptrot 8 :A= 7 /1, -62/  
\figptsorthoprojline 9 = 8 /1, 3/
\figptrot 10 :B= 7 /1, -31/ \figptrot 11 :C= 7 /1, -34/
\figvectP 12 [1,11] \figvectP 13 [9,8]
\figptinterlines 14 :[1,12; 8,13]
\figptsorthoprojline 15 = 14 /1, 5/
\figptrot 16 :D= 7 /1, -70/ \figpt 17:(0,4.2)
%
% 2. Creation of the graphical file
\psbeginfig{}
\def\Dim{0.2}
\psset(width=1,fillmode=yes,color=0.9)\psline[1,9,14,15]
\psreset{first} \psline[2,3] \psline[4,5] \psline[1,8] 
\psline[1,10] \psarccircP 1 ; 0.7 [8,7] \psarccircP 1 ; 1.2 [10,7]
\psarrowhead[2,3] \psarrowhead[4,5] 
\psset (width=1.5) \pssetdash{3} \psarccircP 1 ; 3 [6,7]
\psendfig
%
% 3. Writing text on the figure
\figvisu{\figBoxA}{}
{
\figwritegce 1 :\small{$\theta$}(7pt,34pt) \figwritegce 1 :\small{$\Theta_2$}(15pt,19pt)
\figwritee 3:$|\xi|$(3pt) \figwritee 16:\small{$\zeta^2 + |\xi|^2 = \frac{\tau^2}{c(x,z)^2}$}(5pt)
\figwritesw 17:$|\zeta|$(6pt)
}
\centerline{\box\figBoxA}
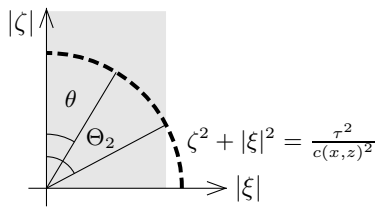
\small{\textbf{Figure 1:} Here the shaded area corresponds to $I_{\Theta_2}$ at given $(t,x,z)$ and a given frequency $\tau$.
The bold dotted line designates the characteristic set $\text{Char}(P)$ and $\theta$ is the propagation angle of the singularities.}\normalsize
\vspace{10pt}

Then under the assumption that the wave front set of $U$ is contained in $I_{\Theta_{2}}$, one can rewrite a microlocal equivalent model to (\ref{eqn:Stolk}) in terms of a coupled system of one-way wave equations with the depth as evolution parameter.

The main result is then the following. The equation
\begin{equation*}
PU = F \qquad \text{microlocally on $I_{\Theta_{2}}$}
\end{equation*}
is equivalent to the system of two first-order strictly hyperbolic partial differential equations of the form
\begin{equation}\label{eqn:System}
P_{0,\pm} u_{\pm} := \big( \partial_z - i B_{\pm}(x,z,D_t,D_x) \big) u_{\pm}  =  f_{\pm}  \quad \text{microlocally on $I_{\Theta_{2}}$}%\\
\end{equation}
and the plus-minus sign refers to downward and upward migration. Here $u_{\pm}, f_{\pm}$ are distributions and $B_{\pm}$ are pseudodifferential operators of order 1 and can be chosen selfadjoint. Furthermore the coupling effect of the counterpropagating constituents $u_{\pm}, f_{\pm}$ and the original data $U,F$ can be computed by a Douglis Nirenberg elliptic pseudodifferential operator transfer matrix.

To model wave propagation in the downward direction in the source-free case one decouples (\ref{eqn:System}) microlocally and we are interested in the solutions to the problem
\begin{equation*}
\begin{split}
&PU = 0 \qquad z > z_0\\
&\text{WF}(U) \cap \{ z= z_0 \ , \ \zeta/\tau > 0 \} = \emptyset
\end{split}
\end{equation*}
for some initial depth $z_0 \in \mathbb{R}$ such that $U |_{z_0}$ is well-defined. 

Using the geometry of the microlocal region we let the set $J_{\Theta,+}(z_0) \subset T^*\mathbb{R}^{n+1} \setminus 0$ consisting of points $(t,x,z,\tau,\xi,\zeta)$ so that the bicharacteristics $(t(z),x(z),\tau(z),\xi(z))$ corresponding to $B_+$, parametrized by $z$, and with propagation angle $\theta(z)$, pass through $(t_0,x_0,\tau_0,\xi_0)$ at initial depth $z=z_0$ and the points $(x(z),z, \tau(z),\xi(z))$ remain in $I_{\Theta}'$ for all $z \in [0,Z]$. So $J_{\Theta,+}(z_0)$ is the set that can be reached from the initial depth $z=z_0$ while staying in $I_{\Theta,+}$ and the propagation angle $\theta(z)$ along the bicharacteristic is always smaller than $\Theta$ (cf. Figure 1).

Since the equation for downward migration in (\ref{eqn:System}) holds only microlocally one obtains approximative solutions when studying a perturbation $P_{+}$ of the operator $P_{0,+}$ including an additional damping term $C = C(x,z,D_t,D_x)$ which vanishes in $I_{\Theta_{1}}$ for some fixed positive angle $\Theta_1 < \Theta_2$ and suppresses singularities outside $I_{\Theta_{2}}$. Also note that $u_{-}$ is vanishing on $I_{\Theta_{2}} \cap \{ z = z_0 \}$ so that the perturbed initial value problem for downward propagation now reads
\begin{equation*}
\begin{aligned}
P_{+}u_{+} & = \left( \partial_z - i B_{+} + C \right) u_{+} = 0 \qquad \mathbb{R}^{n} \times (z_0,Z)\\
u_{+} |_{z_0} & = Q^{-1}_{+} (z_0) U |_{z_0}
\end{aligned}
\end{equation*}
where $Q^{-1}_+$ is the essential component of the transfer matrix mentioned above. As a result the solution to the initial value problem for the perturbed first-order pseudodifferential equation can be related to that of $P_{0,+}u_+ = 0$ using the geometry in the region $J_{\Theta,+}$. In detail one can recover the high frequency part of the original wave field on some subset $J_{\Theta_1,+}$ of the phase space that can be reached from an initial given depth $z_0$ while staying in $I_{\Theta_1}$ for propagation angles $\theta(z) < \Theta_1 < \Theta_2$. 
\section{Basic Notions}\label{sec:BasicNotions}

In this section we specify the basic notions that are needed for our constructions. As the problem is treated within the framework of Colombeau algebras we refer to the literature \cite{Colombeau:85, Moe:92, NPS:98, GKOS:01, GGO:05, Garetto:08} for a systematic treatment in this field. 

One of the main objects in our setting are Colombeau generalized functions based on $L^2$-norm estimates which were first introduced in \cite{BiMoe:92}. The elements in this algebra are given by equivalence classes $u := [(u_{\varepsilon})_{\varepsilon\in (0,1]}]$ of nets of regularizing functions $u_\varepsilon$ in the Sobolev space $H^{\infty}=\cap_{k \in \mathbb{Z}} H^k$ corresponding to certain asymptotic seminorm estimates. More precisely, we denote by $\mathcal{M}_{H^{\infty}}$ the nets of moderate growth whose elements are characterized by the property
\begin{equation*}
\forall \alpha \in \mathbb{N}^n \ \exists N \in \mathbb{N}: \quad \lVert \partial^{\alpha} u_{\varepsilon} \rVert_{L^{2}(\mathbb{R}^n)} = \mathcal{O}({\varepsilon}^{-N}) \qquad  \text{ as } {\varepsilon} \to 0.
\end{equation*}
Negligible nets are denoted by $\mathcal{N}_{H^{\infty}}$ and are nets in $\mathcal{M}_{H^{\infty}}$ whose elements satisfy the following additional condition
\begin{equation*}
\forall q \in \mathbb{N}: \quad \lVert u_{\varepsilon} \rVert_{L^{2}(\mathbb{R}^n)} = \mathcal{O}({\varepsilon}^{q}) \qquad \text{ as } {\varepsilon} \to 0,
\end{equation*}
see \cite[Proposition 3.4]{Garetto:05b}. Then the algebra of generalized functions based on $L^{2}$-norm estimates is defined as the factor space $\mathcal{G}_{H^{\infty}} = \mathcal{M}_{H^{\infty}} / \mathcal{N}_{H^{\infty}}$ for which we continue to write $\mathcal{G}_{2,2} (\mathbb{R}^n)$ as in \cite{BiMoe:92}. For simplicity, we shall also use the notation $(u_\varepsilon)_{\varepsilon}$ instead of $(u_\varepsilon)_{\varepsilon \in (0,1]}$ throughout the paper. 

Using \cite[Theorem 2.7]{BiMoe:92} we first note that the distributions $H^{-\infty} = \cup_{k \in \mathbb{Z}} H^k$ are linearly embedded in $\mathcal{G}_{2,2}(\mathbb{R}^n)$ by convolution with a mollifier $\varphi_\varepsilon(x) = \varepsilon^{-n} \varphi(\varepsilon^{-1}x)$ where $\varphi \in \mathscr{S}(\mathbb{R}^n)$ is a Schwartz function such that
\begin{equation}\label{mollifier}
\int \! \varphi(x) \, dx  =  1, \qquad \int \! x^{\alpha}  \varphi(x) \, dx  =  0 \quad \text{ for all } |\alpha| \geq 1.
\end{equation}
Further, by the same result, $H^{\infty}(\mathbb{R}^n)$ is embedded as a subalgebra of $\mathcal{G}_{2,2} (\mathbb{R}^n)$. 

More generally, we introduce Colombeau algebras based on a locally convex vector space $E$ topologized through a family of seminorms $\{ p_i \}_{i \in I}$ as in \cite[Section 3]{Garetto:05a}, \cite[Section 1]{GHO:09}. Again we call the elements of
\begin{align*}
\mathcal{M}_{E} = \{ (u_{\varepsilon})_{\varepsilon} \in E^{(0,1]} \hspace{2pt} | \hspace{2pt} \forall i \in I  \ \exists N \in \mathbb{N}:  p_i( u_{\varepsilon}) = \mathcal{O} ({\varepsilon}^{-N}) \ \mbox{as} \ {\varepsilon} \to 0\}
\end{align*}
$E$-moderate and the set
\begin{align*}
\mathcal{N}_E = \{ (u_{\varepsilon})_{\varepsilon} \in E^{(0,1]} \hspace{2pt} | \hspace{2pt} \forall i \in I  \ \forall q \in \mathbb{N}: p_i( u_{\varepsilon}) = \mathcal{O} ({\varepsilon}^{q}) \ \mbox{as} \ {\varepsilon} \to 0\}
\end{align*}
is said to be $E$-negligible. Then $\mathcal{N}_E$ is an ideal in $\mathcal{M}_{E}$ and the space of Colombeau algebra based on $E$ is defined by the factor space $\mathcal{G}_E = \mathcal{M}_E / \mathcal{N}_E$ and possesses the structure of a $\widetilde{\mathbb{C}}$-module.

In order to realize the log-type with exponent $r \in \mathbb{N}$ conditions on the coefficients of the model operator (\ref{eqn:Gov}) a rescaling in the mollification is required, see \cite{HdH:06, Moe:92}. In detail, we are going to consider regularized functions where the asymptotic growth is estimated in powers of $\omega_\varepsilon := (\log(C/\varepsilon))^{1/r}$ for some $r \in \mathbb{N}$ and a constant $C >0$ such that $(\omega_\varepsilon)_\varepsilon$ is strongly positive. Here the exponent $r$ is not essential in the further considerations but we give some remarks in this regard in Section~\ref{sec:Diag}.  
In this case the regularization is obtained by convolution with the logarithmically scaled mollifier $\varphi_{\omega_\varepsilon^{-1}}(.) := \omega_\varepsilon^n \varphi(\omega_\varepsilon .)$ with $\varphi \in \mathscr{S}(\mathbb{R}^n)$ as in (\ref{mollifier}). 

\subsection*{Further terminology}
In order to establish a factorization theorem for generalized semiclassical partial differential operator as in (\ref{eqn:ScBasic}) we will have to give a meaning to the square root of such operators. For this reason we introduce generalized semiclassical pseudodifferential operators which are characterized by symbols with respect to a certain phase function. Because the phase function is of simple type we shall initially discuss the main notions of generalized symbols which satisfy asymptotic growth conditions with respect to $\omega_\varepsilon = (\log(C/\varepsilon))^{1/r}$ as already mentioned above. As usual, we use the notation $\langle \xi \rangle := (1 + |\xi|^2)^{1/2}$.

First, we let $m \in \mathbb{R}$ and denote by $S^{\! m} = S^{\! m}(\mathbb{R}^n \times \mathbb{R}^n)$ the set of symbols of order $m$ as introduced by H\"{o}rmander in \cite[Definition 18.1.1.]{Hoermander:3}. Furthermore the space $S^{\! m}_{\! \rm hg} = S^{\! m}_{\! \rm hg}(\mathbb{R}^n \times (\mathbb{R}^n \setminus 0))$ consists of homogeneous functions $a \in \mathcal{C}^{\infty} (\mathbb{R}^n \times (\mathbb{R}^n \setminus 0))$ of order $m$, i.e. $a(x,\lambda \xi) = \lambda^{m}a(x,\xi)$ for all $\lambda > 0$ and $\xi \neq 0$, such that 
\begin{equation*}
\forall \alpha,\beta \in \mathbb{N}^n \ \exists C >0: \qquad |\partial^{\alpha}_\xi \partial^{\beta}_x a(x,\xi)| \leq C|\xi|^{m-|\alpha|} \quad \text{for }(x,\xi) \in \mathbb{R}^n \times (\mathbb{R}^n \setminus 0).
\end{equation*}
Since the symbol class $S^{\! m}$ satisfies global estimates we remark that $\mathcal{M}_{{\displaystyle\mathcal{S}}^m}$ is different to that in \cite[Section 1.4]{GHO:09}. We note that $\mathcal{M}_{{\displaystyle\mathcal{S}}^m}$ is the space $\mathcal{M}_E$ by setting $E = S^{\! m}$. 
In the following we will typically encounter subspaces of $\mathcal{M}_{{\displaystyle\mathcal{S}}^m}$ subjected to two different asymptotic scales. This extends the definition of $\mathcal{M}_{{\displaystyle\mathcal{S}}^m}$ and $\mathcal{N}_{{\displaystyle\mathcal{S}}^m}$ in the following way:
\begin{defn}\label{defn:MSymb}
Let $\nu$ be a non-negative real number and $l,k \in \mathbb{R}$. For $m \in \mathbb{R}$ we let $(a_\varepsilon)_\varepsilon$ be a family of symbols $a_\varepsilon \in S^{\! m}$. We then say that $(a_\varepsilon)_\varepsilon$ is in the generalized symbol class $\MSym[m,k]{\nu,l}$ if and only if
\begin{equation}\label{TwoScaleSymb}
\begin{split}
\exists \eta \in (0,1] \ \forall & \alpha, \beta \in \mathbb{N}^n \ \exists C > 0 \text{ such that } \forall \varepsilon \in (0,\eta]:\\
& q^{(m)}_{\alpha, \beta} (a_{\varepsilon}) \hspace{-1.5pt} := \hspace{-1.5pt} \sup_{(x,\xi) \in  \mathbb{R}^{2n}} | \partial_{\xi}^{\alpha} \partial_x^{\beta} a_\varepsilon (x,\xi) | \langle \xi \rangle^{-m + |\alpha|} \leq C \varepsilon^{k} \omega_{\varepsilon}^{\nu |\beta| +l}.
\end{split}
\end{equation}
We call $m, k$ and $(\nu,l)$ the order, the growth type and the log-type respectively of $\MSym[m,k]{\nu,l}$.
\end{defn}
This definition is similar to that for generalized symbols $\mathcal{S}^{m,\mu}_{\rho,\delta,\omega}$ in \cite[Definition 4.1]{GGO:05}. Also, since $(\omega_\varepsilon)_\varepsilon$ is a strongly positive slow scale net we note that a generalized symbol $(a_\varepsilon)_\varepsilon$ in $\MSym[m,k]{\nu,l}$ can always be estimated in the following way
\begin{equation}\label{RegSymb}
\exists \eta \in (0,1] \ \forall \alpha, \beta \in \mathbb{N}^n \ \exists C > 0: \quad q^{(m)}_{\alpha, \beta} (a_{\varepsilon}) \leq C \varepsilon^{k-1}  \quad \forall \varepsilon \in (0,\eta].
\end{equation}
Later in Lemma~\ref{lem:psiparametrix} we give the construction scheme for approximative inverse operators and the advantage of using (\ref{TwoScaleSymb}) rather than (\ref{RegSymb}) will be clear then. Hereafter we typically work in spaces $\MSym[m,k]{\nu,l}$ with $\nu \in \{ 0,1 \}$. Analogously to Definition~\ref{defn:MSymb} one introduces the space $\MhgSym[m,k]{\nu,l}(\mathbb{R}^n \times (\mathbb{R}^n \setminus 0))$.

Furthermore the notion of negligibility in $\MSym[m,k]{\nu,l}$ is defined as follows:
\begin{defn}
An element of $\MSym[m,k]{\nu,l}$ is said to be negligible, denoted by $\mathcal{N}_{{\displaystyle\mathcal{S}}^m}$, if the following condition is fulfilled
\begin{equation*}
\begin{split}
\exists \eta \in (0,1] \ \forall q \in \mathbb{N} \ \forall \alpha, \beta \in \mathbb{N}^n \ \exists C > 0 : \quad
q^{(m)}_{\alpha, \beta} (a_{\varepsilon}) \leq C \varepsilon^{q} \quad \forall \varepsilon \in (0,\eta].
\end{split}
\end{equation*}
\end{defn}
To give an example, let $P(x,D_x) = \sum_{|\alpha| \leq m} a_{\alpha}(x) D^{\alpha}_x$ be a partial differential operator with bounded and measurable coefficients. Then the logarithmically scaled regularization is given by $p_\varepsilon (x,\xi) := (p(.,\xi) \mathbin{*} \varphi_{\omega_\varepsilon^{-1}})(x)$. Moreover the symbol $(p_\varepsilon)_\varepsilon$ is of class $\MSym[m,0]{1,0}$ and of log-type up to order $(\infty,r)$, cf. \cite[Remark 3.2]{Hoermann:04}.

Also we will make use of the following symbol classes:
\begin{defn}\label{defn:SymbOpenSet}
Let $U \subset \mathbb{R}^n \times \mathbb{R}^n$ be open and conic with respect to the second variable. We say that a generalized symbol $\left(a_\varepsilon\right)_\varepsilon$ is in $\MSym[m,k]{\nu,l}(U)$ if $a_\varepsilon \in \mathcal C^{\infty} (U)$ for fixed $\varepsilon \in (0,1]$ and there is a constant $K>0$ independent of $\varepsilon$ such that: 
\begin{multline*}
\qquad \exists \eta \in (0,1] \ \forall \alpha,\beta \in \mathbb{N}^n \ \exists C>0: \\
|\partial_{\xi}^{\alpha} \partial_x^{\beta} a_\varepsilon(x,\xi)| \leq C \varepsilon^{k} \omega_\varepsilon^{\nu |\beta|+l} \langle \xi\rangle^{m-|\alpha|} \quad \text{for } (x,\xi) \in V_{K}, \ \varepsilon \in (0,\eta],\qquad
\end{multline*}
where $V_{K}:= \{ (x,\xi) \in U \ | \ \ d(\xi,\eta) \geq K,\ \eta \in \partial \text{pr}_2(U) \neq \emptyset \}$. We observe that, if $(x,\xi) \in V_{K}$ then $(x,\lambda \xi)\in V_{K}$ for all $\lambda \geq 1$.
\end{defn}
Note that Definition~\ref{defn:MSymb} is equivalent to Definition~\ref{defn:SymbOpenSet} in the case that $U=\mathbb{R}^n \times \mathbb{R}^n$. Also, Definition~\ref{defn:SymbOpenSet} is different to a straightforward Colombeau generalization of global symbols $S^m(U)$. Also note the differece to the definition of the classical symbols which in the case of local symbols is given in \cite[Definition 2.3, page 141]{ChPi:82}.
\subsection{The governing equation in a semiclassical Colombeau setting}\label{subsec:ScColombeau}

As this is a first investigation we will concentrate on operators as given in (\ref{eqn:Gov}) which are characterized by homogeneity. Also note that our model operators have less regular log-type coefficients than logarithmic slow scale. 

As pointed out in the motivation the main aim is to provide a diagonalization procedure using factorization theorems and microlocal analysis. Because it turns out that the log-type condition is not strong enough for our constructions we will introduce certain pseudodifferential operators defined on the Heisenberg group, so-called semiclassical generalized pseudodifferential operators. 

In the usual theory of semiclassical pseudodifferential operators in the Kohn-Nirenberg or standard quantization one associates to a symbol $a \in S^m$ an operator $A: \ensuremath{{\mathcal S}} \to \ensuremath{{\mathcal S}}$ in the following way
\begin{equation*}
\begin{split}
A_\psi(x,D_x) u(x) &:= (2\pi)^{-n} \int (\mathcal{F}a)(q,p) e^{-ipq \hbar /2} \pi_\hbar(q,p) u (x) \,d q \,d p =\\
&= (2\pi \hbar)^{-n} \int e^{i (x-y) \xi/ \hbar} a(x,\xi) u (y) \,dy \,d \xi
\end{split}
\end{equation*}
where in the first line $\mathcal{F}(a)$ denotes the Fourier transform of $a$ on the projected Heisenberg group and $\pi_\hbar(q,p)$ is the projected Schr\"{o}dinger representation of the Heisenberg group on $\mathbb{R}^{2n}$ with parameter $\hbar$ and the second line corresponds to the more familiar presentation of the first one. In detail we have 
\begin{equation*}
\bigl(\pi_\hbar (q,p) u \bigr)(x) := e^{i(qx + \hbar pq/2)} u(x+ \hbar p)
\end{equation*}
and $\hbar$ denotes the normalized Planck constant. Note that for $\hbar = 1$ this representation leads to the usual Kohn-Nirenberg calculus. Also, using representation theory we observe for $\hbar \neq \hbar'$ that $\pi_\hbar$ and $\pi_{\hbar'}$ are inequivalent representations on $L^2(\mathbb{R}^n)$ and moreover for $\hbar \neq 0$ the representation $\pi_{\hbar}$ is irreducible and unitary on $L^2(\mathbb{R}^n)$.

Once introduced such an additional artificial parameter $\hbar$ gives us now the opportunity to apply the theory of $\hbar$-pseudodifferential operators used in semiclassical analysis which is a formalism in an asymptotic regime $\hbar \ll 1$. As we want to combine this with the theory of generalized pseudodifferential operators we start writing $\hbar := \hbar(\varepsilon)$ with the restriction $\hbar(\varepsilon) \to 0$ in the case that $\varepsilon \to 0$. 

For technical reasons (e.g. construction of an approximative inverse) we also have to prevent other possible values of the function $\hbar(\varepsilon)$. As this is a first approach in this field we have chosen $\hbar(\varepsilon):= \varepsilon$ at this point.

So when given a family $(a_\varepsilon)_\varepsilon \in \mathcal{M}_{S^m}$ we define the corresponding semiclassical pseudodifferential operator of Colombeau type to be the linear operator $A_\psi: \mathcal{G}_{2,2} \to \mathcal{G}_{2,2}$ which acts on the level of representatives as 
\begin{equation}\label{modifiedOP}
(u_\varepsilon)_\varepsilon \mapsto (A_{\psi,\varepsilon}(x,D_x)u_\varepsilon)_\varepsilon \qquad \forall (u_\varepsilon)_\varepsilon \in \mathcal{M}_{H^{\infty}}.
\end{equation}
Here for fixed $\varepsilon \in (0,1]$ the right hand side corresponds to the standard quantization of the symbol $(a_\varepsilon)_\varepsilon$ with parameter $\hbar(\varepsilon) = \varepsilon$, that is
\begin{equation*}
\begin{split}
A_{\psi,\varepsilon}(x,D_x) u_\varepsilon(x) &:= (2\pi)^{-n} \int_{\mathbb{R}^n} (\mathcal{F} a_\varepsilon)(q,p) e^{-ipq/2} \pi_\hbar(q,p)u_\varepsilon (x) \,dq \,dp=\\
&= (2\pi \varepsilon)^{-n} \int e^{i (x-y) \xi/ \varepsilon} a_\varepsilon(x,\xi) u_\varepsilon (y) \,dy \,d \xi.
\end{split}
\end{equation*}
Then with $L_\varepsilon$ as in (\ref{eqn:Gov}) the corresponding $\psi$-pseudodifferential operator $L_\psi$ is essentially given by the oscillatory integral 
\begin{equation*}
L_{\psi,\varepsilon} u_\varepsilon(t,y) =  (2\pi \varepsilon)^{-n-1} \int e^{i t \tau/\varepsilon + iy \eta/\varepsilon} l_\varepsilon(y, \tau,\eta) \hat{u}_\varepsilon (\tau/\varepsilon,\eta/\varepsilon) \, d\tau \, d\eta
\end{equation*}
for fixed $\varepsilon \in (0,1]$ and $(l_\varepsilon)_\varepsilon$ is called the generalized symbol of $L_\psi$ and is given by
\begin{equation*}
l_{\varepsilon}(y,\tau,\xi,\zeta) :=  - \zeta^2 - \langle b_{\varepsilon}(y) \xi,\xi \rangle + c_\varepsilon(y)  \tau^2
\end{equation*}
where we have set $b_\varepsilon(y):= \text{diag} \left( b_{1,\varepsilon}(y),\ldots,b_{n-1,\varepsilon}(y) \right)$, $y:=(x,z)$ and $\eta:=(\xi,\zeta)$. Hence $(l_\varepsilon)_\varepsilon \in \MSym[2,0]{1,0}$ and the operator $L_\varepsilon$ can easily be reconstructed from $L_{\psi,\varepsilon}$, i.e. $L_\varepsilon =  \varepsilon^{-2} L_{\psi,\varepsilon}$, because of the homogeneity of the operators.

Given such a scaled generalized operator as above we carry out all transformations within algebras of generalized functions from now on. More explicitly we will study the action of the linear operator $L_\psi = {\mathop{\rm OP}}_{\psi} (l_\varepsilon)$ from $\mathcal{G}_{2,2}$ into itself in the following sense: on the level of representatives $L_\psi$ acts as in (\ref{modifiedOP}). This explains our governing equation 
\begin{equation*}
L_\psi U = F
\end{equation*}
for which we will now present a microlocal diagonalization method.
\section{\texorpdfstring{$\psi$}{psi}-Pseudodifferential Calculus}\label{sec:PsiPseudoCalculus}
In this section we introduce a general calculus for $\psi$-pseudodifferential operators which are certain semiclassical standard quantizations of generalized symbols as demonstrated in the previous section. For an introduction in semiclassical analysis we refer the reader to \cite{Martinez:02,EvZw,GuSt:10}. Since most of the techniques are similar to the classical theory of pseudodifferential operators we also want to give \cite{Hoermander:3,Taylor:81} as references. Moreover a detailed discussion on pseudodifferential operators with Colombeau generalized symbols can be found in \cite{GGO:05,Garetto:08}.

To continue, given a generalized symbol $(a_\varepsilon(x,\xi))_\varepsilon$ one can assign the semiclassical standard quantization which we denote by ${\mathop{\rm OP}}_\psi(a_\varepsilon):=a_\varepsilon(x,\varepsilon D_x)$. In this sense $\psi= \psi(\varepsilon)$ can be thought of as a scaled phase function on phase space. Therefore we let
\begin{equation*}
\psi_\varepsilon(x,\xi) := \langle x, \xi \rangle/ \varepsilon := x \xi /\varepsilon  \qquad \varepsilon \in (0,1].
\end{equation*}
throughout the paper. 

Moreover we choose the following convention for defining the Fourier transform $\mathcal{F}$ of a function $u \in L^2(\mathbb{R}^n)$:
\begin{align*}
\mathcal{F} u(\xi) &:= \hat{u}(\xi) :=  \int e^{-i x \xi} u(x) \,dx := \lim_{\sigma \to 0_+} \int e^{-i x \xi - \sigma \langle x \rangle} u(x) \,dx.
\end{align*}
Then the Fourier transform is an isomorphism on $L^2$ and the inverse Fourier transform of $u \in L^2(\mathbb{R}^n)$ is given by the following formula
\begin{align*}
\mathcal{F}^{-1} u(x) &:=  (2\pi)^{-n} \int e^{i x \xi} u(\xi) \,d\xi := \lim_{\sigma \to 0_+} (2\pi)^{-n} \int e^{i x \xi - \sigma \langle \xi \rangle} u(\xi) \,d\xi.
\end{align*}
More detailed pieces of information of the Fourier transform on $\mathcal{G}_{2,2}$ can be found in \cite{Anton:99}.
As already mentioned above we will focus on generalized pseudodifferential operators having the following phase-amplitude representation:
\begin{defn}\label{defn:psi-PsiDO}
Let $(a_\varepsilon)_\varepsilon \in \MSym[m,k]{\nu,l}(\mathbb{R}^n \times \mathbb{R}^{n})$. We define the corresponding linear operator $A_\psi : \mathcal{G}_{2,2} \to \mathcal{G}_{2,2}$ such that on the level of representatives we have 
\begin{equation*}
(u_\varepsilon)_\varepsilon \mapsto (A_{\psi,\varepsilon}(x,D_x)u_\varepsilon)_\varepsilon \qquad \forall (u_\varepsilon)_\varepsilon \in \mathcal{M}_{H^{\infty}}
\end{equation*}
and
\begin{equation}\label{psiPsiDO}
\begin{split}
A_{\psi,\varepsilon}(x,D_x) u_\varepsilon(x) 
& := (2\pi \varepsilon)^{-n} \int e^{i (x-y) \xi/ \varepsilon} a_\varepsilon(x,\xi) u_\varepsilon (y) \,dy \,d \xi =\\
& \phantom{:}= \varepsilon^{-n} \mathcal{F}^{-1}_{\xi \to x/\varepsilon} \Bigl( a_\varepsilon(x,\xi) \hat{u}_\varepsilon (\xi/\varepsilon) \Bigr)
\end{split}
\end{equation}
where the last integral is interpreted as an oscillatory integral. Then the map preserves moderateness and negligibility, respectively, so that  $A_\psi$ is well-defined on equivalence classes and continuous. We call $A_\psi$ the $\psi$-pseudodifferential operator with generalized symbol $(a_\varepsilon)_\varepsilon = (a_\varepsilon(x,\xi))_\varepsilon$. Also we will often write $A_\psi \in {\mathop{\rm OP}}_\psi \MSym[m,k]{\nu,l}$ when the generalized symbol of $A_\psi$ belongs to the class $\MSym[m,k]{\nu,l}$.
\end{defn}
\begin{rem} \hspace*{\fill} \newline
\indent (a) We remark that $(A_{\psi,\varepsilon}(x,D_x)u_\varepsilon)_\varepsilon = (A_{\varepsilon}(x, \varepsilon D_x)u_\varepsilon)_\varepsilon$ since formally we can always perform the rescaling. The problem of rescaling will be discussed in the next subsection.

(b) Note that Definition~\ref{defn:psi-PsiDO} agrees with the notion of the semiclassical standard quantization of a generalized symbol $(a_\varepsilon)_\varepsilon \in \MSym[m,k]{\nu,l} (\mathbb{R}^n \times \mathbb{R}^n)$ and can simply written as in (\ref{psiPsiDO}) using scaled Fourier transforms. Another commonly used quantization is the Weyl quantization $a^W(x,\varepsilon D)$ which has the nice property that real symbols correspond to self-adjoint Weyl quantizations. Since we will exclusively deal with the standard quantization of generalized symbols we decided to call them $\psi$-pseudodifferential operators.
\end{rem}
To prepare the factorization theorem of Section~\ref{sec:Fact} we will have to consider products of $\psi$-pseudodifferential operators. We therefore start with some general observations concerning the notion of asymptotic expansion of a generalized symbol in $\MSym[m,k]{\nu,l}$. 
\subsection{Asymptotic expansion of the first kind}
First we present the asymptotic expansion of the first kind which is inspired by the results given in \cite[Section 2.5]{Garetto:08} but also uses a relevant technical aspect from the semiclassical approach. Since this first notion of an asymptotic expansion turns out to have no suitable invariant character under rescaling we will further introduce the asymptotic expansion of second kind which behaves slightly different when performing the rescaling. We should emphasize here that this invariance property of the asymptotic expansion of the second kind will be essential from Section~\ref{sec:WFInfty} onwards, where we discuss microlocal regularity at infinite points. 

The definition of the asymptotic expansion of the first kind is now the following:
\begin{defn}\label{defn:AE1}
For $j \in \mathbb{N}$ let $\{m_j\}_j$ be a strictly decreasing sequence of real numbers with $m_j \searrow -\infty$ as $j \to \infty$, $m_0 = m$ and $\{l_j\}_j$ be a sequence of the form $l_j = \sigma j + l$ for some fixed $\sigma,l \in \mathbb{R}$ and $\sigma \geq 0$. Further let $\{ (a_{j,{\varepsilon}})_{\varepsilon} \}_{j}$ be a sequence with $( a_{j,{\varepsilon}} )_{\varepsilon} \in \MSym[m_j,k]{\nu,l_j}$ so that the following uniform growth type condition is satisfied
\begin{equation}\label{UniformCond}
\exists \eta \in (0,1] \ \forall j \in \mathbb{N} \ \forall \alpha, \beta \in \mathbb{N}^n \ \exists C >0: \quad q^{(m_j)}_{\alpha,\beta} (a_{j,\varepsilon}) \leq C \varepsilon^{k} \omega_\varepsilon^{\nu |\beta|+l_j} \quad \forall \varepsilon \in (0,\eta].
\end{equation}
We say that $\sum_{j=0}^{\infty} (\varepsilon^{j} a_{j,\varepsilon})_\varepsilon$ is the asymptotic expansion of the first kind for $(a_\varepsilon)_\varepsilon \in \MSym[m,k]{\nu,l}$, denoted by $(a_\varepsilon)_\varepsilon \sim \sum_{j} (\varepsilon^j a_{j,\varepsilon})_\varepsilon$, if and only if
\begin{equation*}
(a_\varepsilon - \sum_{j=0}^{N-1} \varepsilon^j a_{j,\varepsilon})_\varepsilon \in \MSym[m_N, N+k-1]{\nu,l} \qquad \forall N \geq 1.
\end{equation*}
Moreover $(a_{0,\varepsilon})_\varepsilon$ is said to be the principal symbol of $(a_\varepsilon)_\varepsilon$ if $a_{0,\varepsilon}$ is not identically vanishing for any fixed $\varepsilon \in (0,\eta]$.
\end{defn}
We are now in a position to state our first result.
\begin{lem}\label{lem:AE1}
Let $\{ l_j\}_j$, $\{ m_j \}_j$ and $\{(a_{j,\varepsilon})_\varepsilon\}_j$ be as in Definition~\ref{defn:AE1}. Then there exists a generalized symbol $(a_\varepsilon)_\varepsilon \in \MSym[m,k]{\nu,l}$ such that $(a_\varepsilon)_\varepsilon \sim \sum_{j} (\varepsilon^j a_{j,\varepsilon})_\varepsilon$ in $\MSym[m,k]{\nu,l}$. Moreover the asymptotic expansion of the first kind determines $(a_\varepsilon)_\varepsilon$ uniquely modulo $\mathcal{N}_{{\displaystyle\mathcal{S}}^{-\infty}}$. 
\end{lem}
Before presenting the proof of Lemma~\ref{lem:AE1} we need a technical auxiliary result.
\begin{lem}\label{lem:ChooseMu}
Let $\{ (a_{j,\varepsilon})_\varepsilon \}_j$ be as in Definition~\ref{defn:AE1} and $\chi$ a smooth function such that $\chi \equiv 0$ on $[0,1]$, $\chi \equiv 1$ on $[2,\infty)$ and $0 \leq \chi \leq 1$. Then there exists a zero sequence $\{\mu_j\}_j$ such that for some fixed $\eta \in (0,1]$ we have for every $j \in \mathbb{N}$ and $\alpha, \beta \in \mathbb{N}^n$ with $|\alpha + \beta|\leq j$: 
\begin{equation*}
|\chi(\mu_j (\varepsilon \omega_\varepsilon^{\sigma})^{-1}) q^{(m_j)}_{\alpha,\beta}(a_{j,\varepsilon})| \leq 2^{-j-1} \varepsilon^{k} \omega_\varepsilon^{\nu |\beta|+l_j} (\varepsilon \omega_\varepsilon^{\sigma})^{-1} \qquad \forall \varepsilon \in(0,\eta],
\end{equation*}
with the same $\sigma$ as in Definition~\ref{defn:AE1}.
\end{lem}
\begin{proof} Using the uniformity condition (\ref{UniformCond}) of the sequence $\{ (a_{j,{\varepsilon}})_{\varepsilon} \}_{j}$ we obtain that for some $\eta \in (0,1]$ and $\forall j \in \mathbb{N}$, $\forall \alpha,\beta \in \mathbb{N}^n$, $\exists C^{(1)}_{j,\alpha,\beta} >0$ and $\forall \varepsilon \in (0,\eta]$ we have
\begin{equation*}
\begin{split}
\varepsilon^{-k} \omega_\varepsilon^{- \nu |\beta|- l_j}|\chi(\mu_j (\varepsilon \omega_\varepsilon^{\sigma})^{-1}) q^{(m_j)}_{\alpha,\beta}(a_{j,\varepsilon}) | & \leq C^{(1)}_{j,\alpha,\beta} \hspace{1pt}  \chi(\mu_j (\varepsilon \omega_\varepsilon^{\sigma})^{-1}) =\\
& = C^{(1)}_{j,\alpha,\beta}  \frac{\mu_j}{\varepsilon \omega_\varepsilon^{\sigma}} \chi(\mu_j (\varepsilon \omega_\varepsilon^{\sigma})^{-1}) \frac{\varepsilon \omega_\varepsilon^{\sigma}}{\mu_j} \leq C^{(1)}_{j,\alpha,\beta}  \frac{\mu_j}{\varepsilon \omega_\varepsilon^{\sigma}}  
\end{split}
\end{equation*}
since $\chi(\mu_j (\varepsilon \omega_\varepsilon^{\sigma})^{-1}) = 0$ for $\mu_j \leq \varepsilon \omega_\varepsilon^{\sigma}$ and $0 \leq \chi \leq 1$. We now choose the sequence $\{\mu_j\}_j$ in the following way
\begin{equation}\label{ChooseMu}
\forall j \in \mathbb{N} \ \forall \alpha, \beta \in \mathbb{N}^n \text{ with } |\alpha + \beta| \leq j \text{ we have:} \quad C^{(1)}_{j,\alpha,\beta} \mu_j \leq 2^{-j-1}
\end{equation}
and the proof is complete.
\end{proof}
\begin{proof}[Proof of Lemma~\ref{lem:AE1}]
In the following proof we basically combine techniques from the theory of non-linear generalized functions, \cite[Theorem 2.2]{Garetto:08}, and Borel's Theorem from the semiclassical approach as in \cite[Theorem 4.11]{EvZw} or alternatively \cite[Proposition 2.3.2]{Martinez:02}. To avoid overload calculations, we may assume without loss of generality that $k = 0$ in Lemma~\ref{lem:AE1} and therefore also in the uniformity condition (\ref{UniformCond}).

We let $\phi \in \mathcal C^{\infty}(\mathbb{R}^n)$, $0 \leq \phi \leq 1$ such that $\phi(\xi) = 0$ for $|\xi| \leq 1$ and $\phi(\xi) = 1$ for $|\xi| \geq 2$. Further let $\chi$ and $\{\mu_j \}_j$ be as in Lemma~\ref{lem:ChooseMu}. We introduce functions
\begin{align*}
\begin{split}
b^{(1)}_{j,\varepsilon}(x,\xi) &:= \chi(\mu_j(\varepsilon \omega_\varepsilon^{\sigma})^{-1})a_{j,\varepsilon}(x,\xi)\\
b^{(2)}_{j,\varepsilon}(x,\xi) &:= \left(1-\chi(\mu_j(\varepsilon \omega_\varepsilon^{\sigma})^{-1})\right)\phi(\lambda_j \xi) a_{j,\varepsilon}(x,\xi)
\end{split}
\end{align*}
where $\{ \lambda_j \}_j$ is a strictly decreasing zero sequence such that $\forall j: \lambda_j \leq 1$ and will be specified later. For fixed $\varepsilon \in (0,1]$ we also define 
\begin{equation*}
a_{\varepsilon}(x,\xi):= \sum_{j \geq 0} \varepsilon^j b_{j,\varepsilon}^{(1)}(x,\xi) +\sum_{j \geq 0} \varepsilon^j b_{j,\varepsilon}^{(2)}(x,\xi).
\end{equation*}
By construction the first sum consists of at most finitely many nonzero terms (depending on $\varepsilon \in (0,1]$ fixed), since $\mu_j \to 0$ as $j \to \infty$. Moreover the second sum is locally finite and we conclude that $(a_\varepsilon)_\varepsilon \in \mathcal{E} (\mathbb{R}^n \times \mathbb{R}^n)^{(0,1]}$.

\textbf{Step 1:} In the first part of the proof we show that $(a_\varepsilon)_\varepsilon \in \MSym[m,0]{\nu,l}$. Since $(a_{j,\varepsilon})_\varepsilon \in \MSym[m_j,0]{\nu,l_j}$ satisfy the uniform growth type condition with $k=0$ we obtain that $\exists \eta \in (0,1]$ such that $\forall j \in \mathbb{N}$, $\forall \alpha, \beta \in \mathbb{N}^n$ we have
\begin{equation}\label{Est1}
|\partial^{\alpha}_\xi \partial^{\beta}_x b^{(1)}_{j,\varepsilon}(x,\xi)| \leq C^{(1)}_{j,\alpha,\beta} \hspace{1pt} \omega_\varepsilon^{\nu |\beta| + l_j} \langle \xi \rangle^{m_j-|\alpha|} \qquad \forall \varepsilon \in(0,\eta]
\end{equation}
and the constant $C^{(1)}_{j,\alpha,\beta} > 0$ is the same as in the proof of Lemma~\ref{lem:ChooseMu}.

Concerning $b^{(2)}_{j,\varepsilon}$ we observe that $\text{supp} (\partial^{\alpha} \phi)(\lambda_j\xi) \subset \{ \xi: \lambda_j^{-1} \leq |\xi| \leq 2 \lambda_j^{-1} \}$ for every $|\alpha| \geq 1$. Accordingly, if $|\alpha| \geq 1$, we can assume $\lambda_j \leq 2/ |\xi| \leq 4 \langle \xi \rangle^{-1}$ on $\text{supp} (\partial^{\alpha} \phi)(\lambda_j\xi)$ and $b^{(2)}_{j,\varepsilon}$ can be estimated as follows: $\exists \eta \in (0,1]$ independent of $\alpha, \beta$ and $j$ such that
\begin{equation}\label{Est2}
\begin{split}
|\partial^{\alpha}_\xi \partial^{\beta}_x b^{(2)}_{j,\varepsilon}(x,\xi)| & \leq \langle \xi \rangle^{m_j - |\alpha|} \sum_{\gamma \leq \alpha}  c(\chi,\phi,\gamma) 4^{|\alpha - \gamma|} q^{(m_j)}_{\gamma,\beta}(a_{j,\varepsilon}) \leq\\
& \leq C^{(2)}_{j,\alpha,\beta} \omega_\varepsilon^{\nu |\beta|+l_j} \langle \xi \rangle^{m_j - |\alpha|}
\end{split}
\end{equation}
for all $\varepsilon \in (0,\eta]$. At this point we choose the sequence $\{ \lambda_j \}_j$ strictly decreasing and so that
\begin{equation}\label{ChooseLambda}
\forall j \in \mathbb{N} \ \forall \alpha, \beta \in \mathbb{N}^n \text{ with } |\alpha+ \beta|\leq j \text{ we have:} \quad C^{(2)}_{j,\alpha,\beta} \lambda_j \leq 2^{-j-1}.
\end{equation}
In order to show that $(a_\varepsilon)_\varepsilon \in \MSym[m,0]{\nu,l}$ we note that
\begin{equation}\label{Decompose1}
\forall \alpha, \beta \in \mathbb{N}^n \ \exists j_0 \in \mathbb{N}: \quad |\alpha + \beta| \leq j_0, \ m_{j_0}+ 1\leq m 
\end{equation}
and decompose $(a_\varepsilon)_\varepsilon$ in an appropriate way, that now is
\begin{equation}\label{Decompose2}
a_\varepsilon(x,\xi)= \sum_{j\leq j_0 -1} \varepsilon^j b_{j,\varepsilon}(x,\xi) + \sum_{j \geq j_0} \varepsilon^j b_{j,\varepsilon}(x,\xi) := f_\varepsilon(x,\xi)+ s_\varepsilon(x,\xi) 
\end{equation} 
where we have set $b_{j,\varepsilon} (x,\xi):=b^{(1)}_{j,\varepsilon}(x,\xi)+ b^{(2)}_{j,\varepsilon}(x,\xi)$. Then $\exists \eta \in (0,1]$ such that $\forall \alpha, \beta \in \mathbb{N}^n$
\begin{align*}
|\partial^{\alpha}_\xi \partial^{\beta}_x f_{\varepsilon}(x,\xi)| & =  |\sum_{j \leq j_0 - 1} \varepsilon^j \partial^{\alpha}_\xi \partial^{\beta}_x  \bigl\{ b^{(1)}_{j,\varepsilon}(x,\xi) + b^{(2)}_{j,\varepsilon}(x,\xi) \bigr\} | \leq \\
&\leq \omega_\varepsilon^{\nu |\beta| +l} \langle \xi \rangle^{m-|\alpha|} \! \! \! \sum_{j \leq j_0 - 1} (\varepsilon \omega_\varepsilon^{\sigma})^j \bigl(C^{(1)}_{j,\alpha,\beta} + C^{(2)}_{j,\alpha,\beta}\bigr) \leq  C_{j_0} \omega_\varepsilon^{\nu |\beta|+l} \langle \xi \rangle^{m-|\alpha|}
\end{align*}
$\forall \varepsilon \in (0,\eta]$ by (\ref{Est1}) and (\ref{Est2}). We next estimate the remainder $s_\varepsilon$. For this purpose we apply (\ref{ChooseMu}) from Lemma~\ref{lem:ChooseMu} and use (\ref{ChooseLambda}) to obtain
\begin{equation*}
\begin{split}
\omega_\varepsilon^{- \nu |\beta|-l} |\partial^{\alpha}_\xi \partial^{\beta}_x s_{\varepsilon}(x,\xi)| & \leq \sum_{j \geq j_0} (\varepsilon \omega_\varepsilon^{\sigma})^{j} \langle \xi \rangle^{m_j - |\alpha|} 2^{-j-1} \left\{ \mu_j^{-1} + \lambda_j^{-1} \right\} \leq \\
& \leq (\varepsilon \omega_\varepsilon^{\sigma})^{j_0 -1} \langle \xi \rangle^{m_{j_0}+1-|\alpha|} \sum_{j \geq j_0}  2^{-j} (\varepsilon \omega_\varepsilon^{\sigma})^{j-j_0} \leq\\
& \leq (\varepsilon \omega_\varepsilon^{\sigma})^{j_0 -1} \langle \xi \rangle^{m_{j_0}+1 - |\alpha|} \leq \langle  \xi \rangle^{m - |\alpha|} \ \qquad  \forall \varepsilon \in (0,\eta]
\end{split}
\end{equation*}
for some $\eta \in (0,1]$ independent of the order of differentiation. Note that in the second inequality we used that $\mu_j^{-1} \leq (\varepsilon \omega_\varepsilon^{\sigma})^{-1}$ on $\text{supp} (b_{j,\varepsilon}^{(1)})$ and $\lambda_j^{-1} \leq \langle \xi \rangle$ on $\text{supp} (b_{j,\varepsilon}^{(2)})$. And so, $(a_\varepsilon)_\varepsilon \in \MSym[m,0]{\nu,l}$ as required. 

\textbf{Step 2:} We complete the proof by showing that for every $N \geq 1$ we have
\begin{equation}\label{Step2}
\qquad (a_\varepsilon - \! \sum_{j \leq N-1} \varepsilon^j a_{j,\varepsilon} )_\varepsilon \in \MSym[m_N,N]{\nu,l_N}(\mathbb{R}^n \times \mathbb{R}^n).
\end{equation}
Therefore we let $N \geq 1$ be a fixed natural number and write
\begin{equation*}
\begin{split}
a_\varepsilon(x,\xi) - \! \sum_{j \leq N-1} \varepsilon^j a_{j,\varepsilon}(x,\xi) &= \sum_{j \leq N-1} \varepsilon^j \bigl( 1-\chi(\mu_j(\varepsilon \omega_\varepsilon^{\sigma})^{-1}) \bigr) \bigl( \phi (\lambda_j \xi) -1 \bigr) a_{j,\varepsilon}(x,\xi) + \\
&+ \sum_{j \geq N} \varepsilon^j b_{j,\varepsilon}(x,\xi) =: g_\varepsilon(x,\xi)+ t_\varepsilon(x,\xi).
\end{split}
\end{equation*}
We first calculate the contribution of $g_\varepsilon$. For this note that $|\xi| \leq 2 \lambda_j^{-1}$ on the support of $(\phi (\lambda_j \xi) -1)$ and we obtain that $\exists \eta \in (0,1]$ such that $\forall j \leq N-1$ we have
\begin{equation*}
\begin{split}
|\partial^{\alpha}_\xi \partial^{\beta}_x \bigl( \phi (\lambda_j \xi) -1 \bigr)& a_{j,\varepsilon}(x,\xi) | \leq \sum_{\gamma \leq \alpha}  c(\phi,\gamma) \lambda_j^{|\gamma|} q^{(m_j)}_{\alpha-\gamma,\beta}(a_{j,\varepsilon}) \langle \xi \rangle^{m_j - |\alpha-\gamma|} \leq\\
& \leq \langle \xi \rangle^{m_N - |\alpha|} \sum_{\gamma \leq \alpha}  c(\phi,\gamma) \lambda_j^{|\gamma|} \langle 2 \lambda_{N-1}^{-1} \rangle^{m_j-m_N+|\gamma|} q^{(m_j)}_{\alpha-\gamma,\beta}(a_{j,\varepsilon}) \leq\\
& \leq C_{N,\alpha,\beta} \omega_\varepsilon^{\nu |\beta| +l_j} \langle \xi \rangle^{m_N - |\alpha|} \ \qquad  \forall \varepsilon \in (0,\eta].
\end{split}
\end{equation*}
Now taking into account that $1-\chi(\mu_j (\varepsilon \omega_\varepsilon^{\sigma})^{-1}) \neq 0$ only if $\mu_j \leq 2 \varepsilon \omega_\varepsilon^{\sigma}$ we compute
\begin{equation*}
\begin{split}
\omega_\varepsilon^{- \nu |\beta|-l} |\partial^{\alpha}_\xi & \partial^{\beta}_x g_\varepsilon(x,\xi) | \leq C_{N,\alpha,\beta} \langle \xi \rangle^{m_N - |\alpha|} \! \! \sum_{j \leq N-1} \hspace{1pt} (\varepsilon \omega_\varepsilon^{\sigma})^{j} \left( 1-\chi(\mu_j (\varepsilon \omega_\varepsilon^{\sigma})^{-1}) \right) =\\
& \! \! = \ C_{N,\alpha,\beta} (\varepsilon \omega_\varepsilon^{\sigma})^{N} \langle \xi \rangle^{m_N - |\alpha|} \! \! \sum_{j \leq N-1} \mu_j^{j-N} \Bigl( \frac{\mu_j}{\varepsilon \omega_\varepsilon^{\sigma}} \Bigr)^{\! N - j} \! \! \left( 1-\chi(\mu_j (\varepsilon \omega_\varepsilon^{\sigma})^{-1}) \right)  \leq\\
& \! \! \leq \ \widetilde{C}_{N,\alpha,\beta} (\varepsilon \omega_\varepsilon^{\sigma})^{N} \langle \xi \rangle^{m_N - |\alpha|} \ \qquad  \forall \varepsilon \in (0,\eta].
\end{split}
\end{equation*}
We proceed with the estimation of $t_\varepsilon$ in which we again use a decomposition as in (\ref{Decompose1}) and (\ref{Decompose2}). Therefore we let $N \in \mathbb{N}$, $N \geq 1$ be fixed. Then
\begin{equation*}
\forall \alpha, \beta \in \mathbb{N}^n \ \exists j_0 \in \mathbb{N}: \quad |\alpha + \beta| \leq j_0, \ m_{j_0}+ 1\leq m_N \text{ and } j_0-1 \geq N
\end{equation*}
and we write
\begin{equation}\label{Decompose:t}
t_{\varepsilon}(x,\xi) = \sum_{j = N}^{j_0-1} \varepsilon^j b_{j,\varepsilon}(x,\xi) + s_\varepsilon(x,\xi)
\end{equation}
with $s_\varepsilon$ as in (\ref{Decompose2}). As already shown in the estimate for $s_\varepsilon$ we also have
\begin{equation*}
|\partial^{\alpha}_\xi \partial^{\beta}_x s_{\varepsilon}(x,\xi)| \leq \omega_\varepsilon^{\nu |\beta|+l} (\varepsilon \omega_\varepsilon^{\sigma})^{j_0 -1} \langle \xi \rangle^{m_{j_0} +1- |\alpha|} .
\end{equation*}
Furthermore we obtain for the first term on the right hand side of (\ref{Decompose:t})
\begin{align*}
|\partial^{\alpha}_\xi \partial^{\beta}_x \sum_{j = N}^{j_0-1} \varepsilon^j b_{j,\varepsilon}(x,\xi)| & \leq \omega_\varepsilon^{\nu |\beta|+l} \langle \xi \rangle^{m_N-|\alpha|} \sum_{j=N}^{j_0 - 1} (\varepsilon \omega_\varepsilon^{\sigma})^{j} \bigl(C^{(1)}_{j,\alpha,\beta} + C^{(2)}_{j,\alpha,\beta}\bigr) \leq\\
&\leq  C_{j_0} (\varepsilon \omega_\varepsilon^{\sigma})^{N} \omega_\varepsilon^{\nu |\beta|+l} \langle \xi \rangle^{m_N-|\alpha|}.
\end{align*}
Putting this together we get (\ref{Step2}) which completes the proof.
\end{proof}
\begin{rem}
Before we proceed we make some observations concerning the construction of the generalized symbol $(a_\varepsilon)_\varepsilon$ in Lemma~\ref{lem:AE1} which was given by the expression
\begin{equation*}
a_{\varepsilon}(x,\xi) = \sum_{j \geq 0} \varepsilon^j \chi \Bigl(\frac{\mu_j}{\varepsilon \omega_\varepsilon^{\sigma}}\Bigr) a_{j,\varepsilon}(x,\xi) + \sum_{j \geq 0} \varepsilon^j \Bigl\{1-\chi\Bigl(\frac{\mu_j}{\varepsilon \omega_\varepsilon^{\sigma}} \Bigr)\Bigr\}\phi(\lambda_j \xi) a_{j,\varepsilon}(x,\xi).
\end{equation*}
Here the first sum is inspired from the semiclassical approach whereas the second sum is also used in the usual Colombeau framework. Then it is may worth to mention the following facts:

(a) In order to construct the square root of the homogeneous operator $L_\psi$ we also need to discuss polyhomogeneous symbols. Then the second sum in the definition of $a_\varepsilon$ also makes sense for homogeneous symbols $a_{j,\varepsilon}$, but the first sum would induce a finite number of singularities (at $\xi = 0$) for fixed $\varepsilon \in (0,1]$. To overcome this problem we will introduce generalized symbols that are smoothed off at the zero section so that we can require a symbol to have an asymptotic expansion in homogeneous terms if we allow some additional error terms in the asymptotic expansion. For the precise description of these so-called polyhomogeneous generalized symbols we refer to Definition~\ref{defn:PhgSymb}.

(b) Another important property concerns the effect of rescaling. This will be used in Section~\ref{sec:WFInfty} where we test microlocal regularity of generalized functions using $L^2$-boundedness results of zeroth-order $\psi$-pseudodifferential operators. To exemplify this let $\{(a_{j,\varepsilon})_\varepsilon\}_j$ be a sequence with $(a_{j,\varepsilon})_\varepsilon \in \MSym[-j,k]{\nu,l_j}$ and $l_j$ as in Definition~\ref{defn:AE1}, $j \geq 0$. By Lemma~\ref{lem:AE1} there exists $(a_\varepsilon)_\varepsilon \in \MSym[0,k]{\nu,l}$ such that $\forall N \geq 1$ we have
\begin{equation}\label{AE:Order0}
(r_\varepsilon)_\varepsilon := (a_\varepsilon - \sum_{j \leq N-1} \! \varepsilon^j a_{j,\varepsilon} )_\varepsilon \in \MSym[-N,N+k-1]{\nu,l}.
\end{equation}
Now using the rescaling $(x,\xi) \mapsto (r_\varepsilon(x,\varepsilon \xi))_\varepsilon$ we obtain for every $N \geq 1$
\begin{equation*}
(x,\xi) \mapsto (r_\varepsilon(x,\varepsilon \xi) )_\varepsilon \in \MSym[-N,k-1]{\nu,l}
\end{equation*}
concluding that the asymptotic expansion of the first kind keeps the order of the error operator invariant under rescaling but not its growth type. According to the $L^2$-boundedness it seems suitable to have an asymptotic expansion which preserves the growth type of the error operator in (\ref{AE:Order0}) under rescaling rather than the order of the same. 

Therefore note that $(r_\varepsilon)_\varepsilon$ in (\ref{AE:Order0}) can also be considered as a symbol of class $\mathcal{N}_{{\displaystyle\mathcal{S}}^0}$. Then the rescaling does not influence the negligibility of the symbol $(r_\varepsilon)_\varepsilon$ since
\begin{equation*}
(x,\xi) \mapsto (r_\varepsilon(x,\varepsilon \xi) )_\varepsilon \in \mathcal{N}_{{\displaystyle\mathcal{S}}^0}.
\end{equation*}
Using this we now introduce the asymptotic expansion of the second kind to achieve an invariant growth type under rescaling at least when studying zeroth-order operators. In detail we will study asymptotic expansions of the form $(\sum_{j\geq0} \varepsilon^j a_{j,\varepsilon})_\varepsilon$ which are similar those in Definition~\ref{defn:AE1} but where the sequence $\{ (a_{j,\varepsilon})_\varepsilon \}_j$ is uniformly bounded in the generalized symbol class of order $m=m_0$.
\end{rem}
\subsection{Asymptotic expansion of the second kind}\label{subsec:AE2} In order to eliminate the rescaling problem arising in (b) above we now present the asymptotic expansion of the second kind.
\begin{defn}\label{defn:AE2}
Let $m, k \in \mathbb{R}$ and $\{l_j\}_j$ be a sequence of real numbers of the form $l_j = \sigma j + l$ for some fixed $\sigma,l \in \mathbb{R}$ and $\sigma \geq 0$, $j \in \mathbb{N}$. Further we let $\{ (a_{j,{\varepsilon}})_{\varepsilon} \}_{j}$ be a sequence with $\bigl( a_{j,{\varepsilon}} \bigr)_{\varepsilon} \in \MSym[m,k]{\nu,l_j}$ satisfying the following uniformity condition
\begin{equation*}
\quad \exists \eta \in (0,1] \ \forall j \in \mathbb{N} \ \forall \alpha, \beta \in \mathbb{N}^n \ \exists C >0: \quad  q^{(m)}_{\alpha,\beta} (a_{j,\varepsilon}) \leq C \varepsilon^k \omega_\varepsilon^{\nu |\beta| +l_j} \quad \forall \varepsilon \in (0,\eta].
\end{equation*}
We say that the $\sum_{j=0}^{\infty} (\varepsilon^j a_{j,\varepsilon})_\varepsilon$ is the asymptotic expansion of the second kind for the symbol $(a_\varepsilon)_\varepsilon \in \MSym[m,k]{\nu,l}$, denoted by $(a_\varepsilon)_\varepsilon \approx \sum_{j} (\varepsilon^j a_{j,\varepsilon})_\varepsilon$, if
\begin{equation*}
(a_\varepsilon - \sum_{j=0}^{N-1} \varepsilon^j a_{j,\varepsilon})_\varepsilon \in \MSym[m,N]{\nu,l_N} \qquad \forall N \geq 1.
\end{equation*}
Again we call $(a_{0,\varepsilon})_\varepsilon$ the principal symbol of $(a_\varepsilon)_\varepsilon$ if $a_{0,\varepsilon}$ is not identically zero for every fixed $\varepsilon >0$ sufficiently small.
\end{defn}
Comparing the Defintions~\ref{defn:AE1} and \ref{defn:AE2} we observe that any generalized symbol $(a_\varepsilon)_\varepsilon$ from Lemma~\ref{lem:AE1} also admits an asymptotic expansion of the second kind. 

Furthermore we note that our notions of an asymptotic expansion are different from the one used in the semiclassical approach. A crucial difference is that in the semiclassical setting one uses the H\"{o}rmander symbol class $S^m_{0,0}$ as the underlying classical space and not $S^m := S^m_{1,0}$ as in our case, cf. \cite{Hoermander:66}. As above, the semiclassical asymptotic expansion provides an invariant character of zeroth-order operators under rescaling. For more details on the semiclassical asymptotic expansion we refer to \cite{DiSj:99,Martinez:02,GuSt:10}.
\subsection{Composition of \texorpdfstring{$\psi$}{psi}-pseudodifferential operators}
In this subsection we establish the composition law of two $\psi$-pseudodifferential operators utilizing the method of stationary phase in order to determine the asymptotic behavior in the occurring oscillatory integrals. Because of the specific construction of the $\psi$-pseudodifferential operators it suffices to study a corollary of the stationary phase formula. 

In detail, this simplified formula says that given a function $a \in \mathcal{C}^{\infty}_{\rm c} (\mathbb{R}^{2n})$ then for every $N \geq 1$ one has
\begin{equation*}
\left( \frac{\lambda}{2 \pi} \right)^n \int e^{-i \lambda xy } a(x,y) \, d x \, d y= \sum_{|\alpha| \leq N-1} \frac{1}{ \alpha! \lambda^{|\alpha|}} \left(  D_x^{\alpha} \partial_y^{\alpha} a \right) (0,0) + S_N (a;\lambda)
\end{equation*}
where the parameter $\lambda$ is considered in the limit $\lambda \to \infty$. Moreover the remainder term $S_N (a;\lambda)$ can be estimated as follows
\begin{equation*}
|S_{N}(a; \lambda)| \leq \frac{C}{N! \lambda^N} \sum_{|\alpha+\beta| \leq 2n+1} \lVert \partial_x^{\alpha} \partial_y^{\beta} ( \partial_x \cdot \partial_y )^N a \rVert_{L^1(\mathbb{R}^{2n})}
\end{equation*}
for $\lambda \geq 1$ and the constant $C > 0$ is independent of $\lambda$. In the application we will work with smooth functions $a$ which are not compactly supported but are well-behaved at infinity and also depend on the parameter $\lambda$. For a more general discussion on the method of stationary phase we refer the reader to \cite[Section 2.6]{Martinez:02} or alternatively to \cite[Section 2]{GiSj:94}.

Using the above formula we can show the following result:
\begin{prop}\label{prop:Product}
Let $A_\psi(x,D_x)$ and $B_\psi(x,D_x)$ be two $\psi$-pseudodifferential operators with generalized symbols $(a_\varepsilon)_\varepsilon \in \MSym[m_1,k_1]{\nu,l_1}$ and $(b_\varepsilon)_\varepsilon \in \MSym[m_2,k_2]{\nu,l_2}$ respectively. Then the product $A_\psi B_\psi$ is well-defined and maps $\mathcal{G}_{2,2}$ into itself. Moreover $A_\psi B_\psi$ is a $\psi$-pseudodifferential operator with generalized symbol $(a_\varepsilon \# b_\varepsilon)_\varepsilon$ in $\MSym[m_1 + m_2,k_1 + k_2]{\nu,l_1+l_2}$ and we have the following representation
\begin{equation}\label{AE:Product}
(a_{\varepsilon} \# b_{\varepsilon})_\varepsilon \sim \sum_{|\alpha| \ge 0} \left( \frac{{\varepsilon}^{|\alpha|}}{\alpha!} D^{\alpha}_{\xi} a_{\varepsilon}(x,\xi) \partial^{\alpha}_{x} b_{\varepsilon}(x,\xi) \right)_\varepsilon.
\end{equation}
\end{prop}
Note that in Proposition~\ref{prop:Product} the generalized symbol $(a_\varepsilon \# b_\varepsilon)_\varepsilon$ is given by its asymptotic expansion of the first kind. Hence it can also be reinterpreted as a symbol having an asymptotic expansion of the second kind.

The following proof is an adaption of \cite[Theorem 2.6.5]{Martinez:02} and \cite[Theorem 9.13]{EvZw}.
\begin{proof}
Without loss of generality we assume that $(a_\varepsilon)_\varepsilon$ and $(b_\varepsilon)_\varepsilon$ are of growth type $0$, i.e. $k_1 = k_2 = 0$, as the proof for more general growth type assumptions requires only slight changes in the argumentation.
 
Let $u \in \mathcal{G}_{2,2}(\mathbb{R}^{n})$ having $(u_\varepsilon)_\varepsilon$ as representative and let $\varepsilon \in (0,1]$ be fixed and arbitrary. Then $A_{\psi} B_{\psi} u$ makes sense as an oscillatory integral and we write
\begin{align}\label{ABproduct}
A_{\psi,\varepsilon} B_{\psi,\varepsilon} u_\varepsilon(x) & = \frac{1}{(2 \pi \varepsilon)^{n}} \lim_{\sigma\to 0_{+}} \int e^{i(x-y) \eta/\varepsilon - \sigma \langle y \rangle - \sigma \langle \eta \rangle} a_\varepsilon(x,\eta) (B_{\psi,\varepsilon} u_\varepsilon)(y) \, dy \, d \eta =\nonumber \\
& = \frac{1}{(2 \pi \varepsilon)^{n}} \lim_{\substack{\sigma\to 0_{+} \\ \tau \to 0_{+}}} \int e^{ix \xi/\varepsilon - \tau \langle \xi \rangle} c_{\sigma,\varepsilon}(x,\xi) \hat{u}_\varepsilon(\xi/\varepsilon) \, d \xi
\end{align}
where we have set
\begin{equation*}
c_{\sigma,\varepsilon} (x,\xi) = (2 \pi \varepsilon)^{-n} \int e^{i(x-y)(\eta - \xi) /\varepsilon - \sigma \langle y \rangle - \sigma \langle \eta \rangle} a_\varepsilon(x,\eta) b_{\varepsilon}(y,\xi) \, dy \, d \eta.
\end{equation*}
In what follows we show that $c_{\sigma,\varepsilon}(x,\xi)$ corresponds to a generalized symbol in $\MSym[m_1+m_2,0]{\nu,l_1+l_2}$ uniformly for all $\sigma >0$. Here the relevant asymptotic behavior of $c_{\sigma,\varepsilon}$ as $\varepsilon \to 0$ is described by using the method of stationary phase. Finally, by passing to the limit $\sigma \to 0_{+}$, we obtain by Lebesgue's dominated convergence theorem that the limit of $(c_{\sigma,\varepsilon})_\varepsilon$ exists in $\MSym[m_1+m_2,0]{\nu,l_1+l_2}$ and admits a representation as stated in (\ref{AE:Product}). 

For that purpose we first split the integral representation of $c_{\sigma,\varepsilon}(x,\xi)$ into three parts. 
Therefore let $\chi \in \mathcal{C}^{\infty}_{\rm c} (\mathbb{R})$ such that $\chi(s) = 1$ for $|s| \leq 1/4$ and $\chi(s) = 0$ for $|s| \geq 1/2$ and set $\chi_1 (x,y) = \chi(|x-y|)$ for $x,y \in \mathbb{R}^n$. For $|\xi| \geq 1$ we write 
\begin{align}\label{Splitting}
c_{\sigma,\varepsilon} (x,\xi) =: d_{\sigma,\varepsilon} (x,\xi) + e_{\sigma,\varepsilon} (x,\xi) + f_{\sigma,\varepsilon} (x,\xi).
\end{align}
Here $d_{\sigma,\varepsilon}, e_{\sigma,\varepsilon}$ and $f_{\sigma,\varepsilon}$ are of the same form as $c_{\sigma,\varepsilon}$ but their integrands are multiplied by an additional factor $1-\chi(|\xi-\eta| / |\xi|)$, $\chi(|\xi-\eta| / |\xi|)(1-\chi_1(x,y))$ and $\chi(|\xi-\eta| / |\xi|)\chi_1(x,y)$, respectively. Also we note that $f_{\sigma,\varepsilon}$ represents the part of $c_{\sigma,\varepsilon}$ that will determine its asymptotic behavior as we meet the critical points of the phase. 

In the following we give the proof under the assumption that $|\xi| \geq 1$. In the case that $|\xi| \leq 1$ a similar decomposition as in (\ref{Splitting}) will lead to the desired result. More precisely, the corresponding summands are of the same form as $c_{\sigma,\varepsilon}$ but their integrands are multiplied with $1-\chi_1(\xi,\eta)$, $\chi_1(\xi,\eta)(1-\chi_1(x,y))$ and $\chi_1(\xi,\eta)\chi_1(x,y)$, respectively, which is similar as in the proof of \cite[Theorem 2.6.5]{Martinez:02}.

As a first step we check that the integrand of $d_{\sigma,\varepsilon}$ is in $L^1(\mathbb{R}^n_y \times \mathbb{R}^n_\eta)$ uniformly in $\sigma > 0$. Therefore we introduce the operator
\begin{equation*}
L_\varepsilon := \left( 1+ \frac{|\xi - \eta|^2}{\varepsilon^2} + \frac{|x-y|^2}{\varepsilon^2} \right)^{-1} \left( 1 + \frac{\xi - \eta}{\varepsilon}D_y + \frac{x-y}{\varepsilon} D_\eta \right)
\end{equation*}
satisfying the following equality
\begin{equation*}
L_\varepsilon e^{i(x-y)(\eta - \xi) /\varepsilon} = e^{i(x-y)(\eta - \xi) /\varepsilon}.
\end{equation*}
Using integration by parts we obtain 
\begin{equation*}
\begin{split}
&(2 \pi \varepsilon)^{n} d_{\sigma,\varepsilon} (x,\xi) =\\
&= \int e^{i(x-y)(\eta - \xi) /\varepsilon} ( {}^t \hspace{-1pt} L_\varepsilon )^k \Bigl\{ (1-\chi(|\xi-\eta|/|\xi|)) e^{- \sigma \langle y \rangle - \sigma \langle \eta \rangle} a_\varepsilon(x,\eta) b_{\varepsilon}(y,\xi) \Bigr\} \, dy \, d \eta.
\end{split}
\end{equation*}
Further, for every $k \in \mathbb{N}$ with $k > n+1/2 $ we get
\begin{align*}
(2 \pi \varepsilon)^{n} & d_{\sigma,\varepsilon}(x,\xi) = \int_{|\xi - \eta| \geq |\xi|/4} \mathcal{O} \left( \frac{\omega_\varepsilon^{k \nu+l_1+l_2} \langle \eta \rangle^{m_1} \langle \xi \rangle^{m_2}}{(1 + |\xi - \eta|/\varepsilon + |x-y|/\varepsilon)^k}\right) \, dy \, d \eta =\\
& = \ \int_{|\xi - \eta| \geq |\xi|/4} \mathcal{O} \left( \varepsilon^n \omega_\varepsilon^{k \nu+l_1+l_2} \langle \eta \rangle^{m_1} \langle \xi \rangle^{m_2} \Bigl\{1 + \frac{|\xi - \eta|}{\varepsilon} \Bigr\}^{n+1/2-k} \right) \, d \eta =\\
& = \ \int_{|\xi - \eta| \geq |\xi|/4} \mathcal{O} \left( \varepsilon^{k-1/2} \omega_\varepsilon^{k \nu +l_1+l_2} \frac{ \langle \eta \rangle^{m_1} \langle \xi \rangle^{m_2}}{ ( 1+|\xi| + |\eta| )^{k-n-1/2}}\right) \, d \eta \qquad \text{as } \varepsilon \to 0.
\end{align*}
Then, by a straightforward calculation one shows that for every $k \geq |m_1| + 2n+1$ we have
\begin{equation*}%\label{dsigma}
(2 \pi \varepsilon)^{n} d_{\sigma,\varepsilon}(x,\xi) = \mathcal{O} \left( \varepsilon^{k-1/2} \omega_\varepsilon^{k \nu+l_1+l_2} \langle \xi \rangle^{m_1+m_2-k+2n+1+|m_1|} \right)
\end{equation*}
for $\varepsilon$ sufficiently small. Since $(\omega_\varepsilon)_\varepsilon$ is a slow scale net we conclude that
\begin{equation*}
(2 \pi \varepsilon)^{n} d_{\sigma,\varepsilon}(x,\xi) = \mathcal{O} \left( \varepsilon^{N} \omega_\varepsilon^{l_1+l_2} \langle \xi \rangle^{m_1+m_2-N} \right) \qquad \forall \ N \geq 0
\end{equation*}
uniformly for every $(x,\xi) \in \mathbb{R}^{2n}$ with $|\xi| \geq 1$ and $\sigma > 0$ as $\varepsilon \to 0$.

We next estimate $e_{\sigma,\varepsilon}(x,\xi)$. Using the coordinate change $z=y-x$ and integration by parts we can write
\begin{equation*}
(2 \pi \varepsilon)^{n} e_{\sigma,\varepsilon} (x,\xi) =\int_{\substack{|z|\geq 1/4 \\ |\xi-\eta| \leq |\xi|/2}} e^{-iz(\eta - \xi) /\varepsilon} ( {}^t \hspace{-1pt} L_{1,\varepsilon} )^k r_{\sigma,\varepsilon}(x,z,\xi,\eta) \, dz \, d \eta
\end{equation*}
where we have set $L_{1,\varepsilon} := -\varepsilon z/|z|^2 D_{\eta}$ and
\begin{equation*}
r_{\sigma,\varepsilon}(x,z,\xi,\eta) := \chi(|\xi-\eta|/|\xi|)(1-\chi(|z|)) e^{- \sigma \langle x+z \rangle - \sigma \langle \eta \rangle} a_\varepsilon(x,\eta) b_{\varepsilon}(x+z,\xi).
\end{equation*}
Note that the operator $L_{1,\varepsilon}$ satisfies
\begin{equation*}
{}^t L_{1,\varepsilon} = -L_{1,\varepsilon} \quad \text{and} \quad L_{1,\varepsilon} e^{-iz(\eta - \xi) /\varepsilon} = e^{-iz(\eta - \xi) /\varepsilon}.
\end{equation*}
Also, the integrand of $e_{\sigma,\varepsilon} (x,\xi)$ is integrable with respect to $z$ if we take $k > n$. In order to check the integrability with respect to the $\eta$-variable we partition the domain of integration into the two regions
\begin{align*}
\Omega_1 &:= \{ \eta \in \mathbb{R}^{n} : |\xi-\eta| \le 1/4 \} &  \Omega_2 &:= \{ \eta \in \mathbb{R}^{n} : |\xi-\eta| \ge 1/4\}%\\
\end{align*}
and write 
\begin{equation*}
(2 \pi \varepsilon)^{n} e_{\sigma,\varepsilon} (x,\xi) = \sum_{j=1}^{2} \int_{\Omega_j} e^{-iz(\eta - \xi) /\varepsilon} ( {}^t \hspace{-1pt} L_{1,\varepsilon} )^k r_{\sigma,\varepsilon}(x,z,\xi,\eta) \, dz \, d \eta =: \sum_{j=1}^{2} I_j .
\end{equation*}
Concerning $I_1$ we use the coordinate transformation $\zeta = \eta- \xi$ and by Peetre's inequality $\langle \xi +\zeta \rangle^s \leq 2^{|s|/2} \langle \xi \rangle^s \langle \zeta \rangle^{|s|}$ for every $s \in \mathbb{R}$ we obtain
\begin{equation*}
I_1 = \int_{\substack{|z|\geq 1/4 \\ |\zeta| \leq 1/4}} \mathcal{O} (\varepsilon^k \omega_\varepsilon^{l_1+l_2} |z|^{-k} \langle \xi \rangle^{m_2} \langle \xi + \zeta \rangle^{m_1-k}) \, dz \, d \zeta = \mathcal{O} (\varepsilon^k \omega_\varepsilon^{l_1+l_2} \langle \xi \rangle^{m_1+m_2-k} ).
\end{equation*}
for every $k > n$ as $\varepsilon \to 0$. For the estimation of the second part $I_2$ we introduce
\begin{equation*}
L_{2,\varepsilon} := \frac{\varepsilon^2}{|\xi-\eta|^2} \frac{\xi-\eta}{\varepsilon} D_{z}
\end{equation*}
and let $j \in \mathbb{N}$. Again using integration by parts we obtain for $|\xi| \geq 1$
\begin{equation*}
\begin{split}
I_2 &= \int e^{-iz(\eta - \xi) /\varepsilon} ( {}^t \hspace{-1pt} L_{2,\varepsilon} )^j ( {}^t \hspace{-1pt} L_{1,\varepsilon} )^k r_{\sigma,\varepsilon}(x,z,\xi,\eta) \, dz \, d \eta=\\
&= \int_{\substack{ |z| \geq 1/4 \\ 1/4 \leq |\xi-\eta| \leq | \xi |/2}} \mathcal{O} (\varepsilon^{j + k} \omega_\varepsilon^{j \nu +l_1+l_2} |z|^{-k} \langle \xi \rangle^{m_2} \langle \eta \rangle^{m_1 - k} |\xi-\eta|^{-j})  \, dz \, d \eta
\end{split}
\end{equation*}
and is integrable with respect to $z$ if we take $k > n$ for sufficiently small $\varepsilon$. Then using the coordinate change $\eta=\xi+\zeta$ and applying Peetre's inequality gives
\begin{equation*}
\begin{split}
I_2 & = \int_{1/4 \leq |\zeta| \leq |\xi|/2} \mathcal{O} (\varepsilon^{j + k} \omega_\varepsilon^{j \nu +l_1+l_2}  \langle \xi \rangle^{m_1+m_2-k} \langle \zeta \rangle^{|m_1|+k -j} ) \, d \zeta =\\
& = \mathcal{O} (\varepsilon^{j + k} \omega_\varepsilon^{j \nu +l_1+l_2}  \langle \xi \rangle^{m_1+m_2-k} )
\end{split}
\end{equation*}
for every $j,k \in \mathbb{N}$ with $k> n$ and $j > |m_1| + k +n$ as $\varepsilon \to 0$. Combining the above yields to
\begin{align*}
(2 \pi \varepsilon)^{n} e_{\sigma,\varepsilon}(x,\xi) &= \mathcal{O} \left( \varepsilon^{N} \omega_\varepsilon^{l_1+l_2}\langle \xi \rangle^{m_1+m_2-N} \right) \qquad \forall N \geq 0
\end{align*}
uniformly for all $\sigma > 0$ and $(x,\xi) \in \mathbb{R}^{2n}$, $|\xi| \geq 1$ as $\varepsilon \to 0$. By the same arguments as above one can also show that for every $\alpha,\beta \in \mathbb{N}^n$, $|\xi| \geq 1$ and $N \in \mathbb{N}$ we have
\begin{equation*}
|\partial_{\xi}^{\alpha} \partial_x^{\beta} d_{\sigma,\varepsilon}(x,\xi)| + |\partial_{\xi}^{\alpha} \partial_x^{\beta} e_{\sigma,\varepsilon}(x,\xi)| = \mathcal{O} \left( \varepsilon^N \omega_\varepsilon^{\nu |\beta| +l_1+l_2}\langle \xi \rangle^{m_1+m_2-|\alpha| - N} \right)
\end{equation*}
uniformly for $\sigma > 0$ as $\varepsilon \to 0$. So we deduce that $(d_{\sigma,\varepsilon})_\varepsilon$ and $(e_{\sigma,\varepsilon})_\varepsilon$ are contained in $\MSym[m_1+m_2-N,N]{\nu,l_1+l_2}$ for every $N \in \mathbb{N}$ uniformly with respect to $\sigma>0$ and $(x,\xi) \in \mathbb{R}^{2n}$, $|\xi| \geq 1$.

So it remains to study the term $(2 \pi \varepsilon)^{n} f_{\sigma,\varepsilon}(x,\xi)$ which was given by 
\begin{equation*}
\int e^{i(x-y)(\eta - \xi) /\varepsilon} \chi(|\xi-\eta|/|\xi|)\chi_1(x,y)  e^{- \sigma (\langle y \rangle + \langle \eta \rangle)} a_\varepsilon(x,\eta) b_{\varepsilon}(y,\xi) \, dy \, d \eta.
\end{equation*}
Now writing $\xi = \lambda \nu$ where $\lambda = |\xi|$ and using the coordinate transformations
\begin{equation*}
\begin{cases}
  \zeta = (\eta - \xi) / \lambda \\
  z=y-x     									  
\end{cases}
\end{equation*}
we get
\begin{equation*}
f_{\sigma,\varepsilon} (x,\xi) =\frac{\lambda^n}{(2 \pi \varepsilon)^n} \int_{\substack{|z| \leq 1/2 \\ |\zeta| \leq 1/2}} e^{-i \lambda z \zeta/\varepsilon} 
t_{\sigma,\varepsilon}(z,\zeta; x,\xi) \, dz \, d\zeta.
\end{equation*}
where $t_{\sigma,\varepsilon}$ is a function in $\mathcal{C}_{\rm c}^{\infty} (\mathbb{R}^n_z \times \mathbb{R}^n_\zeta)$ containing $(x,\xi)$ as parameters and is given by
\begin{equation*}
t_{\sigma,\varepsilon}(z,\zeta; x,\xi) = \chi(|\zeta|) \chi(|z|) e^{- \sigma \langle x+z \rangle - \sigma \langle \lambda(\nu + \zeta) \rangle} a_\varepsilon(x,\lambda(\nu+ \zeta)) b_{\varepsilon}(x+z,\lambda \nu).
\end{equation*}
Then the method of stationary phase gives for every $N \in \mathbb{N},\ N \geq 1$
\begin{equation*}
f_{\sigma,\varepsilon}(x,\xi) = \sum_{|\alpha| \leq N-1} \frac{\varepsilon^{|\alpha|}}{\alpha! \lambda^{|\alpha|}} D^{\alpha}_\zeta \partial_z^{\alpha} t_{\sigma,\varepsilon}(z,\zeta;x,\xi)  \Big\vert _{\substack{z=0 \\ \zeta = 0}} + S_N(t_{\sigma,\varepsilon};\lambda/\varepsilon)
\end{equation*}
and the remainder can be estimated as follows
\begin{multline*}
|S_N(t_{\sigma,\varepsilon};\lambda / \varepsilon)| \leq \frac{C\varepsilon^N}{N! \lambda^N} \sum_{|\alpha + \beta| \leq 2n+1} \lVert \partial_\zeta^{\alpha} \partial_z^{\beta} (\partial_\zeta \cdot \partial_z)^N t_{\sigma,\varepsilon} \rVert_{L^1(\mathbb{R}_z^n \times \mathbb{R}_\zeta^n)} =\\ 
= \mathcal{O} \left( \varepsilon^{N} \omega_\varepsilon^{(2n+1+N)\nu+l_1+l_2} \lambda^{m_1 +m_2 -N} \right) = \mathcal{O} \left( \varepsilon^{N-1} \omega_\varepsilon^{l_1+l_2} \langle \xi \rangle^{m_1 +m_2 -N} \right)
\end{multline*}
for $|\xi| \geq 1$ and $\varepsilon$ sufficiently small. Finally, since similar estimates hold true for the derivatives and we conclude that $(S_N(t_{\sigma,\varepsilon}))_\varepsilon$ is contained in $\MSym[m_1+m_2-N,N-1]{\nu,l_1+l_2}$ for every $N \in \mathbb{N}$ uniformly with respect to $\sigma>0$ for every $(x,\xi) \in \mathbb{R}^{2n}$ with $|\xi| \geq 1$. 

As already mentioned in the beginning of the proof we now pass to the limit $\sigma \to 0_{+}$. 
Then by Lebesgue's dominated convergence theorem we have
\begin{equation*}
c_{\sigma,\varepsilon}(x,\xi) \to c_\varepsilon(x,\xi) \qquad \text{as } \sigma \to 0_+
\end{equation*}
so that equation (\ref{ABproduct}) reads
\begin{align*}
A_{\psi,\varepsilon} B_{\psi,\varepsilon} u_\varepsilon(x) &= (2 \pi \varepsilon)^{-n} \lim_{\tau \to 0_{+}} \int e^{ix \xi/\varepsilon - \tau \langle \xi \rangle} c_{\varepsilon}(x,\xi) \hat{u}_\varepsilon(\xi/\varepsilon) \, d \xi =\\
& = (2 \pi \varepsilon)^{-n} \int e^{ix \xi/\varepsilon} c_{\varepsilon}(x,\xi) \hat{u}_\varepsilon(\xi/\varepsilon) \, d \xi
\end{align*}
where the last equation holds in the sense of oscillatory integrals. Moreover the generalized symbol $(c_\varepsilon)_\varepsilon$ is equal to $(a_\varepsilon \# b_{\varepsilon})_\varepsilon$ as described in (\ref{AE:Product}) and in particular $(c_\varepsilon)_\varepsilon \in \MSym[m_1+m_2,0]{\nu,l_1+l_2}$. This finishes the proof.
\end{proof}
\section{A factorization procedure for \texorpdfstring{$L_\psi$}{Lpsi}}\label{sec:Fact}
Concerning products of $\psi$-pseudodifferential operators that approximate $L_\psi$ we make similar considerations as in the smooth setting which can be found in \cite[Appendix II]{Kumano-go:81} and \cite[Chapter 23]{Hoermander:3}. First we write for $L_\psi = \varepsilon^2 L_\varepsilon$ with $L_\varepsilon$ from the beginning 
\begin{equation}\label{GovOp}
\begin{split}
L_{\psi}(y,D_t,D_x,D_z) &=: (\partial_z^2 + A(y,D_t,D_x))_{\psi} 
\end{split}
\end{equation}
where $A_\psi = A_\psi(y,D_t,D_x)$ is the $\psi$-pseudodifferential operator with generalized symbol $(a_\varepsilon)_\varepsilon$ and
\begin{equation}\label{SymbolA}
a_\varepsilon(y,\tau,\xi) = c_\varepsilon(y)\tau^2 - \langle b_\varepsilon(y) \xi, \xi \rangle \qquad \varepsilon \in (0,1].
\end{equation}
To recall the requirements made on the coefficients $c_\varepsilon(y)$ and $b_{j,\varepsilon}(y), 1 \leq j \leq n-1$ we also refer to Section~\ref{sec:Intro}.

Before stating the main theorem of this section we give a few more details about notation. In the following we will study operators on $\mathbb{R}^{n+1}$ of the form
\begin{equation*}
S_\psi = \sum_{j=0}^2 S_{j,\psi} (y,D_t,D_x) (\partial_z^{2-j})_\psi
\end{equation*}
with coefficients $S_{j,\psi} \in {\mathop{\rm OP}}_\psi \MSym[k_j,0]{1,l_j}$ for some $k_j,l_j\in \mathbb{R}$, $j=0,1,2$. Further we write 
\begin{equation*}
S_\psi = \sum_{j=0}^2 S_{j,\psi} (y,D_t,D_x) (\partial_z^{2-j})_\psi \quad \text{ on } I'
\end{equation*}
when the symbols of the coefficients $S_{j,\psi}$ are restricted to a set $I'$.

Concerning the factorization domain we make similar considerations as in \cite{Stolk:04} and introduce the set
\begin{equation}\label{BasicSet}
I'_{\theta_1} := \Bigl\{ (y,\tau,\xi) \ | \ \tau \neq 0, \ \langle b(y)\xi,\xi \rangle < \sin^2 (\theta_1) c(y) \tau^2 \Bigr\}
\end{equation}
for some $\theta_1 \in (0, \pi/2)$. Then $I'_{\theta_1}$ is an open subset of $\mathbb{R}^n \times (\mathbb{R}^n \setminus 0)$ and conic with respect to the second variable. Moreover we have chosen $I'_{\theta_1}$ independent of the parameter $\varepsilon \in (0,1]$ for simplicity. Note that the main reason for restricting the analysis to the domain $I'_{\theta_1}$ is that $(a_\varepsilon)_\varepsilon$ is non-negative there.

Further, as already mentioned in Section~\ref{sec:PsiPseudoCalculus} it is necessary to introduce the notion of polyhomogeneous generalized symbols that are smoothed off at the origin.

\begin{defn}\label{defn:PhgSymb}
We say that a generalized symbol $(p_\varepsilon)_\varepsilon$ is polyhomogeneous in $\MSym[m,k]{\nu,l}$, denoted by $(p_\varepsilon)_\varepsilon \in \MphgSym[m,k]{\nu,l}$, if there exists a sequence of symbols $(\overline{p}_{m-j,\varepsilon})_\varepsilon$ in $\MhgSym[m-j,k]{\nu,\nu j+l}$, $j \geq 0$ and a cut-off $\varphi \in \mathcal{C}_{\rm c}^{\infty}(\mathbb{R}^n)$ equal to $1$ near the origin so that $\forall N \geq 1$ we have
\begin{equation}\label{AE:PhgSymb}
(p_\varepsilon - \sum_{j = 0}^{N-1} \varepsilon^j (1-\varphi) \overline{p}_{m-j,\varepsilon})_\varepsilon \in \MSym[m-N,N+k]{\nu,\nu N+l} \quad \text{ for } \xi \text{ sufficiently large.}
\end{equation} 
For this we will use the following notation
\begin{equation*}
(p_\varepsilon)_\varepsilon \dot{\sim} \ \sum_{j \geq 0} (\varepsilon^j \overline{p}_{m-j,\varepsilon})_\varepsilon \qquad \text{in } \MSym[m,k]{\nu,l}.
\end{equation*}
Here the homogeneous part $(\overline{p}_{m,\varepsilon})_\varepsilon$ is called the principal symbol of $(p_\varepsilon)_\varepsilon$ if there exist an $\eta \in (0,1]$ and a constant $ K > 0$ such that $\overline{p}_{m,\varepsilon} \not\equiv 0$ for all $|\xi| \geq K$ and any fixed $\varepsilon \in (0,\eta]$. In particular a $\psi$-pseudodifferential operator is polyhomogeneous if its generalized symbol is polyhomogeneous. 
\end{defn}
\begin{rem}
First note that the principal symbol is uniquely determined if the dual variable is sufficiently large. Moreover one can state (\ref{AE:PhgSymb}) more precisely, that is $\forall N \geq 1$ we have 
\begin{equation*}
(p_\varepsilon - \sum_{j = 0}^{N-1} \varepsilon^j (1-\varphi) \overline{p}_{m-j,\varepsilon})_\varepsilon \in \MSym[m-N,N+k]{\nu,\nu N+l} 
\end{equation*}
modulo some compactly supported regular generalized symbol. Here the latter error term corresponds to a regular generalized symbol $(r_\varepsilon)_\varepsilon$ of order $-\infty$ in the following sense
\begin{equation*}
\begin{split}
\exists \eta \in (0,1] \ \exists k \in \mathbb{R} \ \forall m \in \mathbb{R} \ \forall \alpha, \beta \in \mathbb{N}^n & \ \exists C > 0 \text{ such that}\\
& | \partial_{\xi}^{\alpha} \partial_x^{\beta} r_\varepsilon (x,\xi) |  \leq C \varepsilon^{k} \langle \xi \rangle^{m - |\alpha|} \quad \forall \varepsilon \in (0,\eta].
\end{split}
\end{equation*}
\end{rem}
The aim of this section is now to factorize $L_\psi$ on $I'_{\theta_1}$ in terms of two first-order operators of the form $L_{j,\psi} = (\partial_z + A_{j})_\psi$ on $I'_{\theta_1}$ with $A_{j,\psi}$ a polyhomogeneous $\psi$-pseudodifferential operator with generalized symbol in $\MphgSym[1,0]{1,0}$ on $I'_{\theta_1}$, $j=1,2$. Here a first rough approximation would suggest $A_{j,\psi}$ to be the square root of the operator $A_\psi = {\mathop{\rm OP}}_\psi ((a_\varepsilon)_\varepsilon)$ given in (\ref{GovOp}) on the set $I'_{\theta_1}$ where $\pm (i \sqrt{a_\varepsilon})_\varepsilon$ is described explicitly, $j=1,2$.

We are now going to show the following result:
\begin{thm}\label{thm:Fact}
Let $L_\psi = (\partial_z +A)_\psi$ and $I'_{\theta_1}$ be as in (\ref{GovOp}) and (\ref{BasicSet}) respectively. 
Then the operator $L_\psi$ can be factorized into
\begin{equation}\label{eqn:Fact}
L_\psi = L_{1,\psi} L_{2,\psi} + R_\psi \qquad \text{ on } I'_{\theta_1}
\end{equation}
where $L_{j,\psi} = (\partial_z + A_{j})_\psi$ and $A_{j,\psi} \in {\mathop{\rm OP}}_\psi \MphgSym[1,0]{1,0}$ on $I'_{\theta_1}$ whose principal symbol is equal to $\pm (i \sqrt{a_\varepsilon})_\varepsilon$, $j=1,2$. Moreover the remainder is given by $R_\psi = \Gamma_{0,\psi} + \Gamma_{1,\psi}(\partial_z)_\psi$ for some $\psi$-pseudodifferential operators $\Gamma_{j,\psi}$ with generalized symbol $(\gamma_{j,\varepsilon})_\varepsilon$ in $\mathcal{N}_{{\displaystyle\mathcal{S}}^{2-j}}$ on $I'_{\theta_1}$, $j=0,1$.
\end{thm}
We remark that the product $L_{1,\psi} L_{2,\psi}$ in (\ref{eqn:Fact}) is not a $\psi$-pseudodifferential operator on $\mathbb{R}^{n+1}$. But one can overcome this by introducing a generalized cut-off $g_\psi (D_t,D_y)$ such that the difference $L_{1,\psi} L_{2,\psi} g_\psi - L_{1,\psi} L_{2,\psi}$ is insignificant on some adequate subdomain of the phase space $T^* \mathbb{R}^{n+1}\setminus 0$ under a microlocal point of view. 
For the specification for microlocal analysis we refer to the next section.
\subsection*{Technical preliminaries}

Note that $a_\varepsilon$ as in (\ref{SymbolA}) is a homogeneous function with respect to the dual variables $(\tau,\xi)$ and so that $\exists \eta \in (0,1]$ such that
\begin{equation*}
|a_\varepsilon(y,\tau,\xi)| \geq C |(\tau,\xi)|^2 \qquad \text{ on } I'_{\theta_2} \text{ for every } \varepsilon \in (0,\eta]
\end{equation*}
for some fixed $\theta_2 \in (0,\pi/2)$ with $\theta_1 < \theta_2$ and the constant $C>0$ is independent of $\varepsilon$. Moreover we set for some $\varepsilon_0 \in (0,1]$ which will be specified later
\begin{equation}\label{RootA}
\overline{b}_{j,\varepsilon}^{(0)}(y,\tau,\xi) := \begin{cases} 
\pm i \sqrt{a_\varepsilon(y,\tau,\xi)}, & \varepsilon \in (0,\varepsilon_0] \\
0, & \varepsilon \in (\varepsilon_0,1]
\end{cases} \qquad \text{on } I'_{\theta_2}
\end{equation}
which in turn gives $(\overline{b}^{(0)}_{j,\varepsilon})_\varepsilon \in \MhgSym[1,0]{1,0}$ on $I'_{\theta_2}$. To $(\overline{b}^{(0)}_{j,\varepsilon})_\varepsilon$ we now associate a generalized symbol $(\overline{c}^{(0)}_{j,\varepsilon})_\varepsilon \in \MSym[1,0]{1,0}(\mathbb{R}^n \times (\mathbb{R}\setminus 0) \times \mathbb{R}^{n-1})$ so that the difference $(\overline{c}^{(0)}_{j,\varepsilon} -  \overline{b}^{(0)}_{j,\varepsilon})_\varepsilon$ is in $\mathcal{N}_{{\displaystyle\mathcal{S}}^1}$ on $I_{\theta_1}'$ in the sense of Definition~\ref{defn:SymbOpenSet}. 

Therefore since $(c_\varepsilon(y))_\varepsilon$ is strictly non-zero we can define for fixed $\varepsilon \in(0,\varepsilon_1]$
\begin{equation*}
f_\varepsilon(y,\tau,\xi) := \frac{\langle b_\varepsilon(y) \xi, \xi\rangle}{c_\varepsilon(y) \tau^2} \qquad \text{on } \mathbb{R}^n \times (\mathbb{R} \setminus 0) \times \mathbb{R}^{n-1}.
\end{equation*}
Recall that $\varepsilon_1 \in (0,1]$ is so that $\inf_{y \in \mathbb{R}^n} |c_\varepsilon(y)| \geq C$ for some $C>0$ and every $\varepsilon \in (0,\varepsilon_1]$. Furthermore for every fixed $\varepsilon \in (0,\varepsilon_1]$ we let $\widetilde{h}_\varepsilon$ be the smooth function defined on $\mathbb{R}^n \times (\mathbb{R} \setminus 0) \times \mathbb{R}^{n-1}$, $0 \leq \widetilde{h}_\varepsilon \leq 1$, which is given by 
\begin{equation}\label{CutOff}
\widetilde{h}_\varepsilon(y,\tau,\xi) := \begin{cases}
0, & |f_\varepsilon| \geq \sin^2 \gamma_2 \\
1, & |f_\varepsilon| \leq \sin^2 \gamma_1\\
\frac{1}{{1 + e^{\frac{1}{|f_\varepsilon| - \sin^2 \gamma_1} - \frac{1}{\sin^2 \gamma_2 - |f_\varepsilon|}}}}, & \sin^2 \gamma_1 < |f_\varepsilon| < \sin^2 \gamma_2
\end{cases}
\end{equation}
for some fixed $\gamma_1$ and $\gamma_2$ with $0 < \gamma_1 < \gamma_2 < \pi/2$. For all other $\varepsilon \in (0,1]$ we may set $\widetilde{h}_\varepsilon \equiv 0$ and we obtain that $(\widetilde{h}_\varepsilon)_\varepsilon \in \MSym[0,0]{1,0}(\mathbb{R}^n \times (\mathbb{R}\setminus 0) \times \mathbb{R}^{n-1})$ and also $(\widetilde{h}_\varepsilon)_\varepsilon \in \MhgSym[0,0]{1,0}$ on $I'_{\theta_2}$.

Here and hereafter we let $\theta_1, \theta_2$ be fixed so that $0 < \theta_1 < \gamma_1 < \gamma_2 < \theta_2 < \pi/2$. Then there is an $\varepsilon_2 \in(0,1]$ such that
\begin{equation*}
I'_{\theta_1} \subset \{ (y,\tau,\xi) \ | \ \tau \neq 0 \text{ and } |f_\varepsilon| \leq \sin^2 \gamma_1 \} \qquad \varepsilon \in (0,\varepsilon_2]
\end{equation*}
because of the assumptions made on the coefficients $c_\varepsilon(y), b_{j,\varepsilon}(y), y \in \mathbb{R}^n, 1 \leq j \leq n-1$.

Similar observations yield to the following: $\exists \varepsilon_3 \in (0,1]$ such that
\begin{equation*}
\{ (y,\tau,\xi) \ | \ \tau \neq 0, |f| \geq \sin^2 \theta_2 \} \subset \{(y,\tau,\xi) \ | \ \tau \neq 0, |f_\varepsilon| \geq \sin^2 \gamma_2 \}
\end{equation*}
for all $\varepsilon \in (0,\varepsilon_3]$ where we have set $f := \lim_{\varepsilon \to 0} f_\varepsilon$.

Then with $(\overline{b}^{(0)}_{j,\varepsilon})_\varepsilon$ as in (\ref{RootA}) we set $\varepsilon_0 := \min_{1 \leq k \leq 3} \varepsilon_k$ and define its extension $(\overline{c}^{(0)}_{j,\varepsilon})_\varepsilon$ by
\begin{equation*}
(\overline{c}^{(0)}_{j,\varepsilon})_\varepsilon := (\widetilde{h}_\varepsilon \overline{b}^{(0)}_{j,\varepsilon})_\varepsilon \in \MSym[1,0]{1,0} (\mathbb{R}^n \times (\mathbb{R}\setminus 0) \times \mathbb{R}^{n-1}).
\end{equation*}
Indeed we have that $(\overline{c}^{(0)}_{j,\varepsilon} - \overline{b}^{(0)}_{j,\varepsilon})_\varepsilon \in \mathcal{N}_{{\displaystyle\mathcal{S}}^1}$ on $I_{\theta_1}'$ since $\widetilde{h}_\varepsilon \equiv 1$ on $I_{\theta_1}'$.
\subsection*{Factorization procedure}

Our aim here is to decompose the operator $L_\psi$ as announced in (\ref{eqn:Fact}). Therefore we give a construction scheme for the generalized symbols $(a_{j,\varepsilon})_\varepsilon$ of the operators $A_{j,\psi}$, $j=1,2$ by means of their polyhomogeneous asymptotic expansions, i.e. on the set $I'_{\theta_1}$ we have
\begin{equation}\label{AE:aj}
(a_{j,\varepsilon}(y,\tau,\xi))_\varepsilon \ \dot{\sim} \ \sum_{\mu \geq 0} \bigl(\varepsilon^{\mu} \overline{c}_{j,\varepsilon}^{(\mu)}(y,\tau,\xi) \bigr)_\varepsilon \quad \text{ in } \MSym[1,0]{1,0} 
\end{equation}
and the sequence $\{ (\overline{c}_{j,\varepsilon}^{(\mu)})_\varepsilon\}_{\mu \in \mathbb{N}}$ consists of elements $(\overline{c}_{j,\varepsilon}^{(\mu)})_\varepsilon \in \MSym[1-\mu,0]{1,\mu}(\mathbb{R}^n \times (\mathbb{R}\setminus 0) \times \mathbb{R}^{n-1})$ and satisfies the uniformity condition 
\begin{multline*}
\qquad \exists \eta \in (0,1] \ \exists K > 0 \ \forall \mu \in \mathbb{N} \ \forall \alpha,\beta \in \mathbb{N}^n \ \exists C>0: \\
|\partial_{(\tau,\xi)}^{\alpha} \partial_y^{\beta} \overline{c}_{j,\varepsilon}^{(\mu)}(y,\tau,\xi)| \leq C \omega_\varepsilon^{|\beta|} \langle \xi\rangle^{1-\mu-|\alpha|} \quad \text{for } |\tau| \geq K, \ \varepsilon \in (0,\eta]\qquad
\end{multline*}
for $j=1,2$. More precisely we will recursively construct a sequence $\{(\overline{b}_{j,\varepsilon}^{(\mu)})_\varepsilon\}_{\mu}$ of symbols $(\overline{b}_{j,\varepsilon}^{(\mu)})_\varepsilon$ in $\MhgSym[1-\mu,\mu]{1,\mu}$ on $I'_{\theta_2}$, $j=1,2$ such that 
\begin{equation*}
\bigl( \overline{c}_{j,\varepsilon}^{(\mu)}(y,\tau,\xi) \bigr)_\varepsilon = \bigl( \varepsilon^{-\mu} \overline{b}_{j,\varepsilon}^{(\mu)}(y,\tau,\xi) \bigr)_\varepsilon \qquad \text{on } I'_{\theta_2}.
\end{equation*}
Recall that on $I'_{\theta_1}$ (\ref{AE:aj}) is equivalent to the following: there exists a smooth cut-off equal to $1$ near the origin such that for every $N \geq 1$ we have on $I'_{\theta_1}$
\begin{equation*}
\bigl(a_{j,\varepsilon} - \sum_{\mu = 0}^{N-1} \varepsilon^\mu (1-\varphi) \overline{c}_{j,\varepsilon}^{(\mu)} \bigr)_\varepsilon \in \MSym[1-N,N]{1,N} \quad \text{for } |(\tau,\xi)| \text{ sufficiently large.}
\end{equation*}
\begin{proof}[Proof of Theorem~\ref{thm:Fact}]
To begin with, we set for $j=1,2$ and $\varepsilon \in (0,\varepsilon_0]$
\begin{equation*}
\begin{array}{cccccc}
\overline{a}_{j,\varepsilon}^{(1)}:= \overline{b}_{j,\varepsilon}^{(0)} \ && \ l_{j,\varepsilon}^{(1)} := i\zeta + a_{j,\varepsilon}^{(1)}.
\end{array}
\end{equation*}
where $\overline{b}_{j,\varepsilon}^{(0)}(y,\tau,\xi) = \pm i\sqrt{ a_\varepsilon(y,\tau,\xi) }$ on $I'_{\theta_2}$, $j=1,2$. Moreover $(a^{(1)}_{j,\varepsilon})_\varepsilon \in \MphgSym[1,0]{1,0}$ with polyhomogeneous asymptotic expansion $(\widetilde{h}_\varepsilon \overline{b}_{j,\varepsilon}^{(0)})_\varepsilon$ with $\widetilde{h}_\varepsilon$ as in (\ref{CutOff}) and $(\overline{b}_{j,\varepsilon}^{(0)})_\varepsilon$ only prescribed on $I'_{\theta_2}$, $j=1,2$. As in (\ref{RootA}) we set $\overline{b}_{j,\varepsilon}^{(0)}$ and $l_{j,\varepsilon}^{(1)}$ equal to zero for all other $\varepsilon \in (0,1]$ . Furthermore we define the first-order operator $L_{j,\psi}^{(1)}$ through $(\partial_z + A_{j}^{(1)})_\psi$ where $A^{(1)}_{j,\psi} = {\mathop{\rm OP}}_\psi ((a_{j,\varepsilon}^{(1)})_\varepsilon)$ for $j=1,2$.

Now taking $L_{1,\psi}^{(1)} L_{2,\psi}^{(1)}$ as a first approximation for $L_\psi$ we can compute the error as follows:
\begin{equation*}
\begin{split}
L_{1,\psi}^{(1)} L_{2,\psi}^{(1)} - L_{\psi} &=\phantom{:} (\partial_z + A_1^{(1)})_{\psi} (\partial_z + A_2^{(1)})_{\psi} - (\partial^2_z + A)_{\psi} =\\
& =\phantom{:}  (A^{(1)}_{1} + A^{(1)}_{2})_{\psi} (\partial_z)_{\psi} + {\mathop{\rm OP}}_{\psi}\bigl( (\varepsilon \hspace{1pt} \partial_z a_{2,\varepsilon}^{(1)})_\varepsilon \bigr) + A_{1,\psi}^{(1)} A_{2,\psi}^{(1)} - A_{\psi} =\\
& =: \Gamma_{0,\psi}^{(1)} + \Gamma_{1,\psi}^{(1)}(\partial_z)_{\psi}
\end{split}
\end{equation*}
where we have used the Leibniz composition rule for $\psi$-pseudodifferential operators, that is $(\partial_z)_{\psi} A_{2,\psi}^{(1)} = A_{2,\psi}^{(1)} (\partial_z)_{\psi} + {\mathop{\rm OP}}_{\psi} \bigl( (\varepsilon \partial_z a_{2,\varepsilon}^{(1)})_\varepsilon \bigr)$.
 
Then $\Gamma_{1,\psi}^{(1)} = (A^{(1)}_{1} + A^{(1)}_{2})_{\psi} \in {\mathop{\rm OP}}_\psi \mathcal{N}_{{\displaystyle\mathcal{S}}^1}$ on $I'_{\theta_1}$. Moreover $\Gamma_{0,\psi}^{(1)}$ is equal to ${\mathop{\rm OP}}_{\psi}\bigl( (\varepsilon \hspace{1pt} \partial_z a_{2,\varepsilon}^{(1)})_\varepsilon \bigr) + A_{1,\psi}^{(1)} A_{2,\psi}^{(1)} - A_{\psi}$ and is a polyhomogeneous  $\psi$-pseudodifferential operator with generalized symbol $(\gamma_{0,\varepsilon}^{(1)})_\varepsilon$ in $\MphgSym[1,1]{1,1}(I'_{\theta_1})$ by Proposition~\ref{prop:Product}. In detail $(\gamma_{0,\varepsilon}^{(1)})_\varepsilon$ admits the asymptotic expansion 
\begin{equation*}
(\gamma_{0,\varepsilon}^{(1)})_\varepsilon \ \dot{\sim} \ \Bigl(\varepsilon \partial_z \overline{a}_{2,\varepsilon}^{(1)} + \sum_{|\alpha| \ge 1} \frac{{\varepsilon}^{|\alpha|}}{\alpha!} \bigl( D_{\xi}^{\alpha} \overline{a}_{1,\varepsilon}^{(1)} \bigr) \bigl( \partial_{x}^{\alpha} \overline{a}_{2,\varepsilon}^{(1)} \bigr) \Bigr)_\varepsilon  \qquad \text{on } I'_{\theta_1}.
\end{equation*}
To improve this approximation we proceed by induction. For the convenience of the reader we also compute the second order approximation for $L_\psi$.

Therefore we let $\varepsilon \in (0,\varepsilon_0]$ be fixed and define
\begin{equation*}
\begin{array}{cccc}
\overline{a}_{j,\varepsilon}^{(2)}:= \overline{a}_{j,\varepsilon}^{(1)} + \overline{b}_{j,\varepsilon}^{(1)} \ && \ l_{j,\varepsilon}^{(2)} := i\zeta + a_{j,\varepsilon}^{(2)}.
\end{array}
\end{equation*}
where $(a_{j,\varepsilon}^{(2)})_\varepsilon$ denotes the symbol with polyhomogeneous asymptotic expansion $(\widetilde{h}_\varepsilon \overline{a}_{j,\varepsilon}^{(2)})_\varepsilon$ and the existence of the generalized symbol $(\overline{b}_{j,\varepsilon}^{(1)})_\varepsilon$ will be clarified immediately. Again we set $\overline{b}_{j,\varepsilon}^{(1)}$ and $l_{j,\varepsilon}^{(2)}$ equal zero for all other $\varepsilon \in (0,1]$. From the preceding observations we have the approximation
\begin{equation*}
L_{1,\psi}^{(1)} L_{2,\psi}^{(1)}  - L_{\psi} = \Gamma_{0,\psi}^{(1)} + \Gamma_{1,\psi}^{(1)} (\partial_z)_\psi \qquad \text{ on } I_{\theta_1}'.  \end{equation*}
In order to obtain a suitable second order approximation we would have to keep the following expression small
\begin{equation}\label{SecApprox}
\begin{split}
L_{1,\psi}^{(2)} L_{2,\psi}^{(2)} - L_{\psi} &= ( L_{1}^{(1)} + B_{1}^{(1)} )_\psi ( L_{2}^{(1)} +  B_{2}^{(1)} )_\psi - L_{\psi} = \\
 & \hspace{2pt}
 \begin{split} = & \ \Gamma_{0,\psi}^{(1)} + \Gamma_{1,\psi}^{(1)}(\partial_z)_{\psi} + ( B_{1}^{(1)} + B_{2}^{(1)} )_\psi ( \partial_z )_{\psi} +\\ & + {\mathop{\rm OP}}_{\psi}\bigl( (\varepsilon \partial_z b_{2,\varepsilon}^{(1)})_\varepsilon \bigr) + A_{1,\psi}^{(1)} B_{2,\psi}^{(1)} + B_{1,\psi}^{(1)} A_{2,\psi}^{(1)} + B_{1,\psi}^{(1)} B_{2,\psi}^{(1)} 
 \end{split}
\end{split}
\end{equation}
where in the last step we have used the Leibniz rule. 

We now specify $(\overline{b}_{1,\varepsilon}^{(1)})_\varepsilon, (\overline{b}_{2,\varepsilon}^{(1)})_\varepsilon$ on $I'_{\theta_2}$ as follows: because the the generalized symbol $(\overline{a}_{j,\varepsilon}^{(1)})_\varepsilon$ satisfies
\begin{equation}\label{Ell:ps}
\exists \eta \in (0,1]: \quad |\overline{a}_{j,\varepsilon}^{(1)}(y,\tau,\xi)| \geq C |(\tau,\xi)| \qquad \text{ on } I'_{\theta_2}, \ \forall \varepsilon \in (0,\eta]
\end{equation}
for some constant $C>0$ that is independent of $\varepsilon \in (0,\eta]$, $j=1,2$, the matrix
\begin{equation}\label{CoeffMatrix}
\begin{pmatrix} 1 & 1 \\ \overline{a}_{2,\varepsilon}^{(1)} & \overline{a}_{1,\varepsilon}^{(1)} \end{pmatrix}
\end{equation}
is invertible on $I'_{\theta_2}$. We denote by $\bigl(\widetilde{\gamma}_{0,\varepsilon}^{(1)}\bigr)_\varepsilon$ and $\bigl(\widetilde{\gamma}_{1,\varepsilon}^{(1)}\bigr)_\varepsilon$ the naturally extended top order symbols on $I'_{\theta_2}$ of $\bigl(\gamma_{0,\varepsilon}^{(1)}\bigr)_\varepsilon$ and $\bigl(\gamma_{1,\varepsilon}^{(1)}\bigr)_\varepsilon$ on $I'_{\theta_1}$ respectively. Thus the system 
\begin{equation*}
\begin{cases}
\begin{array}{cclc}
-\widetilde{\gamma}_{1,\varepsilon}^{(1)} &=&  \overline{b}_{1,\varepsilon}^{(1)} + \overline{b}_{2,\varepsilon}^{(1)} \\
-\widetilde{\gamma}_{0,\varepsilon}^{(1)} &=&  \overline{b}_{1,\varepsilon}^{(1)} \overline{a}_{2,\varepsilon}^{(1)} + \overline{a}_{1,\varepsilon}^{(1)} \overline{b}_{2,\varepsilon}^{(1)} \end{array}
\end{cases}
\end{equation*}
is uniquely solvable for $(\overline{b}_{j,\varepsilon}^{(1)})_\varepsilon $ in $\MhgSym[0,1]{1,1}$ on $I'_{\theta_2}$, $j=1,2$.

With this choice of $(\overline{b}_{j,\varepsilon}^{(1)})_\varepsilon$ the second-order approximation for the operator $L_\psi$ in (\ref{SecApprox}) now reads
\begin{equation*}
L_{1,\psi}^{(2)} L_{2,\psi}^{(2)}  - L_{\psi} = \Gamma_{0,\psi}^{(2)} + \Gamma_{1,\psi}^{(2)}(\partial_z)_{\psi} \qquad \text{on } I'_{\theta_1}
\end{equation*}
and $\Gamma_{0,\psi}^{(2)} = \Gamma_{0,\psi}^{(2)}(y,D_t,D_x)$ and $\Gamma_{1,\psi}^{(2)} = \Gamma_{1,\psi}^{(2)}(y,D_t,D_x)$ are the $\psi$-pseudodifferential operators with generalized symbols $(\gamma_{0,\varepsilon}^{(2)})_{\varepsilon}$ in $\MphgSym[0,2]{1,2}$ on $I'_{\theta_1}$ and $(\gamma_{1,\varepsilon}^{(2)})_{\varepsilon}$ is in $\mathcal{N}_{{\displaystyle\mathcal{S}}^1}$ on $I'_{\theta_1}$, respectively.

To continue the proof, we assume for $N \ge 1$ that $(\overline{b}_{j,\varepsilon}^{(\nu)})_\varepsilon \in \MhgSym[1-\nu,\nu]{1,\nu}$ on $I'_{\theta_2}$ is determined for all $\nu \le N-1$ and $j=1,2$. For fixed $\varepsilon \in (0,\varepsilon_0]$ and $j=1,2$ we set
\begin{equation*}
\begin{array}{cccccc}
\overline{a}_{j,\varepsilon}^{(N)} := \displaystyle{\sum_{\nu = 0}^{N-1}} \overline{b}_{j,\varepsilon}^{(\nu)} \ && \ L_{j,\psi}^{(N)} = (\partial_z + A_j^{(N)})_{\psi}.
\end{array}
\end{equation*}
with $A_{j,\psi}^{(N)} (y,D_t,D_x) = {\mathop{\rm OP}}_{\psi}\bigl( (a_{j,\varepsilon}^{(N)})_\varepsilon \bigr)$ with $(a_{j,\varepsilon}^{(N)})_\varepsilon$ the symbol with polyhomogeneous expansion $(\widetilde{h}_\varepsilon \overline{a}_{j,\varepsilon}^{(N)})_\varepsilon$. Again for all other $\varepsilon \in (0,1]$ we set $a_{j,\varepsilon}^{(N)}$ and $L_{j,\varepsilon}^{(N)}$ equal to zero, $j=1,2$. Furthermore, we suppose an $N$-th order approximation for $L_\psi$ of the form
\begin{equation*}
L_{1,\psi}^{(N)} L_{2,\psi}^{(N)} - L_{\psi} = \Gamma_{0,\psi}^{(N)} + \Gamma_{1,\psi}^{(N)}(\partial_z)_\psi \qquad \text{on } I'_{\theta_1}
\end{equation*}
where $\Gamma_{0,\psi}^{(N)}$ and $\Gamma_{1,\psi}^{(N)}$ are the $\psi$-pseudodifferential operators with generalized symbols $(\gamma_{0,\varepsilon}^{(N)})_\varepsilon$ in $\MphgSym[2-N,N]{1,N}$ on $I'_{\theta_1}$ and $(\gamma_{1,\varepsilon}^{(N)})_\varepsilon$ in $\mathcal{N}_{{\displaystyle\mathcal{S}}^1}$ on $I'_{\theta_1}$, respectively. 

Again we denote by $\bigl(\widetilde{\gamma}_{0,\varepsilon}^{(N)}\bigr)_\varepsilon$ and $\bigl(\widetilde{\gamma}_{1,\varepsilon}^{(N)}\bigr)_\varepsilon$ the naturally extended top order terms on $I'_{\theta_2}$ of $\bigl(\gamma_{0,\varepsilon}^{(N)}\bigr)_\varepsilon$ and $\bigl(\gamma_{1,\varepsilon}^{(N)}\bigr)_\varepsilon$ on $I'_{\theta_1}$. For the induction step we specify $(b_{1,\varepsilon}^{(N)})_\varepsilon, (b_{2,\varepsilon}^{(N)})_\varepsilon$ on $I'_{\theta_2}$ as follows: since $(\overline{a}_{j,\varepsilon}^{(1)})_\varepsilon$ satisfies (\ref{Ell:ps}) the matrix in (\ref{CoeffMatrix}) is invertible on $I'_{\theta_2}$
and thus the system 
\begin{equation*}
\begin{cases}
\begin{array}{cclc}
-\widetilde{\gamma}_{1,\varepsilon}^{(N)} &=&  \overline{b}_{1,\varepsilon}^{(N)} + \overline{b}_{2,\varepsilon}^{(N)} \\
-\widetilde{\gamma}_{0,\varepsilon}^{(N)} &=&  \overline{b}_{1,\varepsilon}^{(N)} \overline{a}_{2,\varepsilon}^{(1)} + \overline{a}_{1,\varepsilon}^{(1)}  \overline{b}_{2,\varepsilon}^{(N)} \end{array}
\end{cases}
\end{equation*}
is uniquely solvable for $(\overline{b}_{j,\varepsilon}^{(N)})_\varepsilon $ in $\MhgSym[1-N,N]{1,N}$ on $I'_{\theta_2}$, $j=1,2$.

We write $B_{j,\psi}^{(N)}$ for the $\psi$-pseudodifferential operator with polyhomogeneous generalized symbol $(b_{1,\varepsilon}^{(N)})_\varepsilon:=(\widetilde{h}_\varepsilon \overline{b}_{j,\varepsilon}^{(N)})_\varepsilon$ and set $L_{j,\psi}^{(N+1)} := (L_{j}^{(N)} + B_j^{(N)})_\psi$. Then the following is valid
\begin{equation*}
\begin{split}
L_{1,\psi}^{(N+1)} L_{2,\psi}^{(N+1)} - L_{\psi} & = \bigl( L_{1}^{(N)} + B_{1}^{(N)} \bigr)_\psi \bigl( L_{2}^{(N)} + B_{2}^{(N)} \bigr)_\psi - L_{\psi} = \\
 \hspace{20pt}& \hspace{2pt}\begin{split} = & \ \Gamma_{0,\psi}^{(N)} + \Gamma_{1,\psi}^{(N)}(\partial_z)_\psi + \bigl( B_{1}^{(N)} + B_{2}^{(N)} \bigr)_\psi \bigl( \partial_z \bigr)_{\psi} + \\
  &+ {\mathop{\rm OP}}_{\psi} \bigl( (\varepsilon \partial_z b_{2,\varepsilon}^{(N)})_\varepsilon \bigr) + A_{1,\psi}^{(N)} B_{2,\psi}^{(N)} + B_{1,\psi}^{(N)} A_{2,\psi}^{(N)} + B_{1,\psi}^{(N)} B_{2,\psi}^{(N)}.
 \end{split}
\end{split}
\end{equation*}
Indeed we have an ($N$+1)-th order approximation of the following form
\begin{equation*}
L_{1,\psi}^{(N+1)} L_{2,\psi}^{(N+1)} - L_{\psi} = \Gamma_{0,\psi}^{(N+1)} + \Gamma_{1,\psi}^{(N+1)}(\partial_z)_\psi
\end{equation*}
and $\Gamma_{0,\psi}^{(N+1)}$ and $\Gamma_{1,\psi}^{(N+1)}$ are the $\psi$-pseudodifferential operators with generalized symbols $(\gamma_{0,\varepsilon}^{(N+1)})_\varepsilon \in \MphgSym[1-N,N+1]{1,N+1}$ and $(\gamma_{1,\varepsilon}^{(N+1)})_\varepsilon \in \mathcal{N}_{{\displaystyle\mathcal{S}}^1}$ on $I'_{\theta_1}$, respectively. This completes the induction step. 
\end{proof} 
\section{The Generalized Infinite Wave Front Set}\label{sec:WFInfty}
In this section we will discuss an alternative description of microlocality and regularity theory in Colombeau algebras of generalized functions. We refer to \cite{NPS:98, GaHo:05, GaHo:06} for more details on the commonly used notion of a generalized wave front set. Therein the generalized wave front set is explained by replacing the standard $\mathcal{C}^{\infty}$-regularity by $\mathcal{G}^{\infty}$-regularity. In the same manner we can reformulate local $H^{\infty}$-regularity used in the Sobolev based wave front set in terms of $\mathcal{G}^{\infty}$-regularity. 

In the case of generalized pseudodifferential operators that satisfies logarithmic slow scale estimates this is a straightforward modification of the regularity results in \cite{GaHo:05} within the Colombeau theory in a Sobolev based context as in \cite{Wunsch:08}. But the situation changes dramatically when working with $\psi$-pseudodifferential operators with symbols of log-type. 

For this reason we will introduce the notion of generalized infinite wave front set describing negligibility of a generalized function at infinite points. For further studies about the infinite wave front set and the more refined semiclassical wave front set we refer to \cite{CoMa:03, Alexandrova:08} and \cite{Martinez:02, EvZw, GuSt:10}.
\subsection{Microlocal behavior at infinity}

Here we introduce a suitable notion of asymptotic negligibility of a generalized function $u \in \mathcal{G}_{2,2} (\mathbb{R}^{n})$ with respect to certain regularities on the phase space. 

To do so we use the following notation which is similar to \cite{CoMa:03}. For a non-zero vector $\xi_0 \in \mathbb{R}^n$ we write for the projection onto the unit sphere $\xi_0/|\xi_0|$. Moreover, given such $\xi_0$ we say that $\Gamma_{\infty \xi_0} \subset \mathbb{R}^n$ is a conic neighborhood of the direction $\xi_0/|\xi_0|$ if $\Gamma_{\infty \xi_0}$ is the intersection of the complement of some open ball centered at the origin with an open cone containing the direction.

Furthermore, for $(x_0,\xi_0)$ in $T^*\mathbb{R}^{n} \setminus 0$ we say that a generalized symbol $p:=(p_\varepsilon)_\varepsilon \in \MphgSym[m,0]{\nu,0}$ is elliptic at $(x_0,\infty \xi_0)$ or also elliptic at infinity at $(x_0,\xi_0)$, if there is an open neighborhood $U$ of $x_0$ and a constant $C > 0$ such that
\begin{equation*}
|p_\varepsilon(x,\xi)| \geq C \langle \xi \rangle^m \qquad \forall (x,\xi) \in U \times \Gamma_{\infty \xi_0} \text{ as } \varepsilon \to 0.
\end{equation*}
We denote by ${\mathop{\rm Ell}}^{\rm i}(p)$ the set of all points $(x_0,\xi_0) \in T^* \mathbb{R}^n \setminus 0$ where $p$ is elliptic at infinity. 

Then the generalized infinite wave front set is defined as follows:
\begin{defn}
For $u \in \mathcal{G}_{2,2}(\mathbb{R}^n)$ we denote by ${\mathop{\rm WF}}^{\rm i}(u) \subset T^* \mathbb{R}^n \setminus 0$ the generalized infinite wave front set of $u$ which is characterized as follows. We say that $(x_0,\xi_0) \notin {\mathop{\rm WF}}^{\rm i}(u)$ at infinity, denoted by $(x_0,\infty \xi_0) \notin {\mathop{\rm WF}}^{\rm i}(u)$, if $\exists \chi := (\chi_\varepsilon)_\varepsilon \in \MphgSym[\infty,0]{0,0}$ elliptic at $(x_0,\infty \xi_0)$ so that
\begin{equation*}
{\mathop{\rm OP}}_{\psi}(\chi) u = 0 \quad \text{ in } \mathcal{G}_{2,2}(\mathbb{R}^n).
\end{equation*}
\end{defn}
\begin{rem}
Let $(u_\varepsilon)_\varepsilon \in \mathcal{M}_{H^{\infty}}$ and suppose that $\exists \chi_{m} := (\chi_{m,\varepsilon})_\varepsilon \in \MphgSym[m,0]{0,0}$ for some $m \in \mathbb{R}$ elliptic at $(x_0,\infty \xi_0)$ with
\begin{equation*}
\forall q \in \mathbb{N}: \quad \lVert {\mathop{\rm OP}}_{\psi,\varepsilon}(\chi_{m,\varepsilon}) u_\varepsilon \rVert_{L^2(\mathbb{R}^{n})} = \mathcal O (\varepsilon^q) \qquad \text{as } \varepsilon \to 0.
\end{equation*}
Using \cite[Proposition 3.4]{Garetto:05b} then the above line is equivalent to ${\mathop{\rm OP}}_{\psi}(\chi_{m}) u  = 0$ in $\mathcal{G}_{2,2}$.
\end{rem}
Furthermore we introduce the notion of an infinite conic support to describe singularities of $\psi$-pseudodifferential operators. Therefore let $p:=(p_\varepsilon)_\varepsilon \in \MphgSym[m,k]{\nu,l}$ be the generalized symbol of $P_\psi$. Then the infinite conic support of $P_\psi$ is the set ${\mathop{\rm conesupp}}^{\rm i}(p) \subset T^* \mathbb{R}^n \setminus 0$ and is defined as the complement (in $T^* \mathbb{R}^n \setminus 0$) of the set of points $(x_0,\xi_0) \in T^* \mathbb{R}^n \setminus 0$ such that there exists an open neighborhood $U$ of $x_0$, a conic open neighborhood $\Gamma_{\xi_0}$ of the direction $\xi_0$ and a constant $K > 0$ such that the following is satisfied
\begin{equation}\label{Iconesupp} 
\forall \alpha, \beta \in \mathbb{N}^{n}, \forall N \in \mathbb{N}: \quad
|\partial^{\alpha}_{\xi} \partial^{\beta}_{x} p_\varepsilon(x,\xi)| \langle \xi \rangle^{-m+N+|\alpha|} = \mathcal{O}(\varepsilon^{N}) \qquad \text{as } \varepsilon \to 0,
\end{equation}
uniformly in $(x,\xi) \in U \times \bigl(\Gamma_{\xi_0} \cap \{ \xi \in \mathbb{R}^n : |\xi| \geq K \} \bigr)$. This is again a condition on the behavior for a generalized symbol at infinity and we therefore say that $(x_0, \xi_0) \notin {\mathop{\rm conesupp}}^{\rm i} (p)$ at infinity and write $(x_0,\infty \xi_0) \notin {\mathop{\rm conesupp}}^{\rm i} (p)$ whenever (\ref{Iconesupp}) is fulfilled.

The idea here is that ${\mathop{\rm conesupp}}^{\rm i}(p)^c$ are the directions on the phase space in which $P_\psi$ annihilates singularities as they are contained in $\mathcal{N}_{{\displaystyle\mathcal{S}}^m}(U \times \Gamma)$. To give a connection to these two notions of wave front sets we state the following theorem. 
\begin{thm}\label{thm:InftyWF}
Given $u \in \mathcal{G}_{2,2}$ and $P_\psi$ a $\psi$-pseudodifferential operator with generalized symbol $p=(p_\varepsilon)_\varepsilon \in \MphgSym[m,k]{\nu,l}$. Then the following statement is valid
\begin{equation*}
{\mathop{\rm WF}}^{\rm i}(P_\psi u) \subset {\mathop{\rm WF}}^{\rm i}(u) \cap {\mathop{\rm conesupp}}^{\rm i}(p)
\end{equation*}
and we say that $\psi$-pseudodifferential operator microlocal at infinity.
\end{thm}
We remark that most of the properties of the infinite wave front set of a generalized function in $\mathcal{G}_{2,2}$ can be derived from the theorem of Calderón-Vaillancourt for the class pseudodifferential operators with symbols in $S^0_{0,0}$.
\begin{proof}
We first show the inclusion relation ${\mathop{\rm WF}}^{\rm i}(P_\psi u) \subset {\mathop{\rm conesupp}}^{\rm i}(p)$. Therefore we let $(x_0,\xi_0) \in T^* \mathbb{R}^n \setminus 0$ such that $(x_0,\infty \xi_0) \notin {\mathop{\rm conesupp}}^{\rm i}(p)$ which in turn implies that $(p_\varepsilon)_\varepsilon$ is in $\mathcal{N}_{{\displaystyle\mathcal{S}}^m} (U \times \Gamma_{\xi_0})$ for some open neighborhood $U$ of $x_0$ and some conic open neighborhood $\Gamma_{\xi_0}$ of $\xi_0$. More precisely, we have that $\exists K_1 > 0$:
\begin{equation*}
\forall \alpha, \beta \in \mathbb{N}^{n} \ \forall N \in \mathbb{N}: \quad |\partial^{\alpha}_{\xi} \partial^{\beta}_{x} p_\varepsilon(x,\xi)| \langle \xi \rangle^{-m+N+|\alpha|} = \mathcal{O}(\varepsilon^{N}) \qquad \text{as } \varepsilon \to 0,
\end{equation*}
uniformly in $(x,\xi) \in U \times \Gamma_{\infty \xi_0}$ with $\Gamma_{\infty \xi_0} = \Gamma_{\xi_0} \cap \{ \xi \in \mathbb{R}^n : |\xi| \geq K_1 \}$.

We will now construct a symbol $\bigl(\chi^{(-m)}_{J,K} \bigr)_\varepsilon \in \MphgSym[-m,0]{0,0}$ elliptic at $(x_0,\infty \xi_0)$ such that
\begin{equation*}
\forall q \in \mathbb{N}: \quad \lVert {\mathop{\rm OP}}_{\psi,\varepsilon}(\chi^{(-m)}_{J,K}) P_{\psi,\varepsilon} u_\varepsilon \rVert_{L^2(\mathbb{R}^n)} = \mathcal O (\varepsilon^q) \qquad \text{as } \varepsilon \to 0.
\end{equation*}
For that reason let $\phi \in \mathcal{C}^{\infty}_{\rm c}(\mathbb{R}^{n})$ such that $\phi(z) = 1$ for $|z| \leq 1/2$ and $\phi(z) = 0$ for $|z| \geq 1$. Now, given $(x_0,\xi_0) \in T^* \mathbb{R}^n \setminus 0$ we define for fixed $J >0$ the function
\begin{equation*}
\lambda_{J}^{(-m)}(x,\xi) := \phi \Bigl(\frac{x-x_0}{J} \Bigr) \phi\left( \Bigl\{ \frac{\xi}{|\xi|} - \frac{\xi_0}{|\xi_0|} \Bigr\} \frac{1}{J}  \right) |\xi|^{-m}   \qquad \quad \text{for } \xi \not= 0.
\end{equation*}
Further for some fixed $K >0$ let
\begin{equation}\label{ConstructSymb}
\chi^{(-m)}_{J,K}(x,\xi) := (1-\phi)\Bigl( \frac{\xi}{2K} \Bigr) \lambda_{J}^{(-m)}(x,\xi).
\end{equation}
Then for $J,K > 0$ fixed $\bigl(\chi^{(-m)}_{J,K}(x,\xi) \bigr)_\varepsilon \in \MphgSym[-m,0]{0,0}$ is elliptic at $(x_0,\infty \xi_0)$ and supported in 
\begin{alignat*}{3}
|x - x_0| &\leq J, & \quad \Bigl| \frac{\xi}{|\xi|} - \frac{\xi_0}{|\xi_0|} \Bigr| &\leq J, & \quad |\xi| &\geq K.
\end{alignat*}
So $\chi^{(-m)}_{J,K}$ is a cut-off in a conic neighborhood of $\xi_0$ and is supported in a cone of directions near $\xi_0$. In particular we have $\chi^{(-m)}_{J,K} \# p_\varepsilon = 0$ for $|\xi| \leq K$. We now choose $J, K>0$ such that
\begin{equation*}
\text{supp} (\chi^{(-m)}_{J,K}) \subset U \times \Gamma_{\infty \xi_0}.  
\end{equation*}
Then by the $L^2$-boundedness theorem of Calderón-Vaillancourt for symbols of class $S^0_{0,0}$, see \cite[Chapter 13, Theorem 1.3]{Taylor:81}, there are constants $j_0, j_1 \in \mathbb{N}$ and $C > 0$ each of which depend on $n$ but independent of $\varepsilon > 0$ such that
\begin{equation*}
\begin{aligned}
\lVert {\mathop{\rm OP}}_{\psi,\varepsilon} (\chi^{(-m)}_{J,K}) P_{\psi,\varepsilon} u_{\varepsilon} \rVert_{L^2} &= \lVert {\mathop{\rm OP}}_{\psi,\varepsilon}(\chi^{(-m)}_{J,K} \#_{_2} p_\varepsilon) u_\varepsilon \rVert_{L^2}\leq \\
&\leq C \sup_{|\alpha| \leq j_0, |\beta| \leq j_1} \lVert \partial^{\alpha}_{\xi} \partial_{x}^{\beta} (\chi^{(-m)}_{J,K} \#_{_2} p_\varepsilon) \rVert_{L^\infty(\mathbb{R}^{2n})} \lVert u_\varepsilon \rVert_{L^2}
\end{aligned}
\end{equation*}
where in the last step we performed the rescaling by means of the asymptotic expansion of the second kind for the composition formula of the pseudodifferential operators expressed by using the notation $\#_{_2}$, see Subsection~\ref{subsec:AE2}. Moreover we used the following estimation
\begin{multline*}
|\partial^{\alpha}_\xi \partial_x^{\beta} (\chi^{(-m)}_{J,K} \#_{_2} p_\varepsilon  (x,\varepsilon \xi))| = |\varepsilon^{|\alpha|} (\partial^{\alpha}_\xi \partial_x^{\beta} \ \chi^{(-m)}_{J,K} \#_{_2} p_\varepsilon)(x,\varepsilon \xi)| \leq\\
\leq \varepsilon^{|\alpha|} \sup_{(x,\xi) \in \mathbb{R}^{2n}} |(\partial^{\alpha}_\xi \partial_x^{\beta} \ \chi^{(-m)}_{J,K} \#_{_2} p_\varepsilon)(x, \varepsilon \xi)| = \varepsilon^{|\alpha|} \lVert \partial^{\alpha}_\xi \partial_x^{\beta} \ \chi^{(-m)}_{J,K} \#_{_2} p_\varepsilon \rVert_{L^\infty (\mathbb{R}^{2n})}.
\end{multline*}
Now using the fact that $(p_\varepsilon)_\varepsilon \in \mathcal{N}_{{\displaystyle\mathcal{S}}^m} (U \times \Gamma_{\xi_0})$ one has for every $q \in \mathbb{N}$ that
\begin{equation*}
\sup_{|\alpha| \leq j_0, |\beta| \leq j_1} \lVert \partial^{\alpha}_{\xi} \partial_{x}^{\beta} (\chi^{(-m)}_{J,K} \#_{_2} p_\varepsilon) \rVert_{L^\infty(\mathbb{R}^{2n})} = \mathcal O (\varepsilon^q) \qquad \text{ as } \varepsilon \to 0
\end{equation*}
by Proposition~\ref{prop:Product} and we deduce that $(x_0,\infty \xi_0) \notin {\mathop{\rm WF}}^{\rm i}(P_\psi u)$ which completes the first part of the proof. 

We continue the proof by showing the following inclusion:

\vspace{10pt}
\noindent \makebox[0.5in][l]{\textbf{Claim.}} \makebox[4.3in][c]{${\mathop{\rm WF}}^{\rm i} (P_\psi u ) \subset {\mathop{\rm WF}}^{\rm i}(u).$}

\noindent \emph{Proof of Claim.}
We take $(x_0,\infty \xi_0) \notin {\mathop{\rm WF}}^{\rm i}(u)$ and let $U$ be some open neighborhood of $x_0$ and $\Gamma_{\xi_0}$ a conic neighborhood of $\xi_0$. Then, for some $j \in \mathbb{R}$ there exists $\chi_{j} := (\chi_{j,\varepsilon})_\varepsilon \in \MphgSym[j,0]{0,0}$ elliptic at $(x_0,\infty \xi_0)$ such that for some constant $K_1 > 0$ we have
\begin{equation}\label{InftyEll}
\exists C > 0 \ \exists \eta \in (0,1]: \qquad |\chi_{j,\varepsilon}(x,\xi)| \geq C \langle \xi \rangle^j \qquad \forall \varepsilon \in (0,\eta]
\end{equation}
uniformly on $U \times \Gamma_{\infty {\xi_0}}$ with $\Gamma_{\infty {\xi_0}}=\Gamma_{\xi_0} \cap \{(x,\xi)  :  |\xi| \geq K_1 \}$ and
\begin{equation*}%\label{NotInFSu}
\forall q \in \mathbb{N}: \quad \lVert {\mathop{\rm OP}}_{\psi,\varepsilon}(\chi_{j,\varepsilon}) u_\varepsilon \rVert_{L^2}  =  \mathcal O (\varepsilon^q) \quad \text{as } \varepsilon \to 0 .
\end{equation*}
Then the assumptions in Lemma~\ref{lem:psiparametrix} from the next section are satisfied on $U \times \Gamma_{\xi_0}$ and there is a $\widetilde{\chi}_{-j}:= (\widetilde{\chi}_{-j,\varepsilon})_\varepsilon \in \MphgSym[-j,0]{0,0}$ such that
\begin{equation*}
{\mathop{\rm OP}}_{\psi} (\widetilde{\chi}_{-j}) {\mathop{\rm OP}}_\psi (\chi_{j}) = I_{\psi} + R_{\psi} \qquad \text{ on } U \times \Gamma_{\xi_0}
\end{equation*}
where the generalized symbol $(r_\varepsilon)_\varepsilon$ of $R_\psi$ is in $\MSym[-N,N]{0,0}$ on $U \times \Gamma_{\xi_0}$ for every $N \in \mathbb{N}$.

As in (\ref{ConstructSymb}) one constructs a generalized symbol $(\kappa_{j-m})_\varepsilon \in \MphgSym[j-m,0]{0,0}$ elliptic at $(x_0, \infty \xi_0)$ such that the $\text{supp}(\kappa_{j-m})$ is contained in the set where (\ref{InftyEll}) is valid. We then write
\begin{equation*}
{\mathop{\rm OP}}_{\psi,\varepsilon} (\kappa_{j-m} \#_{_2} p_\varepsilon) = {\mathop{\rm OP}}_{\psi,\varepsilon} (\kappa_{j-m} \#_{_2} p_\varepsilon \#_{_2} \widetilde{\chi}_{-j,\varepsilon}) {\mathop{\rm OP}}_{\psi,\varepsilon} (\chi_{j,\varepsilon}) - {\mathop{\rm OP}}_{\psi,\varepsilon} (\kappa_{j-m} \#_{_2} p_\varepsilon \#_{_2} r_\varepsilon).
\end{equation*}
Concerning the first term on the right hand side we rescale and obtain by the Calderón-Vaillancourt theorem that there are constants $j_0,j_1 \in \mathbb{N} $ and $C >0$ such that
\begin{multline*}
\lVert {\mathop{\rm OP}}_{\psi,\varepsilon} (\kappa_{j-m} \#_{_2} p_\varepsilon \#_{_2} \widetilde{\chi}_{-j,\varepsilon}) {\mathop{\rm OP}}_{\psi,\varepsilon} (\chi_{j,\varepsilon}) u_\varepsilon \rVert_{L^2} \leq \\
 \leq C \sup_{|\alpha| \leq j_0, |\beta| \leq j_1} \lVert \partial_{\xi}^{\alpha} \partial_x^{\beta} \kappa_{j-m} \#_{_2} p_\varepsilon \#_{_2} \widetilde{\chi}_{-j,\varepsilon}) \lVert_{L^{\infty}(\mathbb{R}^{2n})} \lVert {\mathop{\rm OP}}_{\psi,\varepsilon} (\chi_{j,\varepsilon}) u_\varepsilon \rVert_{L^2}
\end{multline*}
and the latter expression is $\mathcal{O}(\varepsilon^q)$ as $\varepsilon \to 0$ for every $q \in \mathbb{N}$ by assumption.
Similarly we obtain for every $q \in \mathbb{N}$:
\begin{eqnarray*}
\lVert {\mathop{\rm OP}}_{\psi,\varepsilon} (\kappa_{j-m} \#_{_2} p_\varepsilon \#_{_2} r_\varepsilon) u_\varepsilon \rVert_{L^2} = \mathcal{O} (\varepsilon^q) \quad \text{as } \varepsilon \to 0
\end{eqnarray*}
since $(r_\varepsilon)_\varepsilon \in \mathcal{N}_{{\displaystyle\mathcal{S}}^0}$ on the support of $\kappa_{j-m}$. Therefore $(x_0, \infty \xi_0) \notin {\mathop{\rm WF}}^{\rm i} (P_\psi u)$.
\end{proof}
As a next result we reformulate the infinite wave front set in terms of the Fourier transform of a localized function.
\begin{prop}
Let $u \in \mathcal{G}_{\rm c}(\mathbb{R}^n)$. If
$(x_0,\infty \xi_0) \notin {\mathop{\rm WF}}^{\rm i} (u)$ then there exists $\gamma \in \mathcal{C}^{\infty}_{\rm c}(\mathbb{R}^n)$ with $\gamma(x_0) \neq 0$ and a conic neighborhood $\Gamma$ of $\xi_0$ such that $\forall N,q \in \mathbb{N}$, $\exists C >0$ satisfying
\begin{equation}\label{EquivIWF}
\lvert \mathcal{F}(\gamma u_\varepsilon) (\xi) \rvert \leq C \varepsilon^q \langle \xi \rangle^{-N} \quad \text{ as } \varepsilon \to 0
\end{equation}
for all $\xi \in \Gamma$ with $|\xi| \geq K/\varepsilon$ for some $K > 0$ independent of $\varepsilon$.
\end{prop}
Note that our notion of regularity is derived from the infinite wave front set and therefore differs from \cite[Section 6]{HoMoe:04} and \cite[Sections 2,3]{GaHo:05}. It is also different to \cite[Section 3.2.2]{NPS:98} if one replaces in the Definitions 3.13 and 3.14 the condition of rapid decrease in each derivative by simple rapid decrease.
\begin{proof}
The following proof is similar to \cite[Proposition 7.4]{GiSj:94}. First note that $\mathcal{G}_{\rm c} \subset \mathcal{G}_{2,2}$. Since $(x_0,\infty \xi_0) \notin {\mathop{\rm WF}}^{\rm i}(u)$ we can find a $\chi:=\bigl(\chi_{\varepsilon}\bigr)_\varepsilon \in \MphgSym[-m,0]{0,0}$ elliptic at $(x_0,\infty \xi_0)$ such that ${\mathop{\rm OP}}_\psi (\chi) u = 0$ in $\mathcal{G}_{2,2}$ for some $m \in \mathbb{R}$. Furthermore we let $\gamma \in \mathcal{C}^{\infty}_{\rm c} (\mathbb{R}^n)$ with $\gamma (x_0) \neq 0$. Then there exists a symbol $\phi \in S^0 (\mathbb{R}^n), \text{supp}(\phi) \subset \Gamma_{\infty \xi_0}= \Gamma_{\xi_0} \cap \{ \xi: |\xi| \geq K \}$ for some $K>0$ and $\phi(t \xi) = 1$ for $t \geq 1, \xi \in \Gamma_{\infty \xi_0}$ so that 
\begin{equation*}
\phi_\psi (D) \gamma (x) = A_\psi {\mathop{\rm OP}}_\psi(\chi) + R_\psi
\end{equation*}
where $A_\psi$ and $R_\psi$ is a $\psi$-pseudodifferential operator with generalized symbol in $\MSym[m,0]{0,0}$ and $\mathcal{N}_{{\displaystyle\mathcal{S}}^0}$, respectively. Then by assumption we deduce that $\phi_\psi (D) \gamma u = 0$ in $\mathcal{G}_{2,2}$. In particular we deduce that $\phi_\psi (D) \gamma u = 0$ in $\mathcal{G}_{\mathscr{S}}$ since $\mathcal{G}_{\mathscr{S}} \hookrightarrow \mathcal{G}_{2,2}$, see \cite[Proposition 3.5]{Garetto:05b}. Using the fact that the Fourier transform is an isomorphism on $\mathcal{G}_{\mathscr{S}}$ we obtain $\xi \mapsto (\phi(\varepsilon \xi) \widehat{\gamma u}_\varepsilon (\xi))_\varepsilon$ is in $\mathcal{N}_{\mathscr{S}}$ under consideration of the scaling. Hence we also have
\begin{equation*}
\forall N \in \mathbb{N} \ \forall q \in \mathbb{N}: \quad \lVert \langle \xi / \varepsilon \rangle^N \phi (\xi) \widehat{\gamma u}_\varepsilon (\xi / \varepsilon) \rVert_{L^{\infty}} = \mathcal{O} (\varepsilon^q)
\end{equation*}
showing (\ref{EquivIWF}).
\end{proof}
We also study the behavior of the infinite wave front set of a function under the action of $\psi$-pseudodifferential operators that are elliptic at infinity.
\begin{thm}
Let $P_\psi$ be a $\psi$-pseudodifferential operator with symbol $p:=(p_\varepsilon)_\varepsilon$ in $\MphgSym[m,0]{\nu,0}$ that is elliptic at $(x_0, \infty \xi_0) \in T^* \mathbb{R}^n \setminus 0$. Then the following holds:
\begin{equation*}
{\mathop{\rm WF}}^{\rm i} (u) \subset {\mathop{\rm WF}}^{\rm i}(P_\psi u) \cup {\mathop{\rm Ell}}^{\rm i}(p)^{\rm c}.
\end{equation*}
\end{thm}
\begin{proof}
Suppose $(x_0, \infty \xi_0)$ is not contained in the right hand side of the claimed inclusion relation. Since $(x_0, \infty \xi_0) \in {\mathop{\rm Ell}}^{\rm i}(p)$ the symbol $(p_\varepsilon)_\varepsilon$ is elliptic at $(x_0,\infty \xi_0)$ and therefore there is an open neighborhood $U$ of $x_0$ and conic neighborhood $\Gamma_{\xi_0}$ containing $\xi_0$ such that
\begin{equation*}
|p_\varepsilon(x,\xi)| \geq C \langle \xi \rangle^m \qquad \forall (x,\xi) \in U \times \Gamma_{\infty \xi_0} \text{ as } \varepsilon \to 0.
\end{equation*}
By Lemma~\ref{lem:psiparametrix} from the next section there exists an approximative inverse $Q_\psi$ so that
\begin{equation}\label{IMLInverse}
I_\psi = Q_\psi P_\psi  + R_\psi  \qquad \text{on } U \times \Gamma_{\xi_0}
\end{equation}
and $(x_0,\infty \xi_0) \notin {\mathop{\rm conesupp}}^{\rm i} (r)$ and $r:= (r_\varepsilon)_\varepsilon$ is the generalized symbol of $R_\psi$. Hence $(x_0,\infty \xi_0) \notin {\mathop{\rm WF}}^{\rm i}(R_\psi u)$ by Theorem~\ref{thm:InftyWF}.

Furthermore, since also $(x_0,\infty \xi_0) \notin {\mathop{\rm WF}}^{\rm i} (P_\psi u)$ we obtain that $(x_0,\infty \xi_0) \notin {\mathop{\rm WF}}^{\rm i}(Q_\psi P_\psi u) \subset {\mathop{\rm WF}}^{\rm i}(P_\psi u)$ and by (\ref{IMLInverse}) we deduce that $(x_0,\infty \xi_0) \notin {\mathop{\rm WF}}^{\rm i}(u)$. 
\end{proof}
\subsection{Microlocal factorization}
In this section we use the notion of microlocal behavior at infinite points of a given generalized function in $\mathcal{G}_{2,2}(\mathbb{R}^n)$ to give a microlocal interpretation of Theorem~\ref{thm:Fact}. To do so let $I$ be a conic subset of $\mathbb{R}^n \times (\mathbb{R}^n \setminus 0)$. We then say that two generalized functions $u,v \in \mathcal{G}_{2,2}$ are microlocally equivalent at infinity on $I$ if and only if $\exists (\chi_\varepsilon)_\varepsilon \in \MphgSym[\infty,0]{0,0}$ elliptic at $(x_0,\infty \xi_0)$ for every $(x_0,\xi_0) \in I$ such that $\chi_\psi (u-v) = 0$ in $\mathcal{G}_{2,2}$. %$\FS(u-v) \cap I = \emptyset$.

Similar to \cite{Stolk:04} we introduce a subset of the phase space associated to $I'_{\theta_1}$ which is given by
\begin{equation}\label{IMLRegion}
I_{\theta_1}:= \{ (t,y,\tau,\eta) \ | \ (y,\tau,\xi) \in I'_{\theta_1}, \ |\zeta| \leq \sqrt{c_1} |\tau| \}
\end{equation} 
and $c_1>0$ is the upper bound of the H\"{o}lder continuous coefficient $c(y)$. Note that in the following $I_{\theta_1}$ will serve as the adequate space-frequency domain on which we establish the microlocal factorization theorem at infinite points.

As already mentioned in the preceding section $L_{j,\psi}= (\partial_z + A_j(y,D_t,D_x))_\psi$, $j=1,2$ is not a $\psi$-pseudodifferential operator on $\mathbb{R}^{n+1}$. To overcome this we are going to introduce a microlocal cut-off $g_\psi(D_t,D_y)$ for $I_{\theta_1}$ such that
\begin{equation}\label{EquivIML}
g_\psi L_{j,\psi} u = L_{j,\psi} u \qquad \text{microlocally at infinity on } I_{\theta_1}
\end{equation}
where $u \in \mathcal{G}_{2,2}$ and $g_\psi L_{j,\psi}$ is a $\psi$-pseudodifferential operator with generalized symbol in $\MSym[1,0]{1,0}$. Here the microlocal cut-off $g_\psi$ is constructed in the following way.

First let $K_0, K_1$ be some fixed constants so that $0<K_0<K_1<\infty$. Then $g \in \mathcal{C}^{\infty} (\mathbb{R}^{n+1})$, $0 \leq g \leq 1$ is defined by
\begin{equation*}
g(\tau,\eta) := \begin{cases}
0, & |\zeta| \geq 3 \sqrt{c_1} |(\tau,\xi)| \text{ or } |(\tau,\xi)| \leq K_0\\
\tilde{\sigma} (\tau,\xi), & |\zeta| \leq 2 \sqrt{c_1} |(\tau,\xi)|\\
\frac{\tilde{\sigma} (\tau,\xi)}{1 + e^{\frac{1}{|\zeta|/|(\tau,\xi)| - 2\sqrt{c_1}} - \frac{1}{3\sqrt{c_1} - |\zeta|/|(\tau,\xi)|}}}, & 2\sqrt{c_1}|(\tau,\xi)| < |\zeta| < 3 \sqrt{c_1}|(\tau,\xi)|
\end{cases}
\end{equation*}
and the function $\tilde{\sigma}$ is a cut-off near the origin given by
\begin{equation*}
\tilde{\sigma}(\tau,\xi) := \begin{cases}
0, & |(\tau,\xi)| \leq K_0\\
1, & |(\tau,\xi)| \geq K_1\\
\frac{1}{{1 + e^{\frac{K_1-K_0}{|(\tau,\xi)| - K_0} - \frac{K_1-K_0}{K_1 - |(\tau,\xi)|}}}}, & K_0 < |(\tau,\xi)| < K_1.
\end{cases}
\end{equation*}
Then $g \in S^0(\mathbb{R}^{n+1})$ and have the following picture:\vspace{10pt}

% 1. Definition of characteristic points
\figinit{0.4cm}
\figpt 1: (0,0) \figpt 2: (-0.5,0) \figpt 3: (15,0) \figpt 4: (0,-0.5)
\figpt 5: (0,7) \figpt 6: (0,4) \figpt 7: (0,2) \figpt 8: (4.5,4)
\figpt 9: (15,4) \figpt 10: (6,4) \figvectP 11 [1,8] \figpt 12: (15,7)   
\figvectP 13 [1,12] \figpt 14: (2.7,5) \figvectP 15 [6,9] \figpt 16: (15,2)   
\figvectP 17 [7,16] \figptinterlines 18:[1,13;7,17] \figvectP 19 [5,12]
\figptinterlines 20:[1,11;5,19] \figpt 21: (10,2) \figptinterlines 22: [1,13;6,15] 
\figpt 23: (0,8) \figpt 24: (16,0) \figptinterlines 25:[22,13;5,19]
%
% 2. Creation of the graphical file
\psbeginfig{}
\psset(width=1,fillmode=yes,color=0.85)\psline[5,6,8,20,5]
\psset(width=1,fillmode=yes,color=0.95)\psline[1,7,18,12,3,1]
\psreset{first} \psarrow[2,24] \psarrow[4,23] \psline[6,22]
\psline[7,18] \psline[18,22] \psline[8,20] \psline[22,25]
\pssetdash{8} \psset (dash=6, width=0.2) \psline[1,8] \psline[1,18]
\psendfig
%
% 3. Writing text on the figure
\figvisu{\figBoxA}{}
{
\figwritew 7:\footnotesize{$K_0$}(2pt) \figwritew 6:\footnotesize{$K_1$}(2pt)
\figwritese 20:\scriptsize{$2\sqrt{c_1}|(\tau,\xi)| = |\zeta|$}(1pt)
\figwritee 22:\scriptsize{$\hspace{10pt} 3\sqrt{c_1}|(\tau,\xi)| = |\zeta|$}(1pt)
\figwritesw 23:\footnotesize{$|(\tau,\xi)|$}(6pt) \figwritee 24:\footnotesize{$|\zeta|$}(2pt)
\figwriten 21:\footnotesize{$g = 0$}(0pt) \figwriten 14:\footnotesize{$g = 1$}(5pt)
}
\centerline{\box\figBoxA}

\vspace{10pt}
\noindent Moreover, $g_\psi L_{j,\psi}$ acts as a $\psi$-pseudodifferential operator in $(t,x,z)$ with generalized symbol in $\MSym[1,0]{1,0}$ which can be shown by an adaption of \cite[Theorem 18.1.35]{Hoermander:3}. 

Furthermore (\ref{EquivIML}) is satisfied, since
${\mathop{\rm conesupp}}^{\rm i}(g - 1) \cap I_{\theta_1} = \emptyset$. In the same manner one shows that $g_\psi L_{1,\psi} L_{2,\psi} = L_{1,\psi} L_{2,\psi}$ microlocally at infinity on $I_{\theta_1}$.

Summarizing the observations from above we obtain the following main theorem:
\begin{thm}\label{thm:IMLFact}
Let $L_\psi$ and $I_{\theta_1}$ be as in (\ref{GovOp}) and (\ref{IMLRegion}). Then the operator $L_\psi$ can be factorized into a product of two first-order $\psi$-pseudodifferential operators as follows
\begin{equation*}
L_\psi = L_{1,\psi} L_{2,\psi} \quad \text{microlocally at infinity on} \ I_{\theta_1}
\end{equation*}
where $L_{j,\psi} = \bigl(\partial_z + A_j\bigr)_\psi$ and $A_{j,\psi}$ is the $\psi$-pseudodifferential operator is as in Theorem~\ref{thm:Fact}.
\end{thm}
\section{Microlocal Diagonalization for \texorpdfstring{$L_\psi$}{Lpsi}}\label{sec:Diag}
The main issue in this section is to diagonalize of the microlocal equation $L_\psi U = F$ using the refined factorization theorem~\ref{thm:IMLFact}. In detail, we will rewrite the equation
\begin{equation*}
L_{\psi} U = F \qquad \text{microlocally at infinity on } I_{\theta_1}
\end{equation*}
into an equivalent system of the form
\begin{equation*}
\big( \partial_z - i B_{\pm}(x,z,D_t,D_x) \big)_\psi u_{\pm}   =   f_{\pm}  \quad \text{microlocally at infinity on } I_{\theta_1}.
\end{equation*}
To show this we will discuss a different approach to the one given by Stolk in \cite{Stolk:04} as it turns out that the factorization theorem already contains all the ingredients for the diagonalization.
Let us begin with the following lemma which was already used in the previous section.
\begin{lem}\label{lem:psiparametrix}
Let $m\in \mathbb{R}$ and $P_\psi$ be a polyhomogeneous $\psi$-pseudodifferential operator whose generalized symbol is given by $(p_\varepsilon)_\varepsilon \dot{\sim} \sum_{j=0}^{\infty}( \varepsilon^j \overline{p}_{m-j,\varepsilon})_\varepsilon$ in $\MSym[m,0]{\nu,0}$ on $I'_{\theta_1}$ for some $(\overline{p}_{m-j,\varepsilon})_\varepsilon \in \MhgSym[m-j,0]{\nu,\nu j}$ on $I'_{\theta_1}$, $j \geq 0$. Furthermore suppose that the principal symbol $(\overline{p}_{m,\varepsilon})_\varepsilon$ satisfies an estimate of the form
\begin{equation}\label{StrongEll}
\exists C >0 \ \exists \eta \in (0,1] : \quad |\overline{p}_{m,\varepsilon}(y,\tau,\xi)| \geq C |(\tau,\xi)|^{m} \quad \text{on } I'_{\theta_1}, \ \varepsilon \in (0,\eta].
\end{equation}
Then there exists a $\psi$-pseudodifferential operator $Q_\psi$ with generalized symbol in $\MphgSym[-m,0]{\nu,0}$ such that the following is valid
\begin{equation}\label{psiparametrix}
Q_\psi P_\psi = I_\psi + R_\psi \qquad \text{ on } I'_{\theta_1},
\end{equation}
where $I$ is the identity and the generalized symbol of $R_\psi$ is in $\mathcal{N}_{{\displaystyle\mathcal{S}}^0}$ on $I'_{\theta_1}$. More precisely, the polyhomogeneous generalized symbol $(q_\varepsilon)_\varepsilon$ of $Q_\psi$ is written in terms of its asymptotic expansion as
\begin{equation}\label{AE:parametrix}
(q_\varepsilon)_\varepsilon \dot{\sim} \sum_{k \geq 0} (\varepsilon^k \overline{q}_{-m-k,\varepsilon})_\varepsilon \qquad \text{ in } \MSym[-m,0]{\nu,0} \text{ on } I'_{\theta_1}
\end{equation}
for some $(\overline{q}_{-m-k,\varepsilon})_\varepsilon \in \MhgSym[-m-k,0]{\nu,\nu k} \text{ on } I'_{\theta_1},\ k \geq 0$.
\end{lem}
\begin{proof} The proof follows the classical arguments given in \cite[Chapter 2]{Kumano-go:81} for $\varepsilon$-dependent symbols. We construct the generalized symbol of $Q_\psi$ by means of its asymptotic expansion. Therefore we will recursively define symbols $(\overline{q}_{-m-k,\varepsilon})_\varepsilon \in \MhgSym[-m-k,0]{\nu,\nu k}$ on $I'_{\theta_1}$, $k \geq 0$, so that the symbol of $Q_\psi$ is given by (\ref{AE:parametrix}) and satisfies (\ref{psiparametrix}).

Because of (\ref{StrongEll}) we may define for $(y,\tau,\xi) \in I'_{\theta_1}$ and some $\eta \in (0,1]$
\begin{equation*}
\overline{q}_{-m,\varepsilon}(y,\tau,\xi) :=\begin{cases}
\overline{p}_{m,\varepsilon}(y,\tau,\xi)^{-1}, & \varepsilon \in (0,\eta]\\
0, & \text{else.}
\end{cases}
\end{equation*}
Concerning the asymptotic behavior of the derivatives of $\overline{q}_{-m,\varepsilon}$ we use the Leibniz rule and obtain that $\exists C > 0$ such that for every $\varepsilon > 0$ sufficiently small 
\begin{equation*}
\begin{split}
|\partial^{\alpha}_{(\tau,\xi)} \partial^{\beta}_{y} \overline{q}_{-m,\varepsilon}| &= \bigl| \sum_{\substack{\alpha_1 + \ldots + \alpha_\mu = \alpha \\ \beta_1 + \ldots + \beta_\mu = \beta} } \left( \partial^{\alpha_1}_{(\tau,\xi)} \partial^{\beta_1}_{y} \overline{p}_{m,\varepsilon} \right) \ldots \left( \partial^{\alpha_\mu}_{(\tau,\xi)} \partial^{\beta_\mu}_{y} \overline{p}_{m,\varepsilon} \right) \frac{1}{\overline{p}_{m,\varepsilon}^{1+\mu}} \bigr| \leq\\
& \leq C \hspace{1pt} \omega_\varepsilon^{\nu |\beta|} |( \tau, \xi)|^{-m-|\alpha|} \quad \text{on } I'_{\theta_1}
\end{split}
\end{equation*}
from which we deduce that $(\overline{q}_{-m,\varepsilon})_\varepsilon \in \MhgSym[-m,0]{\nu,0}$ on $I'_{\theta_1}$. We now proceed by induction. Therefore we define $(\overline{q}_{-m-k,\varepsilon})_\varepsilon$ on $I'_{\theta_1}$ by
\begin{equation}\label{ConstParametrix}
\overline{q}_{-m-k,\varepsilon} := - \Bigl\{ \sum_{\substack{|\gamma| + j + l =k \\ l <k}} \frac{1}{\gamma!} D^{\gamma}_\xi \overline{q}_{-m-l,\varepsilon} \partial^{\gamma}_x \overline{p}_{m-j,\varepsilon} \Bigr\} \frac{1}{\overline{p}_{m,\varepsilon}} \quad \quad  k \geq 1 \text{ on } I'_{\theta_1}.
\end{equation}
Assuming that $(\overline{q}_{-m-N,\varepsilon})_\varepsilon \in \MhgSym[-m-N,0]{\nu,\nu N}(I'_{\theta_1})$ for $N < k$ then from the construction given in (\ref{ConstParametrix}) one easily verifies that $(\overline{q}_{-m-k,\varepsilon})_\varepsilon \in \MhgSym[-m-k,0]{\nu,\nu k}(I'_{\theta_1})$. 

Moreover by Lemma~\ref{lem:AE1} we get the existence of a polyhomogeneous generalized symbol $(q_\varepsilon)_\varepsilon \in \MphgSym[-m,0]{\nu,0}$ having the following asymptotic expansion
\begin{equation*}
(q_\varepsilon)_\varepsilon \dot{\sim} \sum_{k \ge 0} (\varepsilon^k \overline{q}_{-m-k,\varepsilon})_\varepsilon \qquad \text{in } \MSym[-m,0]{\nu,0} \text{ on } I_{\theta_1}'.
\end{equation*}
Let $Q_\psi$ be the $\psi$-pseudodifferential operator with generalized symbol $(q_\varepsilon)_\varepsilon$. Then (\ref{AE:parametrix}) is satisfied and it remains to show equation (\ref{psiparametrix}). Concerning the asymptotic expansion of $Q_\psi P_\psi$ one has for every $N \geq 1$ 
\begin{equation}\label{Check}
\begin{split}
\sum_{|\gamma| < N} \frac{\varepsilon^{|\gamma|}}{\gamma!} D^{\gamma}_\xi q_\varepsilon \partial^{\gamma}_x p_\varepsilon = & \ \overline{q}_{-m,\varepsilon} \overline{p}_{m,\varepsilon} + \sum_{k=1}^{N-1} \sum_{\substack{|\gamma| + l +j = k}} \frac{\varepsilon^{k}}{\gamma!} D^{\gamma}_\xi  \overline{q}_{-m-l,\varepsilon} \partial^{\gamma}_x  \overline{p}_{m-j,\varepsilon} +\\
& \ + \ \sum_{k \geq N}\sum_{\substack{|\gamma| +l+ j =k \\ |\gamma| < N}} \frac{\varepsilon^{k}}{\gamma!} D^{\gamma}_\xi \overline{q}_{-m-l,\varepsilon} \partial^{\gamma}_x \overline{p}_{m-j,\varepsilon} % \qquad \text{on } J'_{\theta_1,\varepsilon},
\end{split}
\end{equation}
on $V_K :=I'_{\theta_1} \cap (\mathbb{R}^n \times \{ (\tau,\xi) : |(\tau,\xi)| \geq K \})$ for some $K > 0$ independent of $\varepsilon$.
Here the second term on the right hand side vanishes by (\ref{ConstParametrix}) on $V_K$. Furthermore the last expression of (\ref{Check}) is in $\mathcal{N}_{{\displaystyle\mathcal{S}}^0} (I'_{\theta_1})$ and $\overline{q}_{-m,\varepsilon} \overline{p}_{m,\varepsilon} = 1$ on $V_K$ which establishes the statement made in (\ref{psiparametrix}). 
\end{proof}
Similarly one can construct a $\psi$-pseudodifferential operator $\widetilde{Q}_\psi$ with generalized symbol in $\MphgSym[-m,0]{\nu,0}$ having a representation of the form (\ref{AE:parametrix}) and satisfies 
\begin{equation*}
P_\psi \widetilde{Q}_\psi = I_\psi + \widetilde{R}_\psi \qquad \text{ on } I'_{\theta_1}
\end{equation*}
where the generalized symbol of $\widetilde{R}_\psi$ is in $\mathcal{N}_{{\displaystyle\mathcal{S}}^0}$ on $I'_{\theta_1}$. Furthermore $Q_\psi$ is related to $\widetilde{Q}_\psi$ in the following way
\begin{equation*}
Q_\psi = Q_\psi (P_\psi \widetilde{Q}_\psi) = (Q_\psi P_\psi) \widetilde{Q}_\psi = \widetilde{Q}_\psi \quad \mod{} {\mathop{\rm OP}}_\psi \mathcal{N}_{{\displaystyle\mathcal{S}}^m} \text{ on } I'_{\theta_1}.
\end{equation*}
\subsection*{Diagonalization}
To get an idea we start rewriting the inhomogeneous equation $L_{\psi}U = F$ into a first-order system with respect to the parameter $z$: 
\begin{equation} \label{EquivSystem}
\bigg[ \left(\text{Id} \, \partial_z\right)_{\psi} - \begin{pmatrix} 0 & 1 \\ -A & 0  \end{pmatrix}_{\substack{\! \psi \\ \phantom{} }} \bigg] \begin{pmatrix} U \\ \left(\partial_z\right)_{\psi} U \end{pmatrix} = \begin{pmatrix} 0 \\ F  \end{pmatrix}
\end{equation}
with $U,F \in \mathcal{G}_{2,2}(\mathbb{R}^{n+1})$, $\text{Id}$ the 2$\times$2 identity matrix and $A_\psi$ the $\psi$-pseudodifferential operator with generalized symbol $(a_\varepsilon)_\varepsilon$ as in (\ref{SymbolA}). For brevity we will drop the identity matrix $\text{Id}$ in the equations from now on. Hereafter we are going to reduce the operator in (\ref{EquivSystem}) to diagonal form. To make this notion rigorous we make the following arrangements.

First let
\begin{align*}
Q_\psi = \begin{pmatrix} Q^{11} & Q^{12} \\ Q^{21} & Q^{22} \end{pmatrix}_{\substack{\! \psi \\ \phantom{} }}, & \qquad P_\psi = \begin{pmatrix} P^{11} & P^{12} \\ P^{21} & P^{22} \end{pmatrix}_{\substack{\! \psi \\ \phantom{} }}
\end{align*}
be 2$\times$2 matrices of $\psi$-pseudodifferential operators whose entries satisfy the requirements of Lemma~\ref{lem:psiparametrix}.
Also we choose the top order of the symbols of $Q^{11}_\psi$, $Q^{21}_\psi$ equal to $1-m$ and those of $Q^{12}_\psi$, $Q^{22}_\psi$ equal to $-m$ for some fixed $m \in \mathbb{R}$.

Using Lemma~\ref{lem:psiparametrix} we choose $P_\psi$ to be the approximative inverse of $Q_\psi$ on $I'_{\theta_1}$ in the following sense: with $P^{11}_\psi$, $P^{12}_\psi$ being operators of order $m-1$ and $P^{21}_\psi$, $P^{22}_\psi$ of order $m$ and one has $P_\psi Q_\psi = \text{Id}_\psi + E_\psi$ on $I'_{\theta_1}$ for some 2$\times$2 $\psi$-pseudodifferential operator matrix $E_\psi$ of the form
\begin{equation*}
E_\psi = \begin{pmatrix} E^{11} & E^{12} \\ E^{21} & E^{22} \end{pmatrix}_{\substack{\! \psi \\ \phantom{} }}
\end{equation*}
and $E^{11}_\psi$, $E^{22}_\psi \in {\mathop{\rm OP}}_\psi \mathcal{N}_{{\displaystyle\mathcal{S}}^0}$, $E^{12}_\psi \in {\mathop{\rm OP}}_\psi \mathcal{N}_{{\displaystyle\mathcal{S}}^{-1}}$ and $E^{21}_\psi \in {\mathop{\rm OP}}_\psi \mathcal{N}_{{\displaystyle\mathcal{S}}^1}$ on $I'_{\theta_1}$.

Then by (\ref{EquivIML}) and (\ref{EquivSystem}) the equation
\begin{equation*}
L_\psi U = F \qquad \text{ microlocally at infinity on } I_{\theta_1} 
\end{equation*}
holds if and only if 
\begin{equation*}
g_\psi Q_\psi \begin{pmatrix} \partial_z & -1  \\ A & \partial_z  \end{pmatrix}_{\substack{\! \psi \\ \phantom{} }}  P_\psi Q_\psi \begin{pmatrix} U \\ \left(\partial_z\right)_{\psi} U \end{pmatrix} = g_\psi Q_\psi \begin{pmatrix} 0 \\ F \end{pmatrix} \ \quad \text{microlocally at infinity on } I_{\theta_1}
\end{equation*}
where $g_\psi$ is the microlocal cut-off function from the previous section.

Furthermore we will use the following notation: for $l=1,2$ and $k \in \mathbb{N}$ we denote by $\OPNerror{k}{l}$ operators of the form
\begin{equation*}
\sum_{k \leq j \leq k+l} R_{2-j,\psi}(y,D_t,D_x) (\partial_z^{j-k})_\psi
\end{equation*}
and the generalized symbol of $R_{2-j,\psi} = R_{2-j,\psi}(y,D_t,D_x)$ is in $\mathcal{N}_{{\displaystyle\mathcal{S}}^{2-j}}$ for $k \leq j \leq k+l$. 

In the following we assume that $R_\psi$ is a $\psi$-pseudodifferential operator valued 2$\times$2 error matrix with entries in $\OPNerror{1}{1}$ on $I'_{\theta_1}$. Also, we let $B_{\pm,\psi}= B_{\pm,\psi}(y,D_t,D_x) \in {\mathop{\rm OP}}_\psi \MphgSym[1,0]{1,0}$ on $I'_{\theta_1}$.

In order to obtain a diagonalization for the operator in (\ref{EquivSystem}) we will search for operators $P_\psi, Q_\psi, R_\psi$ and $B_{\pm,\psi}$ as above such that the following ansatz is valid in the region $I'_{\theta_1}$:
\begin{equation}\label{Ansatz}
Q_{\psi} \bigg[ \left(\partial_z\right)_\psi - \begin{pmatrix} 0 & 1  \\ -A & 0  \end{pmatrix}_{\substack{\! \psi \\ \phantom{} }} \bigg] P_\psi = \begin{pmatrix} \partial_z - i B_{+} & 0 \\ 0 & \partial_z - i B_{-} \end{pmatrix}_{\substack{\! \psi \\ \phantom{} }} + R_\psi.
\end{equation}
To show the existence of these operators we will construct them explicitly. Therefore we presume that (\ref{Ansatz}) can be solved for some $P_\psi, Q_\psi, R_\psi$ and $B_{\pm,\psi}$ with the above conditions. 

Multiplying equation (\ref{Ansatz}) with $Q_\psi \left( \begin{smallmatrix} 1 \\ \partial_z  \end{smallmatrix}\right)_{\substack{\! \psi \\ \phantom{} }}$ from the right we obtain for the left hand side an expression of the form
\begin{equation}\label{ExpandAnsatz1}
Q_\psi \begin{pmatrix} \partial_z & -1 \\ A & \partial_z \end{pmatrix}_{\substack{\! \psi \\ \phantom{} }} P_\psi Q_\psi \begin{pmatrix} 1 \\ \partial_z \end{pmatrix}_{\substack{\! \psi \\ \phantom{} }} = \begin{pmatrix} Q^{12}_\psi ( \partial^2_z + A )^{}_{\psi}  + S^{(1,+)}_{\psi}  \\ Q^{22}_\psi (\partial^2_z + A )^{}_{\psi} + S^{(1,-)}_{\psi}  \end{pmatrix} \quad \mbox{on} \ I'_{\theta_1}
\end{equation}
with $S^{(1,\pm)}_\psi \in \OPNerror{m}{2}$ on $I'_{\theta_1}$. Similarly we compute for the right hand side
\begin{multline}\label{ExpandAnsatz2}
\quad \bigg[ \begin{pmatrix} \partial_z \hspace{-2pt} - \hspace{-1pt} i B_{+} & 0 \\ 0 & \partial_z \hspace{-2pt} - \hspace{-1pt} i B_{-} \end{pmatrix}_{\substack{\! \psi \\ \phantom{} }} + R_\psi \bigg] Q_\psi \begin{pmatrix} 1 \\ \partial_z  \end{pmatrix}_{\substack{\! \psi \\ \phantom{} }} = \\ = \begin{pmatrix} \enspace (\partial_z \hspace{-1pt} - \hspace{-1pt} iB_{+})^{}_{\psi} (Q^{12}_\psi (\partial_z)^{}_{\psi} + Q^{11}_\psi) + S^{(2,+)}_\psi \enspace \\ (\partial_z \hspace{-2pt} - \hspace{-1pt} iB_{-})^{}_{\psi} (Q^{22}_\psi (\partial_z)^{}_{\psi} + Q^{21}_\psi) + S^{(2,-)}_\psi \end{pmatrix} \quad \mbox{on} \ I'_{\theta_1} \quad
\end{multline}
where $S^{(2,\pm)}_\psi \in \OPNerror{m}{2}$ on $I'_{\theta_1}$. Combining (\ref{ExpandAnsatz1}) and (\ref{ExpandAnsatz2}) we get an improved ansatz which reads
\renewcommand{\arraystretch}{1.6}
\begin{equation}\label{RefinedAnsatz}
\begin{aligned}
Q^{12}_\psi ( \partial^2_z + A )^{}_{\psi} &= (\partial_z - iB_{+})^{}_{\psi} \bigl(Q^{12}_\psi (\partial_z)^{}_{\psi} + Q^{11}_\psi \bigr) + R^{(+)}_\psi \quad \mbox{on} \ I'_{\theta_1} \\
Q^{22}_\psi ( \partial^2_z + A )^{}_{\psi} &= (\partial_z - iB_{-})^{}_{\psi} \bigl(Q^{22}_\psi (\partial_z)^{}_{\psi} + Q^{21}_\psi \bigr) + R^{(-)}_\psi \quad \mbox{on} \ I'_{\theta_1},
\end{aligned}
\end{equation}
with $R^{(\pm)}_\psi \in \OPNerror{m}{2}$ on $I'_{\theta_1}$.

We note that (\ref{RefinedAnsatz}) is a coupled system of two equations each of which stating a factorization similar to (\ref{eqn:Fact}) in Theorem~\ref{thm:Fact}. In the following we relate (\ref{eqn:Fact}) to (\ref{RefinedAnsatz}) which will guarantee existence of the ansatz made in (\ref{Ansatz}). Now since Theorem~\ref{thm:Fact} allows two different factorizations (depending on the sign of the principal symbol) we will modify both of them to deduce the refined ansatz (\ref{RefinedAnsatz}). 

Thus, with a view to Theorem~\ref{thm:Fact}, we first write $L_\psi= ( \partial_z^2  + A )_\psi$ in the following form
\begin{equation}\label{Fact1}
L_\psi = ( \partial_z + A_{11} )_\psi ( \partial_z + A_{12} )_\psi + \Gamma_{1,\psi} \quad \mbox{on} \ I'_{\theta_1}
\end{equation}
and for $j=\! 1,2$, $A_{1j,\psi} = A_{1j,\psi}(y,D_t,D_x)$ is a polyhomogeneous $\psi$-pseudodifferential operator with generalized symbol $(a_{1j,\varepsilon})_\varepsilon$ as in Theorem~\ref{thm:Fact}. Moreover we choose $A_{11,\psi}$ and $-A_{12,\psi}$ such that their principal symbols are equal to $(-i \sqrt{a_\varepsilon})_\varepsilon$. Furthermore $\Gamma_{1,\psi}$ is in $\OPNerror{0}{1}$ on $I'_{\theta_1}$. Likewise we obtain
\begin{equation}\label{Fact2}
L_\psi = ( \partial_z + A_{21} )_\psi ( \partial_z + A_{22} )_\psi + \Gamma_{2,\psi} \quad \mbox{on} \ I'_{\theta_1}
\end{equation}
where $A_{2j,\psi} = A_{2j,\psi}(y,D_t,D_x)$, $j=1,2$ are polyhomogeneous $\psi$-pseudodifferential operators but at this point the top order symbols of $A_{21,\psi}$ and $-A_{22,\psi}$ equal to $(i \sqrt{a_\varepsilon})_\varepsilon$. Again $\Gamma_{2,\psi}$ is in $\OPNerror{0}{1}$ on $I'_{\theta_1}$.

An expansion in (\ref{Fact1}) and (\ref{Fact2}) then gives the following for $j=1,2$
\begin{equation*}
(\partial^2_z + A)_\psi = (\partial^2_z)_\psi + (A_{j1}+A_{j2})_\psi(\partial_z)_\psi + {\mathop{\rm OP}}_\psi\bigl( (\varepsilon \partial_z a_{j2})_\varepsilon \bigr) + A_{j1,\psi}A_{j2,\psi} + \Gamma_{j,\psi}
\end{equation*}
on $I'_{\theta_1}$ where we have used the Leibniz rule. By construction of the generalized symbols of $A_{j1,\psi}$ and $A_{j2,\psi}$ we observe that $(a_{j1,\varepsilon})_\varepsilon = (-a_{j2,\varepsilon})_\varepsilon$ modulo $\mathcal{N}_{{\displaystyle\mathcal{S}}^1}$ on $I'_{\theta_1}$, $j=1,2$. Using this (\ref{Fact1}) and (\ref{Fact2}) then read
\begin{equation}\label{TwoFact}
\begin{split}
(\partial^2_z + A)_\psi 
&= (\partial_z + A_{11})_\psi (\partial_z - A_{11})_\psi \quad \mod{} \OPNerror{0}{1} \ \mbox{on} \ I'_{\theta_1} \\
&= (\partial_z + A_{21})_\psi (\partial_z - A_{21})_\psi \quad \mod{} \OPNerror{0}{1} \ \mbox{on} \ I'_{\theta_1}. 
\end{split}
\end{equation}
With $m$ being the fixed real number from the beginning of this section we now choose $\widetilde{Q}^{(\pm)}_\psi \in {\mathop{\rm OP}}_\psi \MphgSym[m,0]{1,0}$ and so that the requirements of Lemma~\ref{lem:psiparametrix} are fulfilled. Using the same result we obtain the existence of two $\psi$-pseudodifferential operators $Q^{(\pm)}_\psi$ with generalized symbols in $\MphgSym[-m,0]{1,0}$ satisfying (\ref{AE:parametrix}) and
\begin{equation*}
\widetilde{Q}^{(\pm)}_\psi Q^{(\pm)}_\psi = Q^{(\pm)}_\psi \widetilde{Q}^{(\pm)}_\psi = I_\psi \quad \mod{} {\mathop{\rm OP}}_\psi \mathcal{N}_{{\displaystyle\mathcal{S}}^0} \text{ on } I'_{\theta_1}.
\end{equation*}
Inserting $\widetilde{Q}^{(+)}_\psi Q^{(+)}_\psi$ into the first line of (\ref{TwoFact}) and $\widetilde{Q}^{(-)}_\psi Q^{(-)}_\psi$ into the second line yields
\begin{equation*}
\begin{aligned}
(\partial^2_z + A)_\psi 
&= (\partial_z + A_{11})^{}_{\psi} \widetilde{Q}^{(+)}_\psi \bigl(Q^{(+)}_\psi (\partial_z)^{}_{\psi} - Q^{(+)}_\psi A_{11,\psi}^{} \bigr) & \mod{} \OPNerror{0}{2} \text{ on } I'_{\theta_1}\ \\
&= (\partial_z + A_{21})^{}_{\psi} \widetilde{Q}^{(-)}_\psi \bigl(Q^{(-)}_\psi (\partial_z)^{}_{\psi} - Q^{(-)}_\psi A_{21,\psi}^{} \bigr) & \mod{} \OPNerror{0}{2} \text{ on } I'_{\theta_1}.
\end{aligned}
\end{equation*}
Here we define
\begin{equation}\label{SolutionsB}
\begin{aligned}
-iB_{+,\psi}^{}  & :=  Q^{(+)}_\psi A_{11,\psi}^{} \widetilde{Q}^{(+)}_\psi 
- {\mathop{\rm OP}}^{}_{\psi} \bigl( (\varepsilon \partial_z q_\varepsilon^{(+)} )^{}_{\varepsilon} \bigr) \widetilde{Q}^{(+)}_\psi \ \\
-iB_{-,\psi}^{}  & :=  Q^{(-)}_\psi A_{21,\psi}^{} \widetilde{Q}^{(-)}_\psi 
- {\mathop{\rm OP}}^{}_{\psi} \bigl( (\varepsilon \partial_z q_\varepsilon^{(-)} )^{}_{\varepsilon} \bigr) \widetilde{Q}^{(-)}_\psi.
\end{aligned}
\end{equation}
Consequently, $B_{\pm,\psi}$ are polyhomogeneous $\psi$-pseudodifferential operators with real-valued top order symbol in $\MhgSym[1,0]{1,0}$ on $I'_{\theta_1}$. Furthermore a straightforward calculation shows that with this choices for $B_{\pm,\psi}$ we have 
\begin{align*}
\bigl( \partial_z + A_{11} \bigr)_\psi \widetilde{Q}^{(+)}_\psi &= \widetilde{Q}^{(+)}_\psi \bigl( \partial_z - i B_{+} \bigr)_\psi \quad \mod{} \OPNerror{1-m}{0} \ \mbox{on} \ I'_{\theta_1}\\
\bigl( \partial_z + A_{21} \bigr)_\psi \widetilde{Q}^{(-)}_\psi &= \widetilde{Q}^{(-)}_\psi \bigl( \partial_z - i B_{-} \bigr)_\psi \quad \mod{} \OPNerror{1-m}{0} \ \mbox{on} \ I'_{\theta_1}.
\end{align*}
leading to the following decomposition on $I'_{\theta_1}$
\begin{align*}
\bigl( \partial^2_z  + A \bigr)^{}_{\psi}
 & = \widetilde{Q}^{(+)}_\psi \big( \partial_z - iB_{+} \big)^{}_{\psi} \big( Q^{(+)}_\psi  (\partial_z)^{}_\psi - Q^{(+)}_\psi A^{}_{11,\psi} \big) \quad \mod{} \OPNerror{0}{2}\phantom{.} \\
 & = \widetilde{Q}^{(-)}_\psi \big( \partial_z - iB_{-} \big)^{}_{\psi} \big( Q^{(-)}_\psi  (\partial_z)^{}_\psi - Q^{(-)}_\psi A^{}_{21,\psi} \big) \quad \mod{} \OPNerror{0}{2}.
\end{align*}
Comparing this with the refined ansatz (\ref{RefinedAnsatz}) we make the following choice for the matrix $Q_\psi$ from the beginning
\begin{align*}
   & Q^{11}_\psi := - Q^{(+)}_\psi A^{}_{11,\psi}, \hspace{-0.7cm} & Q^{12}_\psi &:= Q^{(+)}_\psi,& \\
   & Q^{21}_\psi := - Q^{(-)}_\psi A^{}_{21,\psi}, \hspace{-0.7cm} & Q^{22}_\psi &:= Q^{(-)}_\psi.&
\end{align*}
Hence
\begin{equation}\label{TransferMatrix}
Q_\psi = \begin{pmatrix} -Q^{(+)}_\psi A^{}_{11,\psi} & Q^{(+)}_\psi \\ -Q^{(-)}_\psi A^{}_{21,\psi} & Q^{(-)}_\psi \end{pmatrix} = \begin{pmatrix} Q^{(+)} & 0 \\ 0 & Q^{(-)} \end{pmatrix}_{\substack{\! \psi \\ \phantom{} }} \begin{pmatrix} -A_{11} & 1 \\ -A_{21} & 1 \end{pmatrix}_{\substack{\! \psi \\ \phantom{} }} \hspace{5pt}
\end{equation}
and $Q^{(\pm)}_\psi$ are polyhomogeneous $\psi$-pseudodifferential operators of order $-m$ on $I'_{\theta_1}$ and elliptic in the sense of (\ref{StrongEll}).

Since the top order symbol of $A_{j1,\psi}$ is given by $(\pm i \sqrt{a_\varepsilon})_\varepsilon$, $j=1,2$ it follows from (\ref{TransferMatrix}) that $Q_\psi$ is elliptic in the sense of (\ref{StrongEll}). Furthermore, the approximative inverse matrix $P_\psi$ of $Q_\psi$ is given by
\begin{equation*}
P_\psi = \begin{pmatrix} -C_\psi & C_\psi \\ -A_{21,\psi} C_\psi & 1_\psi + A_{21,\psi} C_\psi \end{pmatrix} \begin{pmatrix} \widetilde{Q}^{(+)} & 0 \\ 0 & \widetilde{Q}^{(-)} \end{pmatrix}_{\substack{\! \psi \\ \phantom{} }} \qquad \text{on } I'_{\theta_1}
\end{equation*}
where $C_\psi$ is the generalized parametrix of $(A_{11} - A_{21})_\psi$ in the sense of Lemma~\ref{lem:psiparametrix}. We have therefore found an appropriate operator-valued matrix $P_\psi,Q_\psi,R_\psi$ and operators $B_{\pm,\psi}$ solving (\ref{Ansatz}).
We have therefore proved the following theorem:
\begin{thm}
Let $L_\psi$ as in (\ref{GovOp}) and $U, F \in \mathcal{G}_{2,2}(\mathbb{R}^{n+1})$. Then there are operators $Q_\psi$ as in (\ref{TransferMatrix}) and $B_{\pm,\psi}$ as in (\ref{SolutionsB}) such that the equation
\begin{equation}\label{eqn:IMLGov}
L_\psi U = F \quad \mbox{microlocally at infinity on} \ I_{\theta_1}
\end{equation}
holds if and only if
\begin{equation}\label{eqn:IMLSystem}
\begin{aligned}
\bigl( \partial_z - iB_{+}(y,D_t,D_x) \bigr)_\psi u_{+} &= f_{+} \quad \mbox{microlocally at infinity on } I_{\theta_1} \mbox{ and}\\
\bigl( \partial_z - iB_{-}(y,D_t,D_x) \bigr)_\psi u_{-} &= f_{-} \quad \mbox{microlocally at infinity on } I_{\theta_1}.
\end{aligned}
\end{equation}
Furthermore the coupling effect is computed as follows
\begin{equation}\nonumber
\begin{pmatrix} u_{+} \\ u_{-} \end{pmatrix} := Q_\psi \begin{pmatrix} U \\ (\partial_z)_\psi U \end{pmatrix} , \hspace{25pt} \begin{pmatrix} f_{+} \\ f_{-} \end{pmatrix} := Q_\psi \begin{pmatrix} 0 \\ F \end{pmatrix}.
\end{equation}
\end{thm}
\subsection*{Closing remarks.} Let us briefly summarize the above. First we explained a factorization procedure for the semiclassical operator $L_\psi$ in the non-trivial case of generalized coefficients of log-type when acting on $\mathcal{G}_{2,2}$. To overcome the error made in this factorization we further introduced an adapted notion of microlocal regularity.

Concerning the motivating part of this paper we want to make the following remarks. Because of the lack of an adequate description of propagation of singularities in this setting it is not clear so far if and how one can derive approximated solutions to (\ref{eqn:IMLGov}) from solutions of a perturbation of the problem (\ref{eqn:IMLSystem}) as we have seen in the smooth case in Subsection~\ref{subsec:PreviousResults}. Even if we allow the operator $L_\varepsilon$ given in (\ref{eqn:Gov}) to have logarithmic slow scale regular coefficients this problem remains unsolved. Once again, we want to mention that in the case of logarithmic slow scale case no semiclassical interpretation of the situation is necessary and microlocal regularity is based on local $\mathcal{G}_{2,2}^\infty$-regularity. 

Moreover we note that apart from the microlocal restrictions in the equations of (\ref{eqn:IMLSystem}) the operators $L_{j,\psi}$ meet the conditions of Theorem 3.1 of \cite{Hoermann:04}, $j=1,2$, if the coefficients in (\ref{eqn:Gov}) from the beginning are of log-type with an appropriately chosen exponent $r \in \mathbb{N}$ which depends only on the dimension $n$. More precisely $r$ plays the same role as $k$ does in Remark 3.2 in \cite{Hoermann:04}. Therefore if it is possible to associate to (\ref{eqn:IMLSystem}) a global description of the same by only slight manipulations of the operators $L_{j,\psi} = \bigl( \partial_z - iB_{\pm}(y,D_t,D_x) \bigr)_\psi$, $j=1,2$ in a microlocal sense (as in the smooth setting) one can derive well-posedness to the corresponding Cauchy problems which in turn approximate (\ref{eqn:IMLGov}) on $I_{\theta_1}$. 

However, as already pointed out in Subsection~\ref{subsec:ScColombeau} we so far only handled the case where the semiclassical asymptotic regime was restricted to $\hbar(\varepsilon)=\varepsilon$ as $\varepsilon \to 0$. This suggests itself to ask for general criteria of other possible choices for the semiclassical scale $\hbar(\varepsilon)$. Under the viewpoint of parameter-dependent representation theory for generalized pseudodifferential operators this will hopefully also give more insight in the notion of microlocal regularity at infinity. Here a future aim is to obtain a refined notion of microlocalization which is capable to give a global characterization and hence has to include regularity results for Colombeau generalized objects also for finite points.
\vspace{10pt}

\textbf{Acknowlegdments.}
The author is very grateful to G\"{u}nther H\"{o}rmann for many helpful discussions. 
\bibliographystyle{amsplain}
%\bibliography{bibJabRef_DraftPaper}

\end{document}